\documentclass[11pt]{amsart}
\usepackage{amssymb,amsmath,amsthm}
\usepackage{amsfonts}
\usepackage{enumerate}
\usepackage{dsfont}
\usepackage{cite}
\usepackage[colorlinks, citecolor=red]{hyperref}
\usepackage{mathrsfs}
\usepackage{epsfig}
\usepackage{lscape}
\usepackage{subfigure}
\usepackage{epstopdf}
\usepackage{caption}
\usepackage{algorithm}
\usepackage{algpseudocode}
\usepackage{multirow}
\usepackage{geometry}
\usepackage{paralist}
\usepackage{enumerate}
\usepackage{url}
\usepackage{graphicx,cite}
\usepackage{longtable}

\newtheorem{definition}{Definition}[section]

\newtheorem{theorem}[definition]{Theorem}
\newtheorem{lemma}[definition]{Lemma}

\newtheorem{remark}[definition]{Remark}

\newtheorem{assumption}{Assumption}[section]
\date{}

\begin{document}
\baselineskip 18pt
\bibliographystyle{plain}
\title[Pseudoinverse-free randomized methods]
{On pseudoinverse-free randomized methods for linear systems: Unified framework and acceleration}
%{Randomized iterative methods without pseudoinverse for linear systems: Unified framework and acceleration}

\author{Deren Han}
\address{LMIB of the Ministry of Education, School of Mathematical Sciences, Beihang University, Beijing, 100191, China. }
\email{handr@buaa.edu.cn}

\author{Jiaxin Xie}
\address{LMIB of the Ministry of Education, School of Mathematical Sciences, Beihang University, Beijing, 100191, China. }
\email{xiejx@buaa.edu.cn}

\begin{abstract}
We present a new framework for the analysis and design of randomized algorithms for solving  various types of linear systems, including  consistent or inconsistent, full rank or rank-deficient.
Our method is formulated with \emph{four randomized sampling parameters}, which allows the method to cover many existing randomization algorithms within a unified framework, including the doubly stochastic Gauss-Seidel, randomized Kaczmarz method, randomized coordinate descent method, and Gaussian Kaczmarz method.
 Compared with the projection-based block algorithms where a pseudoinverse for solving a least-squares problem is utilized at each iteration, our design is  \emph{pseudoinverse-free}. Furthermore,
the flexibility of the new approach also enables the design of a number of new methods as special cases.
Polyak's \emph{ heavy ball momentum}  technique is also introduced in our framework for improving  the convergence behavior of the method.
 We prove the global linear convergence rates of our method
as well as an accelerated linear rate for the case of the norm of expected iterates.
Finally, numerical experiments are  provided to confirm our results.
\end{abstract}

\maketitle

\let\thefootnote\relax

\let\thefootnote\relax

\let\thefootnote\relax\footnotetext{Corresponding author: Jiaxin Xie, LMIB of the Ministry of Education, School of Mathematical Sciences, Beihang University, Beijing, 100191, China. E-mail: xiejx@buaa.edu.cn}
% ---------------------- section introduction ----------------------
\section{Introduction}

Randomized iterative methods such as the randomized Kaczmarz method have been very popular recently for solving large-scale linear systems.
They are preferred mainly because of their low iteration cost and
low memory storage so that they scale better with the size of the problems.
Randomized methods are now  playing a major role in areas like numerical linear algebra, scientific computing, and optimization.
 Their applications include tensor recovery \cite{chen2021regularized}, image reconstruction \cite{herman1993algebraic}, signal processing \cite{Byr04}, optimal control \cite{Pat17},  partial differential equations \cite{Ols14}, and machine learning \cite{Cha08}.

Let us consider the following linear system
\begin{equation}
\label{main-prob}
Ax=b,
\end{equation}
where $A\in\mathbb{R}^{m\times n}$ and $b\in\mathbb{R}^m$. The \emph{Kaczmarz} method \cite{Kac37}, also known as \emph{algebraic reconstruction technique} (ART) \cite{herman1993algebraic,gordon1970algebraic}, is one of the popular methods for solving \eqref{main-prob}.
At each step, the method cyclically projects the current iterate onto the solution space of a single constraint and  converges to a certain solution of the consistent linear systems. However, the rate of convergence is hard to obtain. In a seminal paper \cite{Str09}, Strohmer and Vershynin studied the \emph{randomized  Kaczmarz} (RK) method and proved its linear convergence  for solving consistent linear systems.
Their elegant result  has inspired a lot of subsequent work. We refer the interested reader to
\cite{bai2021greedy,HSXDR2022,Bai18Con,Bai18Gre,Du20Ran,Du20Kac,DS2021,Gow15,Gow19,Hef15,Jia17,Lev10,Nut16,Sha21,Ste20Ran,Ste20Wei,Str09,yuan2022adaptively}
and the references therein.

In \cite{Gow15}, Gower and Richt\'{a}rik developed a versatile randomized iterative method for solving consistent linear systems  which includes the RK algorithm as a special case.
%Other special cases include randomized coordinate descent, randomized Gaussian descent and randomized Newton method. Block versions and versions with importance sampling of all these methods also arise as special cases.
%a unified randomized iteration scheme on both regular and block Kaczmarz methods, which is also referred to as the \emph{sketch-and-project} method.
 The  method updates with the following iterative strategy:
\begin{equation}\label{sketch-project}
x^{k+1}=x^k-\alpha B^{-1}A^\top S(S^\top AB^{-1}A^\top S)^\dagger S^\top(Ax^k-b),
\end{equation}
where $S\in\mathbb{R}^{m\times q}$ are random matrices drawn from a distribution $\mathcal{D}$ at each step,  $B\in\mathbb{R}^{n\times n}$ is a user-given symmetric positive matrix, $\alpha\in(0,2)$ is a stepsize parameter, and the symbols $\dagger$ and $\top$ denote the Moore-Penrose pseudoinverse and the transpose of either a vector or a matrix, respectively. By varying these two parameters $S$ and $B$, one can recover a number of popular methods as special cases, including the  RK method, the Gaussian Kaczmarz method, the randomized Newton method, random Gaussian pursuit, and variants of all these methods using blocks and importance sampling. Their method is the first algorithm uncovering the close relations between
%``sibling'' relationship between
 these methods and  is also known  as the \emph{sketch-and-project} method.
 The method is shown to enjoy linear convergence from which existing
convergence results can also be obtained.
%They proved a general linear convergence result for the method from which existing convergence results can  be obtained.

However, the main disadvantage of \eqref{sketch-project} is that each iteration is expensive since we need to apply the pseudoinverse to a vector, or equivalently, we have to solve a least-squares problem at each iteration. Moreover, it is also difficult to parallelize.
In this paper, we aim to develop randomized algorithms with accelerated  convergence rates that do not employ the pseudoinverse.
Indeed, we  will utilize Polyak's heavy ball momentum \cite{polyak1964some} technique to accelerate the convergence rate of the method.

%Our motivation can be explained as follows. Let us consider the  \emph{doubly stochastic Gauss-Seidel} (DSGS) method proposed by  Razaviyayn \emph{et al} \cite{RHRL2019}.
%We use $a_1^\top,\ldots,a_m^\top$ to denote the rows of $A$ and  $a_{i,j}$ to denote the $(i,j)$-th element of $A$. At each step, the DSGS method updates with the following iterative strategy
%\begin{equation}\label{DSGSiter}
%x^{k+1}=x^k-\alpha \frac{\langle a_i,x^k\rangle-b_i}{a_{i,j}}e_j,
%\end{equation}
%where the index pair $(i,j)$  is selected with probability
%$p_{i,j}=\frac{a_{i,j}^2}{\|A\|_F^2}$.
%Here $\|A\|_F$  denotes the Frobenius norm of the matrix $A$ and $\alpha>0$ is the stepsize.
%The DSGS method can be viewed as the \emph{row-and-column sampling} method.
%Indeed, let us take $S=\frac{e_i}{|a_{i,j}|}\in\mathbb{R}^m$ and $T=e_j\in\mathbb{R}^n$, and let the index pair $(i,j)$ be randomly selected with probability $\frac{a_{i,j}^2}{\|A\|^2_F}$. Then the iteration \eqref{DSGSiter} can be alternatively written as
%\begin{equation}\label{dsgsiter}
%x^{k+1}=x^k-\alpha TT^\top A^\top SS^\top(Ax^k-b),
%\end{equation}
%where $T$ is for row sampling  and $S$ is for column sampling.
%A natural and interesting question is that \emph{Can we  equip \eqref{dsgsiter} with general sampling parameters $(S,T)$ to design  new types  of randomized algorithms?}
%In addition, we  will also utilize Polyak's heavy ball momentum \cite{polyak1964some} technique to accelerate the convergence rate of the method.

{\bf The heavy ball method.}
Consider the following optimization problem:
$$
\min\limits_{x\in\mathbb{R}^n} f(x),
$$
where $f$ is a differentiable convex function.
Gradient descent (GD) is  widely used method for
solving the above problem.
Starting with an arbitrary point $x^0$,  the iteration scheme of the original
GD method reads as
$$
x^{k+1}=x^{k}-\gamma_{k} \nabla f\big(x^{k}\big),
$$
where $\gamma_{k}>0$ is a stepsize and $\nabla f\left(x^{k}\right)$ denotes the gradient of $f$ at $x^k$. For convex functions with $L$-Lipschitz gradient, GD needs $O(L / \varepsilon)$ steps to guarantee the error
of the solution is within $\varepsilon$. When, in addition, $f$ is $\mu$-strongly convex,
it converges linearly with convergence rate of $O(\log (\varepsilon^{-1})(L / \mu))$ \cite{nesterov2003introductory}.
%To improve the convergence behavior of the method,
To improve the rate of convergence, Polyak proposed to modify GD by the introduction of a (heavy ball) momentum term, $\omega\left(x^{k}-x^{k-1}\right)$. This leads to the gradient descent method with momentum (mGD), popularly known as the heavy ball method:
$$
x^{k+1}=x^{k}-\gamma_{k} \nabla f\big(x^{k}\big)+\omega\big(x^{k}-x^{k-1}\big) .
$$
Polyak \cite{polyak1964some} proved that for twice continuously differentiable objective function $f(x)$ with $\mu$ strong convexity constant and $L$-Lipschitz gradient, mGD achieves a local accelerated linear convergence rate of $O\left(\log (\varepsilon^{-1})\sqrt{L / \mu}\right)$ (with the appropriate choice of the stepsize parameters $\gamma_{k}$ and momentum parameter $\omega$).
%Polyak \cite{polyak1964some} proved that with the appropriate choice of the stepsize parameters $\gamma_{k}$ and momentum parameter $\omega$, a local accelerated linear convergence rate of $O\left(\log (\varepsilon^{-1})\sqrt{L / \mu}\right)$ can be achieved, in the case of twice continuously differentiable, $\mu$-strongly convex objective functions with $L$-Lipschitz gradient.

\subsection{Our contribution}

In this paper, we propose the following more general randomized algorithms framework for solving the systems of linear equations
\begin{equation}\label{pdfriter}
x^{k+1}=x^k-\alpha T_1T_2^\top A^{\top}S_1S_2^\top(Ax^k-b)+\omega(x^k-x^{k-1}),
\end{equation}
where $\alpha>0$ is a stepsize parameter, $\omega\geq0$ is a momentum parameter , and $(S_1,S_2,T_1,T_2)$ is a random parameter matrix pair which is sampled independently in each iteration from a certain distribution $\mathcal{D}$.
%The main contributions of this work are:
The main contributions of this work are as follows.
\begin{itemize}
  \item[1.] We propose a novel randomized algorithmic framework for solving different types of  linear systems,  including consistent or inconsistent, full rank or rank-deficient.
  We show the method enjoys a global linear convergence rate. % for different types of linear systems.
   We also study the convergence of the quantity $\|\mathbb{E}[x^k-x^0_*]\|_2^2$. In this case, we show that by a  proper combination of the stepsize
parameter $\alpha$ and the momentum parameter $\omega$, the proposed method enjoys an accelerated convergence rate.
  \item[2.] The convergence of the algorithm in some special cases is also studied. Particularly, global linear convergence rates are established under various formulations and assumptions about the sampling matrices and the linear systems.
  \item[3.] We recover several known methods as a special case of our general framework which enables us to draw previously unknown links between those methods. The flexibility of the method also allows us to adjust the parameter matrices  to obtain completely new methods. Particularly, we study the methods based on Gaussian sampling.
\end{itemize}

\subsection{Related work}

This result is related to the recent work of Du and Sun in \cite{DS2021}, where two novel pseudoinverse-free randomized methods for solving a linear system are proposed. In particular, they studied the framework in \eqref{pdfriter} with $S_1=S_2$, $T_1=T_2$, and $\omega=0$. When $A$ is a full column rank matrix, they \cite[Theorem 4]{DS2021} showed that the method converges linearly in the mean square sense.
We  emphasize that our convergence result is more general and works for all types of matrices (consistent or inconsistent, full rank or rank-deficient); see Theorems \ref{THMr-in} and \ref{THMr-in629}. In addition,  we use Polyak's heavy ball momentum for accelerating the convergence of the method in our framework.

There is a variant of the \emph{randomized block Kaczmarz} (RBK) that avoids the issues of pseudoinverse,
we emphasize the results of Necoara \cite{Nec19}.
 This leads to the following iteration:
\begin{equation}\label{BKM}
x^{k+1}=x^k-\alpha\bigg(\sum\limits_{i\in J}\omega_i\frac{\langle a_i,x^k\rangle-b_i}{\|a_i\|^2_2}a_i\bigg),
\end{equation}
where the weights $\omega_i\in[0,1]$ such that $\sum\limits_{i\in J}\omega_i=1$, $J$  is drawn from a certain distribution $\mathcal{D}$ at each step,  $\alpha>0$ is the stepsize parameter, and $a_1^\top,\ldots,a_m^\top$ denote the rows of $A$. Let $I_{:J}$
denote a column concatenation of the columns of the $m\times m$ identity matrix $I$ indexed by $J$, and the diagonal matrix  $D_J:=\mbox{diag}(\sqrt{\omega_i}/\|a_i\|_2,i\in J)$.
Then the iteration \eqref{BKM} can alternatively be written as
\begin{equation}\label{BKMs}
x^{k+1}=x^k-\alpha A^\top S S^\top (Ax^k-b),
\end{equation}
where $S = I_{:J}D_J$, which can be viewed as randomized matrices selected from the distribution $\mathcal{D}$ at each step. We note that the iteration scheme \eqref{BKMs} can also be recovered by our framework in the same  special cases; see Section \ref{sbst:mRBK} for more details. Indeed, it is also an interesting topic that one can equip \eqref{BKMs} with more general sampling matrices $S$  for obtaining  new classes of randomized algorithms.

Recently, Loizou and Richt{\'a}rik \cite{loizou2020momentum} studied several kinds of stochastic optimization algorithms enriched by Polyak's heavy ball momentum for solving convex quadratic problems. They proved the global, non-asymptotic linear convergence rates of the proposed methods. The momentum variants of the sketch-and-project method has been investigated in \cite{P2017Stochastic,loizou2021revisiting}.
In \cite{morshed2020stochastic}, Polyak's momentum technique was incorporated  into the stochastic steepest descent methods for solving linear systems. Their  momentum framework integrated several momentum algorithms, such as the Kaczmarz method, the block Kaczmarz method, and the coordinate descent,  into one framework.
 In \cite{HSXDR2022}, the authors studied the Douglas-Rachford (DR) method \cite{Lin20,Ara14,Li16,Eck92,aragon2019cyclic,censor2016new,han2022survey} enriched with randomization and Polyak's heavy
ball momentum for solving linear systems. They showed the power of randomization in simplifying the analysis of the DR method and made the divergent $r$-sets-DR method converge linearly.

We note that a more widely used, and better understood alternative to Polyak's momentum is the momentum introduced by Nesterov \cite{nesterov1983method,nesterov2003introductory}, leading to the famous accelerated gradient descent (AGD) method \cite{beck2009fast}. Recently, variants of Nesterov's momentum have also been introduced for the acceleration of stochastic optimization algorithms \cite{lan2020first}.  In
\cite{liu2016accelerated}, Liu and Wright applied the acceleration scheme of Nesterov to the randomized Kaczmarz method. Very recently,  Nesterov's acceleration scheme has been applied to the sampling Kaczmarz Motzkin (SKM) algorithm for linear feasibility problems \cite{morshed2020accelerated,morshed2021sampling}.

\subsection{Organization}
 The remainder of the paper is organized as follows.
In Section 2, we give some notations and a useful lemma.
In Section 3, we present the pseudoinverse-free randomized (PFR) method and analysis its convergence properties.
In Section 4, we propose the PFR with momentum (mPFR) and show its accelerated
linear convergence rate. We also study the convergence of the algorithm in some special cases.
In Section 5, we mention how by selecting the parameters $(T_1,T_2,S_1,S_2)$ of our method, we recover several existing methods as well as obtain new methods.
In Section 6, we perform some numerical experiments to show
the effectiveness of the proposed methods. Finally, we conclude the paper in Section 7.
Proofs of all main results can be found in the ``Appendix''.

\section{Notation and preliminaries}
Throughout the paper, for any random variables $\xi$ and $\zeta$, we use $\mathbb{E}[\xi]$ and $\mathbb{E}[\xi\lvert \zeta]$ (or $\mathbb{E}[\xi \lvert \zeta_0]$) to denote the expectation of $\xi$ and the conditional expectation of $\xi$ given $\zeta$.
For an integer $m\geq 1$, let $[m]:=\{1,\ldots,m\}$. For any vector $x\in\mathbb{R}^n$, we use $x_i,x^\top$ and $\|x\|_2$ to denote the $i$-th entry, the transpose and the Euclidean norm of $x$, respectively. We use $e_i$ to denote the column vector with a $1$ at the $i$-th position and zeros elsewhere.
For any matrix $A$, we use $a_i, A_j,a_{i,j},A^\top,A^\dagger,\|A\|_2,\|A\|_F$ and $\mbox{Range}(A)$ to denote the $i$-th row, the $j$-th column, the $(i,j)$-th entry, the transpose, the Moore-Penrose pseudoinverse, the spectral norm, the Frobenius norm and the column space of $A$, respectively. Let $A\in\mathbb{R}^{m\times n}$, we use $A=U\Sigma V^\top$ to denote the singular value decomposition (SVD) of $A$, where $U\in\mathbb{R}^{m\times m}$, $\Sigma\in\mathbb{R}^{m\times n}$, and $V\in\mathbb{R}^{n\times n}$.
The nonzero singular values of a matrix $A$ are $\sigma_1(A)\geq\sigma_2(A)\geq\ldots\geq\sigma_{r}(A):=\sigma_{\min}(A)>0$, where $r$ is the rank of $A$
and we use $\sigma_{\min}(A)$ to denote  the smallest nonzero singular values of $A$.  We see that $\|A\|_2=\sigma_{1}(A)$ and $\|A\|_F=\sqrt{\sum\limits_{i=1}^r \sigma_i(A)^2}$. For a square matrix $B\in\mathbb{R}^{n\times n}$, we use $\mbox{Tr}(B):=\sum\limits_{i=1}^n b_{ii}$ to denote its trace.
For index sets $\mathcal{R}\subseteq[m]$ and $\mathcal{L}\subseteq [n]$, we use $A_{\mathcal{R}:}$ and $A_{:\mathcal{L}}$ to denote the row submatrix indexed by
$\mathcal{R}$ and the column submatrix indexed by $\mathcal{L}$, respectively. We use $|\mathcal{R}|$ to denote the cardinality of a set $\mathcal{R}\subseteq[m]$.

Throughout this paper, we use $x^*$ to denote a certain solution of the linear system \eqref{main-prob}, and for any $x^0\in\mathbb{R}^n$, we set $$x_0^*:=A^\dagger b+(I-A^\dagger A)x^0 $$
and
$$ x_{LS}^*:=A^\dagger b.$$
We know that $x_0^*$
is the orthogonal projection of $x^0$ onto the set
$$
\{x\in\mathbb{R}^n|A^\top A x=A^\top b\},
$$
and $x_{LS}^*$ is the least-squares or the least-norm least-squares solution of the linear system.
In this paper, by ``momentum'' we refer to the heavy ball technique originally developed by Polyak \cite{polyak1964some} to accelerate the convergence rate of gradient-type methods.

Throughout this paper, we will always  assume that the sampling matrices satisfy the following assumption.
\begin{assumption}\label{assumption1}
The random parameter matrix pair $(S_1,S_2,T_1,T_2)$  is sampled independently in each iteration from a distribution $\mathcal{D}$ and satisfies
$$\mathbb{E}[T_1T_2^\top A^\top S_1S_2^\top]=\frac{A^\top}{\|A\|^2_F}.$$
\end{assumption}
We note that in Assumption \ref{assumption1}, the random parameter matrices $S_1,S_2\in\mathbb{R}^{m\times p}$ and $T_1,T_2\in\mathbb{R}^{n\times s}$ can be independent or not.
 We also emphasize that we do not restrict the numbers of columns of $S_1$, $S_2,T_1$ and $T_2$; indeed, we allow $p$ and $s$ to vary (and hence $p$ and $s$ are random variables).

We will utilize the following result to estimate the upper bound of the stepsize parameter for the Gaussian-type randomized algorithms.

\begin{lemma}
\label{lemma-gaussmatrix}
Suppose that $S\in\mathbb{R}^{m\times p}$ is a Gaussian matrix, i.e. has independent $\mathcal{N}(0,1)$ entries.
Then for any fixed $A\in\mathbb{R}^{m\times n}$,
$$
\mathbb{E}[ S S^\top AA^\top S S^\top ]=(p^2+p) AA^\top +p\|A\|^2_F I.
$$
\end{lemma}

\section{The pseudoinverse-free randomized method for linear systems }
\label{section-RrDR}

Let us first study the \emph{pseudoinverse-free randomized} (PFR) method for solving linear systems.
Given an arbitrary initial guess $x^0\in\mathbb{R}^n$, the $k$-th iterate of the  PFR method is defined as
\begin{equation}\label{dsbi1}
x^{k+1}=x^{k}-\alpha T_1T_2^\top A^\top S_1S_2^\top(Ax^{k}-b),
\end{equation}
where the stepsize parameter $\alpha>0$. Note that \eqref{dsbi1} is equivalent with the no-momentum variant of \eqref{pdfriter}, i.e. $\omega=0$, and
under Assumption \ref{assumption1}, we have
$$
\mathbb{E}[x^k|x^{k-1}]=x^{k-1}-\frac{\alpha}{\|A\|^2_F} A^\top(Ax^{k-1}-b),
$$
which is the update of the Landweber iteration \cite{landweber1951iteration}. The PFR
method is formally described in Algorithm \ref{PFR1}.

\begin{algorithm}[htpb]
\caption{The pseudoinverse-free randomized (PFR) algorithm \label{PFR1}}
\begin{algorithmic}
\Require
Distribution $\mathcal{D}$ from which to sample matrices, $A\in \mathbb{R}^{m\times n}$, $b\in \mathbb{R}^m$,  stepsize/relaxation parameter $\alpha>0$, $k=0$, and an initial $x^0\in \mathbb{R}^{n}$.
\begin{enumerate}
\item[1:] Randomly select a matrix pair $(S_1,S_2,T_1,T_2)$ from the distribution $\mathcal{D}$.
\item[2:] Update
$$x^{k+1}=x^{k}-\alpha T_1T_2^\top A^\top S_1S_2^\top(Ax^{k}-b).$$
\item[3:] If the stopping rule is satisfied, stop and go to output. Otherwise, set $k=k+1$ and return to Step $1$.
\end{enumerate}

\Ensure
  The approximate solution.
\end{algorithmic}
\end{algorithm}

\subsection{Convergence of iterates: Linear rate}

In this subsection, we analyze some convergence properties of Algorithm \ref{PFR1}. We emphasize that the matrix pair $(S_1,S_2,T_1,T_2)$ selected from the distribution $\mathcal{D}$ satisfies the Assumption \ref{assumption1}.
The following theorem shows that Algorithm \ref{PFR1} converges linearly in expectation for different types of linear systems (consistent or inconsistent, full rank or rank-deficient).

\begin{theorem}\label{THMr-in}
Let $x_{LS}^*$ be the least-squares or the least-norm least-squares solution of the linear system $Ax=b$ and $\{x^k\}_{k=0}^\infty$ be the sequence generated by Algorithm \ref{PFR1}. Denote by $r^*=Ax_{LS}^*-b$ and $r^k=Ax^k-b$.
For arbitrary $x^0\in\mathbb{R}^n$, if $0<\alpha<\frac{\sigma^2_{\min}(A)}{\beta\|A\|^2_F}$, then the
$k$-th iterative residuals $r^k=Ax^k-b$ of Algorithm \ref{PFR1} satisfies
$$
\mathbb{E}\big[\|r^k-r^*\|^2_2\big]\leq\eta^k \|r^0-r^*\|^2_2+\frac{2\alpha^2\beta(1-\eta^k)\|r^*\|^2_2}{1-\eta},
$$
where
\begin{equation}\label{eta-beta}
\eta=1-2\alpha\frac{\sigma_{\min}^2(A)}{\|A\|^2_F}+2\alpha^2\beta \ \ \mbox{and} \ \  \beta=\big\|\mathbb{E}[ S_2S_1^\top AT_2T_1^\top A^\top A T_1T_2^\top A^\top S_1S_2^\top ]\big\|_2.
\end{equation}
\end{theorem}

Next, we consider the convergence of $x^k-x^0_*$. Particularly, we will consider the convergence of $\left\|\mathbb{E}\left[x^k-x^0_*\right]\right\|^2_2$.
\begin{theorem}\label{THMnm2}
Let $x^0\in\mathbb{R}^n$ be an arbitrary initial vector and $x^0_*=A^{\dagger}b+(I-A^\dagger A)x^0$.
The stepsize parameter $0<\alpha\leq \frac{\|A\|^2_F}{\sigma_{\max}^2(A)}$.
  Then the iteration sequence of $\{x^k\}_{k=1}^{\infty}$ in Algorithm \ref{PFR1} satisfies
  $$
  \left\|\mathbb{E}\left[x^k-x^0_*\right]\right\|^2_2\leq\left(1-\alpha\frac{\sigma_{\min}^2(A)}{\|A\|^2_F}
\right)^{2k}\left\|x^0-x^0_*\right\|^2_2.
  $$
\end{theorem}

\subsection{Convergence direction}

Inspired by the recent work of Steinerberger \cite{Ste20Ran},
this section aims to consider the convergence direction of Algorithm \ref{PFR1}.
The following result shows  different convergence rates of Algorithm \ref{PFR1} along different singular vectors
of $A$.

\begin{theorem}
\label{THMmain}
 Let $x^0\in\mathbb{R}^n$ be an arbitrary initial vector and $x^0_*=A^{\dagger}b+(I-A^\dagger A)x^0$.
 Suppose that $v_\ell$ is a (right) singular vector of $A$ associated to the singular
value $\sigma_{\ell}$. Then the iteration sequence $\{x_k\}_{k=1}^{\infty}$ in Algorithm \ref{PFR1} satisfies
$$
\mathbb{E}[\langle x^k-x^0_*,v_{\ell}\rangle]=\bigg(1-\alpha\frac{\sigma_{\ell}^2}{\|A\|^2_F}\bigg)^k\langle x^0-x^0_*,v_{\ell}\rangle.
$$
\end{theorem}

This exhibits that Algorithm \ref{PFR1} decays exponentially at different rates depending on the singular values and accounts for the typical \emph{semiconvergence} phenomenon.
That is, the residual $ \|Ax^{k}-b\|_2^2 $ decays faster at the beginning, but then gradually stagnates.
Recently, such semiconvergence phenomenon has been studied in the literature for randomized methods.
In \cite{Ste20Ran}, Steinerberger studied semiconvergence phenomenon for the randomized Kacazmarz method.
In \cite{Wu21}, Wu and Xiang exploited  the semiconvergence phenomenon for the sketch-and-project method \cite{Gow15},
where they generalized the study in \cite{Jia17} and split the total error into the low- and high-frequency solution spaces. Very recently,  the semiconvergence phenomenon has also been studied for randomized DR method in \cite{HSXDR2022}.
%
%In addition, the semiconvergence phenomenon has also been exploited by Wu and Xiang \cite{Wu21} for the randomized row iterative method \cite{Gow15},  where they generalized the study in \cite{Jia17} and split the total error into the low- and high-frequency solution spaces. In \cite{HSXDR2022}, the semiconvergence phenomenon has been studied for the randomized DR method.

\section{Momentum acceleration}

In this section, we provide the momentum induced PFR method, i.e. the mPFR method for solving a linear system.
The method is formally described in Algorithm \ref{mPDFR}.

\begin{algorithm}[htpb]
\caption{The pseudoinverse-free randomized algorithm with momentum (mPFR) \label{mPDFR}}
\begin{algorithmic}
\Require
Distribution $\mathcal{D}$ from which to sample matrices, $A\in \mathbb{R}^{m\times n}$, $b\in \mathbb{R}^m$,  stepsize/relaxation parameter $\alpha>0$, the heavy ball/momentum parameter $\omega$, $k=1$, and initial vectors $x^1,x^0\in \mathbb{R}^{n}$.
\begin{enumerate}
\item[1:] Randomly select a matrix pair $(S_1,S_2,T_1,T_2)$ from the distribution $\mathcal{D}$.
\item[2:] Update
$$x^{k+1}=x^{k}-\alpha T_1T_2^\top A^\top S_1S_2^\top(Ax^{k}-b)+\omega(x^k-x^{k-1}).$$
\item[3:] If the stopping rule is satisfied, stop and go to output. Otherwise, set $k=k+1$ and return to Step $1$.
\end{enumerate}

\Ensure
  The approximate solution.
\end{algorithmic}
\end{algorithm}

Under Assumption \ref{assumption1}, by step $2$ in Algorithm \ref{mPDFR} we have
$$
\mathbb{E}[x^k|x^{k-1}]=x^{k-1}-\frac{\alpha}{\|A\|^2_F} A^\top(Ax^{k-1}-b)+\omega(x^k-x^{k-1}),
$$
which can be viewed as the momentum variant of the Landweber iteration.
In the rest of this section, we will study the convergence properties of Algorithm \ref{mPDFR}.  We note that the matrix pair $(S_1,S_2,T_1,T_2)$ selected from the distribution $\mathcal{D}$ satisfies Assumption \ref{assumption1}.

\subsection{Convergence of iterates: Linear rate}
In this subsection, we study the convergence rate of the quantity $\mathbb{E}[\|r^k-r^*\|^2_2]$ for Algorithm \ref{mPDFR}. We show that the method enjoys a global linear convergence rate  for various types of linear systems  (consistent or inconsistent, full rank or rank-deficient).

\begin{theorem}\label{THMr-in629}
Let $x_{LS}^*$ be the least-squares or the least-norm least-squares solution of the linear system $Ax=b$ and the initial vectors $x^1=x^0\in\mathbb{R}^n$.  Denote by $r^*=Ax_{LS}^*-b$, $r^k=Ax^k-b$, and
$$\beta=\big\|\mathbb{E}[ S_2S_1^\top AT_2T_1^\top A^\top A T_1T_2^\top A^\top S_1S_2^\top ]\big\|_2.$$
Assume $0<\alpha<\frac{\sigma^2_{\min}(A)}{\beta\|A\|^2_F}$, $\omega\geq0$ and
that the expressions
$$
\gamma_1=1+3\omega+2\omega^2-(2\alpha+\alpha\omega)\frac{\sigma^2_{\min}(A)}{\|A\|^2_F}+2\alpha^2\beta
\
\mbox{and}
\
\gamma_2=2\omega^2+\omega+\omega\alpha\frac{\sigma^2_{\max}(A)}{\|A\|^2_F}
$$
satisfy $\gamma_1+\gamma_2<1$.
Then the iteration sequence of residuals $\{r^k\}_{k=1}^{\infty}$ generated by Algorithm \ref{mPDFR} satisfies
$$
\mathbb{E}[\|r^{k+1}-r^*\|^2_2]\leq q^k(1+\tau)\|r^0-r^*\|^2_2+\frac{2\alpha^2\beta(1-q^k)\|r^*\|^2_2}{1-q}, \ k\geq 0,
$$
where $$q=\left\{
           \begin{array}{ll}
             \frac{\gamma_1+\sqrt{\gamma_1^2+4\gamma_2}}{2}, & \hbox{if $\omega>0$;} \\
             \gamma_1, & \hbox{if $\omega=0$,}
           \end{array}
         \right.
\ \mbox{and} \ \tau=q-\gamma_1\geq 0.$$ Moreover,
$\gamma_1+\gamma_2\leq q<1$.
\end{theorem}

Let us explain how to  choose the parameters $\alpha$ and $\omega$ such that $\gamma_1+\gamma_2<1$ is satisfied in Theorem \ref{THMr-in629}. Indeed, let
$$
\tau_1:=4+\alpha\frac{\sigma_{\max}^2(A)}{\|A\|^2_F}-\alpha\frac{\sigma_{\min}^2(A)}{\|A\|^2_F}
$$
and
$$
\tau_2:=-2\alpha^2\beta+2\alpha\frac{\sigma_{\min}^2(A)}{\|A\|^2_F}.
$$
If we choose $0<\alpha<\frac{\sigma^2_{\min}(A)}{\beta\|A\|^2_F}$, then $\tau_2>0$ and the condition $\gamma_1+\gamma_2<1$ now is satisfied for all
$$
0\leq \omega<\frac{1}{8}\bigg(\sqrt{\tau_1^2+16\tau_2}-\tau_1\bigg).
$$
Next, we compare the convergence rates  obtained in Theorems \ref{THMnm2} and \ref{THMr-in629}.
From the definition of $\gamma_1$ and $\gamma_2$, we know that convergence rate $q(\omega)$  in Theorem \ref{THMr-in629} can be viewed as a function of $\omega$. Note that since $\omega\geq0$, we have
$$
\begin{array}{ll}
q(\omega)&\geq\gamma_1+\gamma_2\\
&=4\omega^2+\tau_1\omega-\tau_2+1
\\
&\geq 1-\tau_2
\\
&=1+2\alpha^2\beta-2\alpha\frac{\sigma_{\min}^2(A)}{\|A\|^2_F}=q(0).
\end{array}
$$
Clearly, the lower bound on $q$ is an increasing function of $\omega$. Also, for any $\omega$ the rate is
always inferior to that of Algorithm \ref{PFR1} in Theorem \ref{THMr-in}.

\subsection{Accelerated linear rate for expected iterates}
In this subsection, we study the convergence of the quantity $\left\|\mathbb{E}\left[x^k-x_*^0\right]\right\|_2^2$ to zero for Algorithm \ref{mPDFR}. We show that by a  proper combination of the relaxation
parameter $\alpha$ and the momentum parameter $\omega$, Algorithm \ref{mPDFR} enjoys an accelerated linear convergence rate in mean.

\begin{theorem}\label{THMfm2}
Suppose that $x^1=x^0\in\mathbb{R}^n$ are arbitrary initial vectors and $x^0_*=A^{\dagger}b+(I-A^\dagger A)x^0$.
 Let $\{x^k\}_{k=1}^{\infty}$ be the iteration sequence in Algorithm \ref{mPDFR}. Assume that the stepsize parameter
$0<\alpha\leq\frac{\|A\|^2_F}{\sigma_{\max}^2(A)}$
 and the momentum parameter
$$\left(1-\sqrt{\alpha\sigma_{\min}^2(A)/\|A\|^2_F}\right)^2<\omega<1.$$
Then there exists a constant
$c>0$ such that for all $k\geq 0$ we have
$$
\left\|\mathbb{E}\left[x^k-x_*^0\right]\right\|_2^2\leq \omega^k c.
$$
\end{theorem}

\begin{remark}
	Note that the convergence factor in Theorem  \ref{THMfm2} is precisely equal to the value of the momentum parameter $\omega$. Theorem \ref{THMnm2} shows that Algorithm \ref{PFR1} (without momentum)
	converges with iteration complexity
	$$
	O\left( \log(\varepsilon^{-1})\alpha^{-1}\|A\|^2_F/\sigma_{\min}^2(A)\right).
	$$
	In contrast, based on Theorem \ref{THMfm2} we have,
	for
$$\omega=\left(1-\sqrt{0.99\alpha\sigma_{\min}^2(A)/\|A\|^2_F}\right)^2,$$
the iteration complexity of Algorithm \ref{mPDFR} is
	$$
	O\left(\log (\varepsilon^{-1})\sqrt{\alpha^{-1}\|A\|^2_F/\sigma_{\min}^2(A)} \right),
	$$
	which is a quadratic improvement on the above result.
\end{remark}

\subsection{Convergence analysis: Special cases}
In this subsection, we study the convergence results of Algorithm \ref{mPDFR} (Algorithm \ref{PFR1}) for some special cases where the coefficient matrix $A$ is full column rank or the sampling matrices $T_1,T_2,S_1,S_2$ are
scalar matrices.

\subsubsection{The full column rank case}
If the matrix $A$ is full column rank, then we can study the convergence rate of the quantity $\mathbb{E}\|x^k-A^{\dagger}b\|^2_2$. We show that for a range of stepsize parameters $\alpha>0$ and
momentum terms $\omega\geq 0$, the method enjoys a global linear convergence rate.

\begin{theorem}%[\cite{DS2021}, Theorem 4]
\label{THMx}
Suppose  that $A$ has full column rank.
Let $x^1=x^0\in\mathbb{R}^n$ be  the initial vectors and denote
\begin{equation}\label{beta1}
\beta_1=\|\mathbb{E}[ S_2S_1^\top AT_2T_1^\top T_1T_2^\top A^\top S_1S_2^\top ]\|_2.
\end{equation}
Assume $0<\alpha<\frac{1}{\beta_1\|A\|^2_F}$, $\omega\geq0$ and
that the expressions
$$
\gamma_1=1+3\omega+2\omega^2-(2\alpha+\alpha\omega-2\alpha^2\beta_1\|A\|^2_F)\frac{\sigma^2_{\min}(A)}{\|A\|^2_F}
\
\mbox{and}
\
\gamma_2=2\omega^2+\omega+\omega\alpha\frac{\sigma^2_{\max}(A)}{\|A\|^2_F}
$$
satisfy $\gamma_1+\gamma_2<1$.
Then the iteration sequence  $\{x^k\}_{k=1}^{\infty}$ generated by Algorithm \ref{mPDFR} satisfies
$$
\mathbb{E}[\|x^{k+1}-A^{\dagger}b\|^2_2]\leq q^k(1+\tau)\|x^0-A^{\dagger}b\|^2_2+\frac{2\alpha^2\beta_1(1-q^k)\|AA^{\dagger}b-b\|^2_2}{1-q},\ k\geq0,
$$
where $$q=\left\{
           \begin{array}{ll}
             \frac{\gamma_1+\sqrt{\gamma_1^2+4\gamma_2}}{2}, & \hbox{if $\omega>0$;} \\
             \gamma_1, & \hbox{if $\omega=0$,}
           \end{array}
         \right.
\ \mbox{and} \ \tau=q-\gamma_1\geq 0.$$ Moreover,
$\gamma_1+\gamma_2\leq q<1$.
\end{theorem}

Furthermore, if the linear systems are consistent, then we can show that the quantity $\mathbb{E}\|x^k-A^{\dagger}b\|^2_2$ converges to zero linearly. Besides, one can also choose a larger stepsize.

\begin{theorem}
\label{THMcfull}
Consider a consistent system $Ax = b$ with $A$ being full column rank.
Let $x^1=x^0\in\mathbb{R}^n$ be  the initial vectors and  $\beta_1$ is defined as \eqref{beta1}.
Assume $0<\alpha<\frac{2}{\beta_1\|A\|^2_F}$, $\omega\geq0$ and
that the expressions
$$
\gamma_1=1+3\omega+2\omega^2
-\left(2\alpha+\alpha\omega-\alpha^2\beta_1\|A\|^2_F\right)\frac{\sigma_{\min}^2(A)}{\|A\|^2_F}
\
\mbox{and}
\
\gamma_2=2\omega^2+\omega+\omega\alpha\frac{\sigma^2_{\max}(A)}{\|A\|^2_F}
$$
satisfy $\gamma_1+\gamma_2<1$.
Then the iteration sequence of $\{x^k\}_{k=1}^{\infty}$ in Algorithm \ref{mPDFR} satisfies
$$
\mathbb{E}[\|x^{k+1}-A^{\dagger}b\|^2_2]\leq q^k(1+\tau)\|x^0-A^{\dagger}b\|^2_2,\ k\geq 0,
$$
where $$q=\left\{
           \begin{array}{ll}
             \frac{\gamma_1+\sqrt{\gamma_1^2+4\gamma_2}}{2}, & \hbox{if $\omega>0$;} \\
             \gamma_1, & \hbox{if $\omega=0$,}
           \end{array}
         \right.
\ \mbox{and} \ \tau=q-\gamma_1\geq 0.$$ Moreover,
$\gamma_1+\gamma_2\leq q<1$.
\end{theorem}

\subsubsection{The scalar matrices  case}

If  the sampling matrices $S_1,S_2$ or $T_1,T_2$ are scalar matrices, then  we can study
the convergence rate of the quantity $\mathbb{E}\|r^k-r^*\|^2_2$ or $\mathbb{E}\|x^k-A^{\dagger}b\|^2_2$  for Algorithm \ref{mPDFR} (Algorithm \ref{PFR1}).

\begin{theorem}\label{THMr}
Suppose that the sampling matrix $S_1=S_2=\sqrt{\rho} I$ and $\rho>0$. Let $x_{LS}^*$ be the least-squares or the least-norm least-squares solution of the linear system $Ax=b$ and the initial vectors $x^1=x^0\in\mathbb{R}^n$.
Denote by $r^*=Ax_{LS}^*-b$, $r^k=Ax^k-b$, and
$$\beta_2=\big\|\mathbb{E}[ T_2T_1^\top A^\top A T_1T_2^\top  ]\big\|_2.$$
Assume $0<\alpha<\frac{2}{\rho^2\beta_2\|A\|^2_F}$, $\omega\geq0$ and
that the expressions
$$
\gamma_1=1+3\omega+2\omega^2
-\left(2\alpha-\alpha^2\rho^2\beta_2\|A\|^2_F+\alpha\omega\right)\frac{\sigma_{\min}^2(A)}{\|A\|^2_F}
\
\mbox{and}
\
\gamma_2=2\omega^2+\omega+\omega\alpha\frac{\sigma^2_{\max}(A)}{\|A\|^2_F}
$$
satisfy $\gamma_1+\gamma_2<1$.
Then the iteration sequence of residuals $\{r^k\}_{k=1}^{\infty}$ generated by Algorithm \ref{mPDFR} satisfies
$$
\mathbb{E}[\|r^{k+1}-r^*\|^2_2]\leq q^k(1+\tau)\|r^0-r^*\|^2_2, \ k\geq 0,
$$
where $$q=\left\{
           \begin{array}{ll}
             \frac{\gamma_1+\sqrt{\gamma_1^2+4\gamma_2}}{2}, & \hbox{if $\omega>0$;} \\
             \gamma_1, & \hbox{if $\omega=0$,}
           \end{array}
         \right.
\ \mbox{and} \ \tau=q-\gamma_1\geq 0.$$ Moreover,
$\gamma_1+\gamma_2\leq q<1$.
\end{theorem}

If $T_1,T_2$ are scalar matrices, then we have the following two results for Algorithm \ref{mPDFR} (Algorithm \ref{PFR1}).

\begin{theorem}\label{THMconsclar}
Suppose that the linear system $Ax = b$ is consistent and the sampling matrix  $T_1=T_2=\sqrt{\rho} I$ where $\rho>0$ is a constant. Let $x^1=x^0\in\mathbb{R}^n$ be arbitrary initial vectors and denote $x^0_*=A^{\dagger}b+(I-A^\dagger A)x^0$ and
$$\beta_3=\big\|\mathbb{E}[ S_2S_1^\top AA^\top S_1S_2^\top ]\big\|_2.$$
Assume $0<\alpha<\frac{2}{\rho^2\beta_3\|A\|^2_F}$, $\omega\geq0$ and
that the expressions
$$
\gamma_1=1+3\omega+2\omega^2
-\big(2\alpha-\alpha^2\rho^2\beta_3\|A\|^2_F+\alpha\omega\big)\frac{\sigma_{\min}^2(A)}{\|A\|^2_F}
\
\mbox{and}
\
\gamma_2=2\omega^2+\omega+\omega\alpha\frac{\sigma^2_{\max}(A)}{\|A\|^2_F}
$$
satisfy $\gamma_1+\gamma_2<1$.
Then the iteration sequence of $\{x^k\}_{k=1}^{\infty}$ in Algorithm \ref{mPDFR} satisfies
$$
\mathbb{E}[\|x^{k+1}-x_*^0\|^2_2]\leq q^k(1+\tau)\|x^0-x_*^0\|^2_2, \ k\geq 0,
$$
where $$q=\left\{
           \begin{array}{ll}
             \frac{\gamma_1+\sqrt{\gamma_1^2+4\gamma_2}}{2}, & \hbox{if $\omega>0$;} \\
             \gamma_1, & \hbox{if $\omega=0$,}
           \end{array}
         \right.
\ \mbox{and} \ \tau=q-\gamma_1\geq 0.$$ Moreover,
$\gamma_1+\gamma_2\leq q<1$.
\end{theorem}

\begin{remark}
	If we choose the initial vectors  $x^1=x^0\in\mbox{Row}(A)$, then $x_0^*=A^{\dagger}b=x^*_{LS}$. From Theorem \ref{THMconsclar}, we know that the iteration sequence $ \{ x^k \}_{k=0}^\infty $ generated by Algorithm \ref{mPDFR} now converges to  the least-squares or the least-norm least-squares solution $x^*_{LS}=A^{\dagger}b$.
\end{remark}

\begin{remark}
	\label{remark75}
Suppose that $x$ is a random vector in $\mathbb{R}^{n}$ with finite mean $\mathbb{E}[x]$. Then we have
	$$
	\mathbb{E} \left[ \left\|x-x_0^{*}\right\|_2^{2} \right] =\left\|\mathbb{E} \left[ x-x_0^{*} \right] \right\|_2^{2}
	+\mathbb{E} \left[ \|x-\mathbb{E}[x]\|_2^{2}\right].
	$$
This implies that the quantity $\mathbb{E} \left[ \left\|x-x_0^{*}\right\|_2^{2} \right] $  is larger than $\left\|\mathbb{E} \left[ x-x_0^{*} \right] \right\|_2^{2}$, and hence harder to push to zero. As a corollary, the convergence rate of $\mathbb{E} \left[ \left\|x-x_0^{*}\right\|_2^{2} \right] $ to zero established in Theorem \ref{THMconsclar} implies the same rate for that of $\left\|\mathbb{E} \left[ x-x_0^{*} \right] \right\|_2^{2}$. However, note that in Theorem \ref{THMfm2} we have established an accelerated rate for $ \left\|\mathbb{E} \left[ x-x_0^{*} \right] \right\|_2^{2}$ and Theorem \ref{THMfm2} is for different types of linear systems.
\end{remark}

Finally, we present the following result for the case where $T_1T_2^\top A^\top S_1^\top S_2r^*=0$.

\begin{theorem}\label{THMr7251}
Let $x_{LS}^*$ be the least-squares or the least-norm least-squares solution of the linear system $Ax=b$ and the initial vectors $x^1=x^0\in\mathbb{R}^n$.  Denote by $r^*=Ax_{LS}^*-b$ and $r^k=Ax^k-b$. Suppose that for any random parameter matrix pair $(S_1,S_2,T_1,T_2)$ sampled from the distribution $\mathcal{D}$, it satisfies that $T_1T_2^\top A^\top S_1^\top S_2r^*=0$.
Assume $0<\alpha<\frac{2\sigma^2_{\min}(A)}{\beta\|A\|^2_F}$, $\omega\geq0$ and
that the expressions
$$
\gamma_1=1+3\omega+2\omega^2-(2\alpha+\alpha\omega)\frac{\sigma^2_{\min}(A)}{\|A\|^2_F}+\alpha^2\beta
\
\mbox{and}
\
\gamma_2=2\omega^2+\omega+\omega\alpha\frac{\sigma^2_{\max}(A)}{\|A\|^2_F}
$$
satisfy $\gamma_1+\gamma_2<1$, where $\beta$ is defined as \eqref{eta-beta}.
Then the iteration sequence of $\{x^k\}_{k=1}^{\infty}$ in Algorithm \ref{mPDFR} satisfies
$$
\mathbb{E}[\|r^{k+1}-r^*\|^2_2]\leq q^k(1+\tau)\|r^0-r^*\|^2_2,\ k\geq 0,
$$
where $$q=\left\{
           \begin{array}{ll}
             \frac{\gamma_1+\sqrt{\gamma_1^2+4\gamma_2}}{2}, & \hbox{if $\omega>0$;} \\
             \gamma_1, & \hbox{if $\omega=0$,}
           \end{array}
         \right.
\ \mbox{and} \ \tau=q-\gamma_1\geq 0.$$ Moreover,
$\gamma_1+\gamma_2\leq q<1$.
\end{theorem}

\section{Special cases and  new efficient methods}
\label{section-5}
Our framework flexibility allows us to adjust the parameter matrices and leads to a number of popular methods.
In this section, we briefly mention how by selecting
the parameters $(S_1,S_2,T_1,T_2)$ of our method, we recover several existing methods and lead to completely new
methods; see Table \ref{table1} for a quick summary.

%--------------------------- table 1 ------------------------
\begin{table}
\setlength{\tabcolsep}{2pt}
\caption{\label{table1} Special cases of Algorithm \ref{mPDFR}, where $S$ is an $m\times p$ Gaussian matrix, $T$ is an $n\times s$ Gaussian matrix, and $\eta$ is a Gaussian vector. mSGC works for the case where $A$ is a symmetric matrix.  }
\centering
{\small
\begin{tabular}{ c c c c c c c c c c l   }
\hline
  \\
 Algorithms & &$S_1$ & &$S_2$& &$T_1$& &$T_2$& &$x^{k+1}$\\[0.2cm]
\hline

\hline
\\
mRK &  &$\frac{e_j}{\|a_j\|^2_2}$ & &$\frac{e_j}{\|a_j\|^2_2}$& &$I$& &$I$& &$x^k-\alpha\frac{\langle a_j,x^k\rangle-b_j}{\|a_j\|^2_2}a_j+\omega(x^k-x^{k-1})$\\[0.35cm]
mRGS &  &$I$ & &$I$& &$\frac{e_i}{\|A_i\|^2_2}$& &$\frac{e_i}{\|A_i\|^2_2}$& &$x^k-\alpha\frac{A^\top_i(Ax^k-b)}{\|A_i\|^2_2}e_i+\omega(x^k-x^{k-1})$\\[0.35cm]
mDSGS &  &$\frac{e_i}{|a_{i,j}|}$& &$\frac{e_i}{|a_{i,j}|}$& &$e_j$ & &$e_j$& &$x^{k}-\alpha\frac{\langle a_i,x^{k}\rangle-b_i}{a_{i,j}}e_j+\omega(x^{k}-x^{k-1})$\\[0.35cm]
mRBK &  &$\sqrt{\frac{m}{p}}I_{:\mathcal{R}}$& &$\sqrt{\frac{m}{p}}I_{:\mathcal{R}}$& &$\frac{I}{\|A\|_F}$ & &$\frac{I}{\|A\|_F}$& &$x^k-\frac{\alpha m}{p\|A\|^2_F}(A_{\mathcal{R}:})^\top (A_{\mathcal{R}:}x^k-b_{\mathcal{R}})+\omega(x^k-x^{k-1})$\\[0.35cm]
mRBCD &  &$\frac{I}{\|A\|_F}$ & &$\frac{I}{\|A\|_F}$& &$\sqrt{\frac{n}{s}}I_{:\mathcal{L}}$& &$\sqrt{\frac{n}{s}}I_{:\mathcal{L}}$&  &$x^k-\frac{\alpha n}{s\|A\|^2_F}I_{:\mathcal{L}}(A_{:\mathcal{L}})^\top (Ax^k-b)+\omega(x^k-x^{k-1})$\\[0.35cm]
mBGK &  &$S$& &$S$& &$\frac{I}{\sqrt{p}\|A\|_F}$& &$\frac{I}{\sqrt{p}\|A\|_F}$&  &$x^k-\alpha\frac{A^\top S S^\top(Ax^k-b)}{p\|A\|^2_F}+\omega(x^k-x^{k-1})$\\[0.35cm]
mBGLS &  &$\frac{I}{\sqrt{s}\|A\|_F}$& &$\frac{I}{\sqrt{s}\|A\|_F}$& &$T$& &$T$&   &$x^k-\alpha\frac{TT^\top A^\top(Ax^k-b)}{s\|A\|^2_F}+\omega(x^k-x^{k-1})$\\[0.35cm]
mSGC &  &$\eta$& &$\frac{e_j}{\mbox{Tr}(A)}$& &$\frac{e_i}{a_{i,j}}$& &$\eta$&   &$x^k-\frac{\alpha \eta^\top A\eta}{\mbox{Tr}(A)}\cdot \frac{\langle a_j,x^k\rangle-b_j}{a_{i,j}}e_i+\omega(x^k-x^{k-1})$\\[0.2cm]
\hline
\end{tabular}
}
\end{table}

\subsection{Randomized Kaczmarz method with momentum}
The randomized Kaczmarz method is one special case of Algorithm \ref{PFR1}. Choosing $T_1=T_2=I$ and $S_1=S_2=\frac{e_j}{\|a_j\|_2}$ with probability $\frac{\|a_j\|^2_2}{\|A\|^2_F}$, we have
$$
\mathbb{E}[A^\top S_1S_2^\top]=\mathbb{E}\bigg[A^\top\frac{e_je_j^\top}{\|a_j\|^2_2}\bigg]=\frac{A^\top}{\|A\|^2_F},
$$
and Algorithm \ref{PFR1} recovers the RK iteration
$$
x^{k+1}=x^k-\alpha\frac{\langle a_j,x^k\rangle-b_j}{\|a_j\|^2_2}a_j.
$$
For this case, we have $\beta_3=\frac{1}{\|A\|^2_F}$ and Theorem \ref{THMconsclar} yields the convergence
estimate of \cite[Theorem 1]{Cai12}:
$$
\mathbb{E}[\|x^{k}-x^0_*\|^2_2]\leq \bigg(1-(2\alpha-\alpha^2)\frac{\sigma^2_{\min}(A)}{\|A\|^2_F}\bigg)^k\|x^{0}-x^0_*\|^2_2,
$$
where the stepsize parameter $\alpha\in(0,2)$. Furthermore, Algorithm \ref{mPDFR} recovers the RK with momentum (mRK)
$$
x^{k+1}=x^k-\alpha\frac{\langle a_j,x^k\rangle-b_j}{\|a_j\|^2_2}a_j+\omega(x^k-x^{k-1}).
$$
Theorem \ref{THMconsclar} also yields the following convergence result of mRK for solving the consistent linear systems:
$$
\mathbb{E}[\|x^{k+1}-x^0_*\|^2_2]\leq q^k(1+\tau)\|x^0-x^0_*\|^2_2,\ k\geq 0,
$$
where $q=\frac{\gamma_1+\sqrt{\gamma_1^2+4\gamma_2}}{2}$, $\tau=q-\gamma_1\geq 0$ with
$
\gamma_1=1+3\omega+2\omega^2
-\big(2\alpha-\alpha^2+\alpha\omega\big)\frac{\sigma_{\min}^2(A)}{\|A\|^2_F}
$, $
\gamma_2=2\omega^2+\omega+\omega\alpha\frac{\sigma^2_{\max}(A)}{\|A\|^2_F}
$, and the stepsize  parameter $0<\alpha<2$. We note that when $\omega=0$, then $\gamma_1=1-\big(2\alpha-\alpha^2\big)\frac{\sigma_{\min}^2(A)}{\|A\|^2_F}>0$, and now  $q=\frac{\gamma_1+\sqrt{\gamma_1^2+4\gamma_2}}{2}=\gamma_1$.

The above result has first been established in \cite[Theorem $1$]{loizou2020momentum}. However, they only treated the consistent linear systems.
In this paper, our convergence results are more general, see Theorems \ref{THMr-in629} and \ref{THMfm2}, which work for various types of linear systems.
We note that the method studied in \cite{loizou2020momentum} is in the setting of quadratic objectives, where the equivalence between  stochastic gradient descent, stochastic Newton method, and stochastic proximal point method can be established.

\subsection{Randomized Gauss-Seidel method with momentum}
The randomized Gauss-Seidel (RGS) method, also known as the randomized coordinate descent (RCD) method, is one special case of Algorithm \ref{PFR1}. Choosing $S_1=S_2=I$ and $T_1=T_2=\frac{e_i}{\|A_i\|_2}$ with probability $\frac{\|A_i\|^2_2}{\|A\|^2_F}$, we have
$$
\mathbb{E}[T_1T_2^\top A^\top]=\mathbb{E}\bigg[\frac{e_ie_i^\top}{\|A_i\|^2_2}A^\top\bigg]=\frac{A^\top}{\|A\|^2_F},
$$
and Algorithm \ref{PFR1} recovers the RGS/RCD iteration
$$
x^{k+1}=x^k-\alpha\frac{A^\top_i(Ax^k-b)}{\|A_i\|^2_2}e_i.
$$
For this case, we have $\beta_2=\frac{1}{\|A\|^2_F}$ and Theorem \ref{THMr} yields the convergence
estimate of \cite{Gri12}:
$$
\mathbb{E}[\|r^{k}-r^*\|^2_2]\leq \bigg(1-(2\alpha-\alpha^2)\frac{\sigma^2_{\min}(A)}{\|A\|^2_F}\bigg)^k\|r^{0}-r^*\|^2_2,
$$
where the stepsize parameter $\alpha\in(0,2)$. In addition,  Algorithm \ref{mPDFR} derives the following momentum variant of RGS/RCD (mRGS/mRCD):
$$
x^{k+1}=x^k-\alpha\frac{A^\top_i(Ax^k-b)}{\|A_i\|^2_2}e_i+\omega(x^{k}-x^{k-1}).
$$
Theorem \ref{THMr} yields the following convergence result for mRGS/mRCD:
$$
\mathbb{E}[\|r^{k+1}-r^*\|^2_2]\leq q^k(1+\tau)\|r^0-r^*\|^2_2, \ k\geq 0,
$$
where $q=\frac{\gamma_1+\sqrt{\gamma_1^2+4\gamma_2}}{2}$ and $\tau=q-\gamma_1\geq 0$
with
$
\gamma_1=1+3\omega+2\omega^2
-\left(2\alpha-\alpha^2+\alpha\omega\right)\frac{\sigma_{\min}^2(A)}{\|A\|^2_F}
$, $
\gamma_2=2\omega^2+\omega+\omega\alpha\frac{\sigma^2_{\max}(A)}{\|A\|^2_F}
$, and $0<\alpha<2$.
We note that for solving the consistent linear system, the mRCD has also been studied in  \cite{loizou2020momentum}.

\subsection{Doubly stochastic Gauss-Seidel with momentum}
The doubly stochastic Gauss-Seidel (DSGS) method  \cite{RHRL2019} is recovered if one takes $S_1=S_2=\frac{e_i}{|a_{i,j}|}\in\mathbb{R}^m$ and $T_1=T_2=e_j\in\mathbb{R}^n$, and the index pair $(i,j)$ is randomly selected with probability $\frac{a_{i,j}^2}{\|A\|^2_F}$. Indeed, now we have
$$
\mathbb{E}[T_1T_2^\top A^\top S_1S_2^\top]=\mathbb{E}\bigg[\frac{e_je_j^\top A^\top e_ie_i^\top}{a_{i,j}^2}\bigg]
=\sum\limits_{i=1}^m\sum\limits_{j=1}^n\frac{e_je_j^\top A^\top e_ie_i^\top}{a_{i,j}^2} \frac{a_{i,j}^2}{\|A\|_F^2}=\frac{A^\top}{\|A\|^2_F},
$$
and Algorithm \ref{PFR1} recovers the DSGS iteration:
$$
x^{k+1}=x^{k}-\alpha\frac{\langle a_i,x^{k}\rangle-b_i}{a_{i,j}}e_j.
$$
Now we have
$$
\begin{array}{ll}
\beta&=\|\mathbb{E}[S_2S_1^\top AT_2 T_1^\top A^\top AT_1T_2^\top A^\top S_1S_2^\top]\|_2=\bigg\|\mathbb{E}\bigg[\frac{e_ie_i^\top A e_j e_j^\top A^\top A e_je_j^\top A^\top e_ie_i^\top}{a_{i,j}^4}\bigg]\bigg\|_2\\
&=\bigg\|\sum\limits_{i=1}^m\sum\limits_{j=1}^n\frac{a_{i,j}^2}{\|A\|^2_F}\frac{e_ia_{i,j} e_j^\top A^\top A e_j a_{i,j}e_i^\top}{a_{i,j}^4}\bigg\|_2
=\bigg\|\sum\limits_{i=1}^m\sum\limits_{j=1}^n\frac{e_iA_j^\top A_je_i^\top}{\|A\|^2_F}\bigg\|_2=\|\sum\limits_{i=1}^m e_ie_i^\top\|_2
\\&=1.
\end{array}
$$
and
$$
\begin{array}{ll}
\beta_1&=\|\mathbb{E}[S_2S_1^\top AT_2 T_1^\top T_1T_2^\top A^\top S_1S_2^\top]\|_2=\bigg\|\mathbb{E}\bigg[\frac{e_ie_i^\top A e_j e_j^\top  e_je_j^\top A^\top e_ie_i^\top}{a_{i,j}^4}\bigg]\bigg\|_2\\
&=\bigg\|\sum\limits_{i=1}^m\sum\limits_{j=1}^n\frac{a_{i,j}^2}{\|A\|^2_F}\frac{e_ia_{i,j} e_j^\top  e_j a_{i,j}e_i^\top}{a_{i,j}^4}\bigg\|_2
=\bigg\|\sum\limits_{i=1}^m\sum\limits_{j=1}^n\frac{e_ie_i^\top}{\|A\|^2_F}\bigg\|_2=\frac{n}{\|A\|^2_F}\big\|\sum\limits_{i=1}^m e_ie_i^\top\big\|_2
\\
&=\frac{n}{\|A\|^2_F}.
\end{array}
$$
By Algorithm \ref{mPDFR}, we can obtain the DSGS with momentum (mDSGS)
$$
x^{k+1}=x^{k}-\alpha\frac{\langle a_i,x^{k}\rangle-b_i}{a_{i,j}}e_j+\omega(x^{k}-x^{k-1}).
$$

Let us consider the convergence properties of mDSGS.
If $A$ is full column rank and the linear system is consistent, then Theorem \ref{THMcfull} yields the following global convergence rate for mDSGS:
$$
\mathbb{E}[\|x^{k+1}-A^{\dagger}b\|^2_2]\leq q^k(1+\tau)\|x^0-A^{\dagger}b\|^2_2,\ k\geq 0,
$$
where $q=\frac{\gamma_1+\sqrt{\gamma_1^2+4\gamma_2}}{2}$ and $\tau=q-\gamma_1\geq 0$ with
$
\gamma_1=1+3\omega+2\omega^2
-\left(2\alpha+\alpha\omega-n\alpha^2\right)\frac{\sigma_{\min}^2(A)}{\|A\|^2_F}
$,
$
\gamma_2=2\omega^2+\omega+\omega\alpha\frac{\sigma^2_{\max}(A)}{\|A\|^2_F}
$, and the stepsize parameter $0<\alpha<\frac{2}{n}$.
When $\omega=0$, the above convergence result recovers the result obtained in \cite[Theorem 1]{RHRL2019} for DSGS.

If $A$ is not full column rank and the linear system is consistent, then Theorem \ref{THMr7251} yields the  following global convergence rate for mDSGS:
$$
\mathbb{E}[\|r^{k+1}-r^*\|^2_2]\leq q^k(1+\tau)\|r^0-r^*\|^2_2,\ k\geq 0,
$$
where $q=\frac{\gamma_1+\sqrt{\gamma_1^2+4\gamma_2}}{2}$ and $\tau=q-\gamma_1\geq 0$ with $
\gamma_1=1+3\omega+2\omega^2-(2\alpha+\alpha\omega)\frac{\sigma^2_{\min}(A)}{\|A\|^2_F}+\alpha^2
$,
$
\gamma_2=2\omega^2+\omega+\omega\alpha\frac{\sigma^2_{\max}(A)}{\|A\|^2_F}
$, and $0<\alpha<\frac{2\sigma^2_{\min}(A)}{\|A\|^2_F}$.  When $\omega=0$, the above convergence result recovers the result obtained in \cite[Theorem 2]{RHRL2019} for DSGS.

Theorem \ref{THMr-in629} yields the  following convergence result of mDSGS for solving different types of linear systems:
 $$
\mathbb{E}[\|r^{k+1}-r^*\|^2_2]\leq q^k(1+\tau)\|r^0-r^*\|^2_2+\frac{2\alpha^2(1-q^k)\|r^*\|^2_2}{1-q}, \ k\geq 0,
$$
 where $q=\frac{\gamma_1+\sqrt{\gamma_1^2+4\gamma_2}}{2}$ and $\tau=q-\gamma_1\geq 0$ with
 $
\gamma_1=1+3\omega+2\omega^2-(2\alpha+\alpha\omega)\frac{\sigma^2_{\min}(A)}{\|A\|^2_F}+2\alpha^2
$,
$
\gamma_2=2\omega^2+\omega+\omega\alpha\frac{\sigma^2_{\max}(A)}{\|A\|^2_F}
$
, and $0<\alpha<\frac{\sigma^2_{\min}(A)}{\|A\|^2_F}$.
To the best of our knowledge, mDSGS has never been analyzed before in any setting.

\subsection{Randomized block Kaczmarz with momentum}
\label{sbst:mRBK}
Our framework also extends to block formulations of the RK method studied in \cite{DS2021}. However, to the best of our knowledge, no momentum variants of this method were analyzed before.

Assume $1\leq p\leq m$. Let $\mathcal{R}$ denote the
set consisting of the uniform sampling of $p$ different numbers of $[m]$. Choose $S = \sqrt{m/p}I_{:\mathcal{R}}$ to be a column concatenation of the columns of the $m\times m$ identity matrix $I$ indexed by $\mathcal{R}$. Let $S_1=S_2=S$ and $T_1=T_2=\frac{I}{\|A\|_F}$.
Then we know that
$$
\mathbb{E}\bigg[\frac{1}{\|A\|^2_F}A^\top S S^\top\bigg]=\frac{A^\top}{\|A\|^2_F} \cdot\frac{m/p}{{m\choose p}}\sum\limits_{\mathcal{R}\subset[m],|\mathcal{R}|=p} I_{:\mathcal{R}}I_{:\mathcal{R}}^\top=\frac{A^\top}{\|A\|^2_F}\cdot \frac{m/p}{{m\choose p}}{m-1\choose p-1}I=\frac{A^\top}{\|A\|^2_F}
$$
and if $p\geq2$, then
\begin{equation}\label{rbkbeta31}
\begin{array}{ll}
\beta_3&=\frac{m^2}{p^2}\left\|\mathbb{E}\left[I_{:\mathcal{R}}I_{:\mathcal{R}}^\top A A^\top I_{:\mathcal{R}}I_{:\mathcal{R}}^\top \right]\right\|_2
\\
&=\frac{m^2}{p^2}\frac{{{m-2\choose p-2}}}{{{m\choose p}}}\left\|AA^\top+\frac{m-p}{p-1}\mbox{diag}(AA^\top)\right\|_2
\\
&=\frac{m(p-1)}{(m-1)p}\left\|AA^\top+\frac{m-p}{p-1}\mbox{diag}(AA^\top)\right\|_2,
\end{array}
\end{equation}
and if $p=1$, then
\begin{equation}\label{rbkbeta32}
\beta_3=m\left\|\mbox{diag}(AA^\top)\right\|_2.
\end{equation}
By Algorithm \ref{PFR1}, the randomized block Kaczmarz (RBK) method updates with the following iterative strategy:
\begin{equation}
\label{rbk77}
x^{k+1}=x^k-\frac{\alpha m}{p\|A\|^2_F}(A_{\mathcal{R}:})^\top (A_{\mathcal{R}:}x^k-b_{\mathcal{R}}).
\end{equation}
Assume that $\|a_i\|_2=1,i=1,\ldots,m$,  the RBK algorithm \eqref{rbk77} can be viewed as a special case of the randomized algorithm \eqref{BKM} proposed in \cite{Nec19}. Indeed, let $J$ be the
set consisting of the uniform sampling of $p$ different numbers of $[m]$, $\omega_i=\frac{1}{p}$ and choose the parameter $\alpha$ appropriately, \eqref{BKM} recovers \eqref{rbk77}.
By Algorithm \ref{mPDFR}, the momentum variant of RBK (mRBK) updates with the following iterative strategy:
$$
x^{k+1}=x^k-\frac{\alpha m}{p\|A\|^2_F}(A_{\mathcal{R}:})^\top (A_{\mathcal{R}:}x^k-b_{\mathcal{R}})+\omega(x^k-x^{k-1}).
$$
Theorem \ref{THMconsclar} yields the following global convergence rate of mRBK for solving consistent linear systems:
$$
\mathbb{E}[\|x^{k+1}-x_*^0\|^2_2]\leq q^k(1+\tau)\|x^0-x_*^0\|^2_2, \ k\geq 0,
$$
where $$q=\left\{
           \begin{array}{ll}
             \frac{\gamma_1+\sqrt{\gamma_1^2+4\gamma_2}}{2}, & \hbox{if $\omega>0$;} \\
             \gamma_1, & \hbox{if $\omega=0$,}
           \end{array}
         \right.
\ \mbox{and} \ \tau=q-\gamma_1\geq 0$$ with
$
\gamma_1=1+3\omega+2\omega^2
-\big(2\alpha-\alpha^2\beta_3/\|A\|^2_F+\alpha\omega\big)\frac{\sigma_{\min}^2(A)}{\|A\|^2_F}
$,
$
\gamma_2=2\omega^2+\omega+\omega\alpha\frac{\sigma^2_{\max}(A)}{\|A\|^2_F}
$, $0<\alpha<\frac{2\|A\|^2_F}{\beta_3}$, and $\beta_3$ is given by \eqref{rbkbeta31} and \eqref{rbkbeta32}.

\subsection{Randomized block coordinate descent with momentum}
\label{sbst:mRBCD}
Our framework extends to block formulations of the randomized coordinate descent method studied in \cite{DS2021}.
To the best of our knowledge, this is the first time that the momentum variant of this method is considered.

Assume $1\leq s\leq n$. Let $\mathcal{L}$ denote the
set consisting of the uniform sampling of $r$ different numbers of $[n]$. Choose $T = \sqrt{n/s}I_{:\mathcal{L}}$ to be a column concatenation of the columns of the $n\times n$ identity matrix $I$ indexed by $\mathcal{L}$. Let $T_1=T_2=T$ and $S_1=S_2=\frac{I}{\|A\|_F}$.
Then we know that
$$
\mathbb{E}\bigg[\frac{1}{\|A\|^2_F}TT^\top A^\top \bigg]=\frac{n/s}{{n\choose s}}\sum\limits_{\mathcal{L}\subset[n],|\mathcal{L}|=s} I_{:\mathcal{L}}I_{:\mathcal{L}}^\top \cdot \frac{A^\top}{\|A\|^2_F} = \frac{n/s}{{n\choose s}}{n-1\choose s-1}I \cdot \frac{A^\top}{\|A\|^2_F}=\frac{A^\top}{\|A\|^2_F}
$$
and if $s\geq2$, then
\begin{equation}\label{rbcdbeta21}
\begin{array}{ll}
\beta_2&=\frac{n^2}{s^2}\left\|\mathbb{E}\left[I_{:\mathcal{L}}I_{:\mathcal{L}}^\top A^\top A I_{:\mathcal{L}}I_{:\mathcal{L}}^\top \right]\right\|_2
\\
&=\frac{n^2}{s^2}\frac{{{n-2\choose s-2}}}{{{n\choose s}}}\left\|A^\top A+\frac{n-s}{s-1}\mbox{diag}(A^\top A)\right\|_2
\\
&=\frac{n(s-1)}{(n-1)s}\left\|A^\top A+\frac{n-s}{s-1}\mbox{diag}(A^\top A)\right\|_2,
\end{array}
\end{equation}
and if $s=1$, then
\begin{equation}\label{rbcdbeta22}
\beta_2=n\left\|\mbox{diag}(A^\top A )\right\|_2.
\end{equation}
By Algorithm \ref{PFR1}, the randomized block coordinate descent (RBCD) method updates with the following iterative strategy:
$$x^{k+1}=x^k-\frac{\alpha n}{s\|A\|^2_F}I_{:\mathcal{L}}(A_{:\mathcal{L}})^\top (Ax^k-b).$$
By Algorithm \ref{mPDFR}, the momentum variant of RBCD (mRBCD) updates with the following iterative strategy:
$$
x^{k+1}=x^k-\frac{\alpha n}{s\|A\|^2_F}I_{:\mathcal{L}}(A_{:\mathcal{L}})^\top (Ax^k-b)+\omega(x^k-x^{k-1}).
$$
Theorem \ref{THMr} shows that mRBCD  converges linearly in expectation for solving different types of linear systems
$$
\mathbb{E}[\|r^{k+1}-r^*\|^2_2]\leq q^k(1+\tau)\|r^0-r^*\|^2_2, \ k\geq 0,
$$
where $$q=\left\{
           \begin{array}{ll}
             \frac{\gamma_1+\sqrt{\gamma_1^2+4\gamma_2}}{2}, & \hbox{if $\omega>0$;} \\
             \gamma_1, & \hbox{if $\omega=0$,}
           \end{array}
         \right.
\ \mbox{and} \ \tau=q-\gamma_1\geq 0$$ with
$
\gamma_1=1+3\omega+2\omega^2
-\left(2\alpha-\alpha^2\beta_2/\|A\|^2_F+\alpha\omega\right)\frac{\sigma_{\min}^2(A)}{\|A\|^2_F}
$,
$
\gamma_2=2\omega^2+\omega+\omega\alpha\frac{\sigma^2_{\max}(A)}{\|A\|^2_F}
$, $0<\alpha<\frac{2\|A\|^2_F}{\beta_2}$, and $\beta_2$ is given by \eqref{rbcdbeta21} and \eqref{rbcdbeta22}.

\subsection{Block Gaussian Kaczmarz with momentum}
In this subsection, we study the  method based on block Gaussian row sampling and derive a new randomized block method.

Let us first propose one special case of Algorithm \ref{PFR1} by using block Gaussian row sampling and refer to it as the \emph{block Gaussian Kaczmarz} (BGK) method.
Choose $S$ to be an $m\times p$ random matrix whose entries are independent, mean zero, Gaussian random variables, and set $S_1=S_2=S$.
Let $T_1=T_2=\frac{I}{\sqrt{p}\|A\|_F}$.  Then we know that
$$
\mathbb{E}\bigg[\frac{1}{p\|A\|^2_F}A^\top S S^\top\bigg]=\frac{A^\top}{\|A\|^2_F}
$$
and
$$
\beta_3=\big\|\mathbb{E}[S_2S_1^\top A A^\top S_1S_2^\top ]\big\|_2=\big\|\mathbb{E}[ S S^\top AA^\top S S^\top ]\big\|_2=(p^2+p)\|A\|_2^2+p\|A\|^2_F,
$$
where the second equality follows from Lemma \ref{lemma-gaussmatrix}.
It follows from Algorithm \ref{PFR1} that the BGK method constructs $x^{k+1}$ by
$$
x^{k+1}=x^k-\alpha\frac{A^\top S S^\top(Ax^k-b)}{p\|A\|^2_F}.
$$
If we take $p=1$, then we can get the Gaussian Kaczmarz proposed in \cite{Gow15}.
By Algorithm \ref{mPDFR}, we can also obtain the accelerated BGK with momentum (mBGK)
$$
x^{k+1}=x^k-\alpha\frac{A^\top S S^\top(Ax^k-b)}{p\|A\|^2_F}+\omega(x^k-x^{k-1}).
$$
Theorem \ref{THMconsclar} yields the following global convergence rate of mBGK for solving consistent linear systems:
$$
\mathbb{E}[\|x^{k+1}-x_*^0\|^2_2]\leq q^k(1+\tau)\|x^0-x_*^0\|^2_2, \ k\geq 0,
$$
where $$q=\left\{
           \begin{array}{ll}
             \frac{\gamma_1+\sqrt{\gamma_1^2+4\gamma_2}}{2}, & \hbox{if $\omega>0$;} \\
             \gamma_1, & \hbox{if $\omega=0$,}
           \end{array}
         \right.
\ \mbox{and} \ \tau=q-\gamma_1\geq 0$$ with
$$
\gamma_1=1+3\omega+2\omega^2
-\left(2\alpha-\alpha^2\frac{(p+1)\|A\|^2_2+\|A\|^2_F}{p\|A\|^2_F}+\alpha\omega\right)\frac{\sigma_{\min}^2(A)}{\|A\|^2_F}
,$$
$
\gamma_2=2\omega^2+\omega+\omega\alpha\frac{\sigma^2_{\max}(A)}{\|A\|^2_F}
$, and $0<\alpha<\frac{2p\|A\|^2_F}{(p+1)\|A\|^2_2+\|A\|^2_F}$.

\subsection{Block Gaussian least-squares with momentum}
In this subsection, we study the  method based on block Gaussian column sampling.
Choose $T$ to be an $n\times s$ random matrix whose entries are independent, mean zero, Gaussian random variables, and set $T_1=T_2=T$.
Let $S_1=S_2=\frac{I}{\sqrt{s}\|A\|_F}$.
Then we know that
$$
\mathbb{E}\bigg[ \frac{1}{s\|A\|^2_F}T T^\top A^\top\bigg]=\frac{A^\top}{\|A\|_F^2}
$$
and
$$
\beta_2=\|\mathbb{E}[ T_2T_1^\top A^\top A T_2T_1^\top   ]\|_2=\|\mathbb{E}[ T T^\top A^\top A T T^\top ]\|_2=(s^2+s)\|A\|_2^2+s\|A\|^2_F,
$$
where the second equality follows from Lemma \ref{lemma-gaussmatrix}.
Algorithm \ref{PFR1} then has the following form
$$
x^{k+1}=x^k-\alpha\frac{TT^\top A^\top(Ax^k-b)}{s\|A\|^2_F},
$$
which we call the \emph{block Gaussian least-squares} (BGLS) method.
By Algorithm \ref{mPDFR}, we can also obtain the accelerated BGLS with momentum (mBGLS)
$$
x^{k+1}=x^k-\alpha\frac{TT^\top A^\top(Ax^k-b)}{s\|A\|^2_F}+\omega(x^k-x^{k-1}).
$$
By Theorem \ref{THMr}, we know that  mBGLS yields a global convergence rate for solving various types of linear systems
$$
\mathbb{E}[\|r^{k+1}-r^*\|^2_2]\leq q^k(1+\tau)\|r^0-r^*\|^2_2, \ k\geq 0,
$$
where $$q=\left\{
           \begin{array}{ll}
             \frac{\gamma_1+\sqrt{\gamma_1^2+4\gamma_2}}{2}, & \hbox{if $\omega>0$;} \\
             \gamma_1, & \hbox{if $\omega=0$,}
           \end{array}
         \right.
\ \mbox{and} \ \tau=q-\gamma_1\geq 0$$ with
$$
\gamma_1=1+3\omega+2\omega^2
-\left(2\alpha-\alpha^2\frac{(s+1)\|A\|^2_2+\|A\|^2_F}{s\|A\|^2_F}+\alpha\omega\right)\frac{\sigma_{\min}^2(A)}{\|A\|^2_F}
,$$
$
\gamma_2=2\omega^2+\omega+\omega\alpha\frac{\sigma^2_{\max}(A)}{\|A\|^2_F}
$, and $0<\alpha<\frac{2s\|A\|^2_F}{(s+1)\|A\|^2_2+\|A\|^2_F}$.

\subsection{The symmetric matrix cases}
Recall that all of the special cases discussed above focus exclusively on the cases $T_1=T_2$ and $S_1=S_2$. In this subsection, we show that one can choose $T_1\neq T_2$ and $S_1\neq S_2$.

We consider the case where $A\in\mathbb{R}^{n\times n}$ is a symmetric matrix with $\mbox{Tr}(A)\neq 0$. Choosing $S_1=T_2=\eta$ with $\eta$ being a random vector $\eta\backsim\mathcal{N}(0,I_n)$, and $(T_1,S_2)=\left(\frac{e_i}{a_{i,j}},\frac{e_j}{\mbox{Tr}(A)}\right)$, and the index pair $(i,j)$ is randomly selected with probability $\frac{a_{i,j}^2}{\|A\|^2_F}$. Now we have
$$
\mathbb{E}[T_1T_2^\top A^\top S_1S_2^\top]=\frac{1}{\mbox{Tr}(A)}\mathbb{E}\left[\frac{e_i}{a_{i,j}}\eta^\top A \eta e_j^\top\right]
=\frac{1}{\mbox{Tr}(A)}\mathbb{E}[\eta^\top A \eta ]\mathbb{E}\left[\frac{e_i e_j^\top}{a_{i,j}}\right]=\frac{A}{\|A\|^2_F}
$$
and
$$
\begin{array}{ll}
\beta&=\left\|\mathbb{E}\bigg[\frac{e_j\eta^\top A \eta e_i^\top A A e_i \eta^\top A \eta e_j^\top}{a_{i,j}^2(\mbox{Tr}(A))^2}\bigg]\right\|_2\\
&=\frac{1}{(\mbox{Tr}(A))^2}\left\|\sum\limits_{i=1}^n\sum\limits_{j=1}^n\frac{a_{i,j}^2}{\|A\|^2_F}\frac{ e_i^\top A A e_i e_je_j^\top}{a_{i,j}^2}\cdot\mathbb{E}\left[(\eta^\top A \eta)^2\right]\right\|_2
\\
&=\frac{1}{(\mbox{Tr}(A))^2}\left\|I\cdot\mathbb{E}\left[(\eta^\top A \eta)^2\right]\right\|_2
\\
&=\frac{3\|A\|^4_2}{(\mbox{Tr}(A))^2}.
\end{array}
$$
By Algorithm \ref{mPDFR}, we get the following algorithm:
$$
x^{k+1}=x^k-\frac{\alpha \eta^\top A\eta}{\mbox{Tr}(A)}\cdot \frac{\langle a_j,x^k\rangle-b_j}{a_{i,j}}e_i+\omega(x^k-x^{k-1}),
$$
which we call the \emph{symmetric Gaussian  coordinate} with momentum (mSGC).
Theorem \ref{THMr-in629} yields  the following global convergence rate for mSGC:
$$
\mathbb{E}[\|r^{k+1}-r^*\|^2_2]\leq q^k(1+\tau)\|r^0-r^*\|^2_2+\frac{6\alpha^2\|A\|^4_2(1-q^k)\|r^*\|^2_2}{(1-q)(\mbox{Tr}(A))^2}, \ k\geq 0,
$$
where $$q=\left\{
           \begin{array}{ll}
             \frac{\gamma_1+\sqrt{\gamma_1^2+4\gamma_2}}{2}, & \hbox{if $\omega>0$;} \\
             \gamma_1, & \hbox{if $\omega=0$,}
           \end{array}
         \right.
\ \mbox{and} \ \tau=q-\gamma_1\geq 0$$ with $
\gamma_1=1+3\omega+2\omega^2-(2\alpha+\alpha\omega)\frac{\sigma^2_{\min}(A)}{\|A\|^2_F}+\frac{6\alpha^2\|A\|^4_2}{(\mbox{Tr}(A))^2}
$
,
$
\gamma_2=2\omega^2+\omega+\omega\alpha\frac{\sigma^2_{\max}(A)}{\|A\|^2_F}
$, and  $0<\alpha<\frac{\sigma^2_{\min}(A)(\mbox{Tr}(A))^2}{3\|A\|_2^4\|A\|^2_F}$.

\section{Numerical experiments}

In this section, we study the computational behavior of the proposed algorithms. In particular, we focus mostly on the evaluation of the performance of  mRBK, mRBCD, mBGK, and mBGLS.
We divide the algorithms discussed in Section \ref{section-5} (see Table \ref{table1}) into two groups
for comparison: (1) mRK, mDGSG, mRBK, and mBGK for solving consistent linear systems; (2) mRGS, mRBCD, and mBGLS for solving inconsistent linear systems.

All the methods are implemented in  {\sc Matlab} R2022a for Windows $10$ on a desktop PC with Intel(R) Core(TM) i7-10710U CPU @ 1.10GHz  and 16 GB memory.

\subsection{Numerical setup}
In our test, we use the  relative solution error (RSE)
$$\mbox{RSE}=\frac{\|x^k-x^0_*\|^2_2}{\|x^0-x^0_*\|^2_2},$$
or the relative residual error (RRE)
 $$\mbox{RRE}=\frac{\|r^k-r^*\|^2_2}{\|r^0-r^*\|^2_2},$$
 as the stopping criterion.
 We use $\kappa_2(A) = \|A\|_2\|A^\dagger\|_2$ to denote the $2$-norm condition number.

During our experiment, we set the stepsize parameter $\alpha=1$ for both mRK and mRGS. For mDSGS, we set $\alpha=\frac{1}{n}$ if the matrix $A$ is full-rank, and set $\alpha=\frac{\sigma_{\min}^2(A)}{\|A\|^2_2}$ otherwise. We set $\alpha=\frac{\|A\|_F^2}{\beta_3}$ for mRBK, where $\beta_3$ is given by \eqref{rbkbeta31} and \eqref{rbkbeta32}. For mBGK, we set $\alpha=\frac{p\|A\|^2_F}{(p+1)\|A\|^2_2+\|A\|^2_F}$. We set $\alpha=\frac{\|A\|_F^2}{\beta_2}$ for mRBCD, where $\beta_2$ is given by \eqref{rbcdbeta21} and \eqref{rbcdbeta22}. For mBGLS, we set $\alpha=\frac{s\|A\|^2_F}{(s+1)\|A\|^2_2+\|A\|^2_F}$. We note that the choices of those stepsize parameters $\alpha$ are not arbitrary. From the discussions in Section  \ref{section-5}, we know that the best convergence rates of those algorithms with no-momentum are obtained precisely for those choices of $\alpha$. Therefore, the comparisons in our experiment will be with the best-in-theory no-momentum variants.

\subsection{The effect of  heavy ball momentum}
In this subsection, we study the computational behavior of momentum variants of the proposed methods and compare them with their no momentum variants for both synthetic and real data.

{\bf Synthetic data.}  The synthetic data is generated based on the {\sc Matlab} function {\tt randn}. Specifically, we generate the matrix $A$ by using the {\sc Matlab} function {\tt randn(m,n)} and generate the exact solution $x^*$
by {\tt randn(n,1)}. The consistent system is constructed by setting $b=Ax^*$. For the inconsistent system, we first generate a $\triangle b={\tt null(A')*randn(m-r,1)}$, then set $b=Ax^*+\triangle b$, where {\tt r} is the rank of $A$.

We also test our algorithms on problem instances generated with different given values of condition number  $\kappa_2(A)$. To do so, we first use the {\sc Matlab} function {\tt randn(m,n)} to generate a Gaussian matrix. Then, we use SVD to modify the singular values to obtain a matrix with the desired values of $\kappa_2(A)$. This is done by linearly scaling the differences of all singular values with $\sigma^2_{\min}(A)$.

{\bf Real data.} The real-world data are available via the SuiteSparse Matrix Collection \cite{Kol19}. Each
dataset consists of a matrix $A\in\mathbb{R}^{m\times n}$ and a vector $b\in\mathbb{R}^m$. In our experiments, we only use the matrices $A$ of the datasets and ignore the vector $b$.
As before, to ensure the consistency of the linear system, we first generate the solution by $x^*={\tt randn(n,1)}$ and then set $b=Ax^*$. For the inconsistent system, we set $b=Ax^*+{\tt null(A')*randn(m-r,1)}$.

In our test, we use $x^0=0\in\mathbb{R}^n$ or $x^1=x^0=0\in\mathbb{R}^n$ as an initial point.
For each consistent linear system, we run mRK, mDGSG, mRBK, and mBGK (Figures \ref{figue722-1} and \ref{figue723-2}) for several values of momentum parameters $\omega$ and plot the performance of the methods (average after 10 trials) for RSE.
For each inconsistent linear system, similarly, we run mRGS, mRBCD, and mBGLS (Figures \ref{figue722-2} and \ref{figue723-1}) for several values of momentum parameters $\omega$ and plot the performance of the methods (average after 10 trials) for RRE. We set $p=20$ and $s=20$ for mRBK, mBGK, mRBCD, and mBGLS, as they are always good options for sufficient fast convergence of those methods.

For $\omega=0$, the methods now are equivalent with their no-momentum variants. We note that in all of the presented tests the momentum parameters $\omega$ of the methods are chosen to be nonnegative constants that
do not depend on parameters that are not known to the users such as $\sigma^2_{\min}(A)$ and $\sigma^2_{\max}(A)$.
From Figures \ref{figue722-1}, \ref{figue723-2}, \ref{figue722-2}, and \ref{figue723-1}, it is clear that the introducing of momentum term leads
to an improvement in the performance of the methods. More specifically, from the figures we observe the
following:

\begin{itemize}
	\item[(1)] The momentum technique can improve the convergence behavior of the methods. It can be observed that  mRK, mDGSG, mRBK, mBGK, mRGS, mRBCD, and mBGLS,  with appropriately chosen momentum parameters $0 <\omega \leq 0.6$, always converge faster than their no-momentum variants.
	\item[(2)] For consistent linear systems with different values of $\kappa_2(A)$ (well- or ill-conditioned),  and for different types of data (synthetic or real-world), $\omega=0.4$ is always a good option for a sufficient fast  convergence of both mRK and mDSGS; $\omega=0.4$ or $0.5$ is always a good option for both mRBK and mBGK (Figures \ref{figue722-1} and \ref{figue723-2}).
\item[(3)] For inconsistent linear systems with different values of $\kappa_2(A)$ (well- or ill-conditioned),  and for different types of data (synthetic or real-world), $\omega=0.4$ is always a good option for a sufficient fast  convergence of mRGS; $\omega=0.4$ or $0.5$ is always a good option for both mRBCD and mBGLS (Figures \ref{figue722-2}, and \ref{figue723-1}).
\end{itemize}

\begin{figure}[hptb]
	\centering
	\begin{tabular}{cc}
        \includegraphics[width=0.24\linewidth]{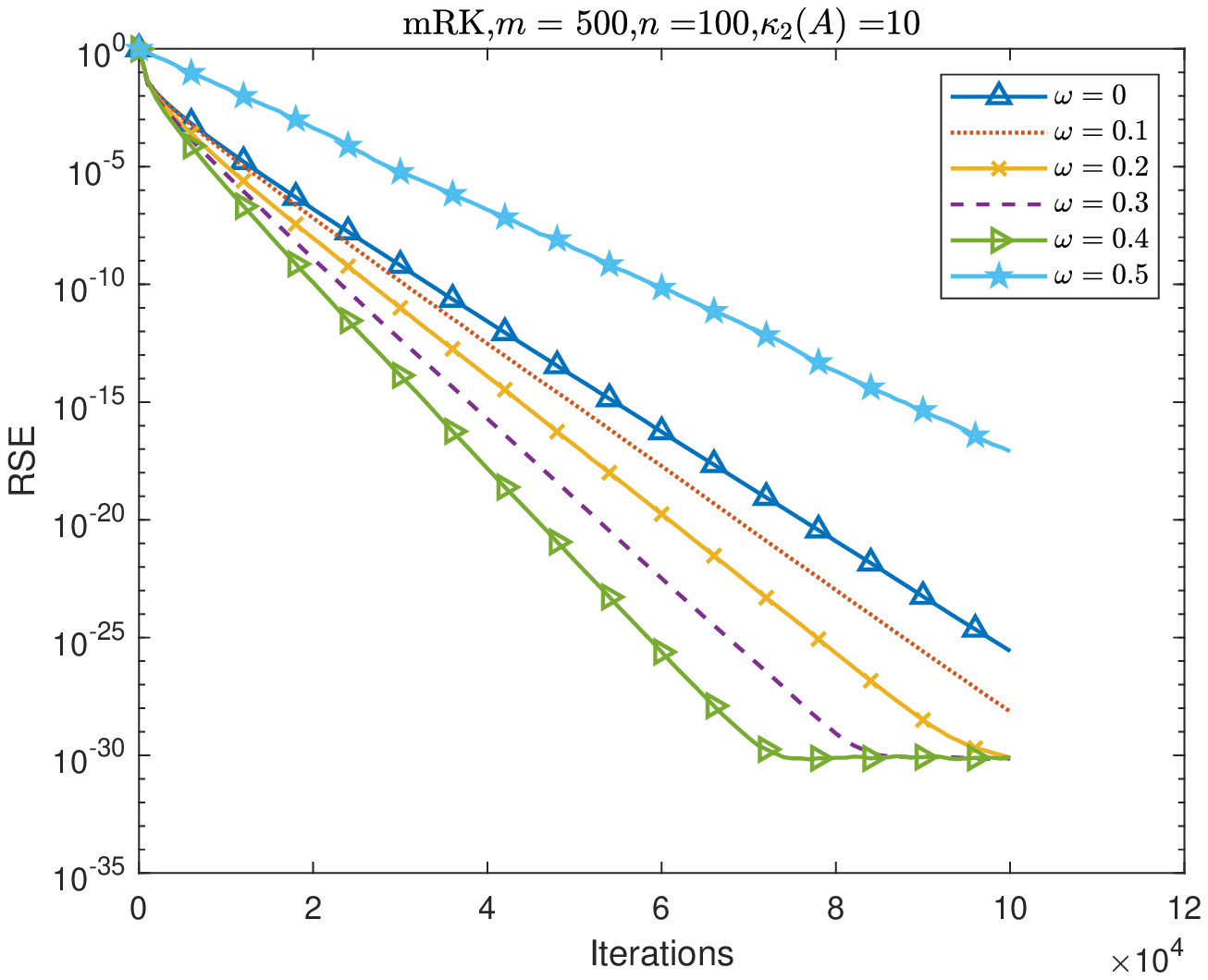}
        \includegraphics[width=0.24\linewidth]{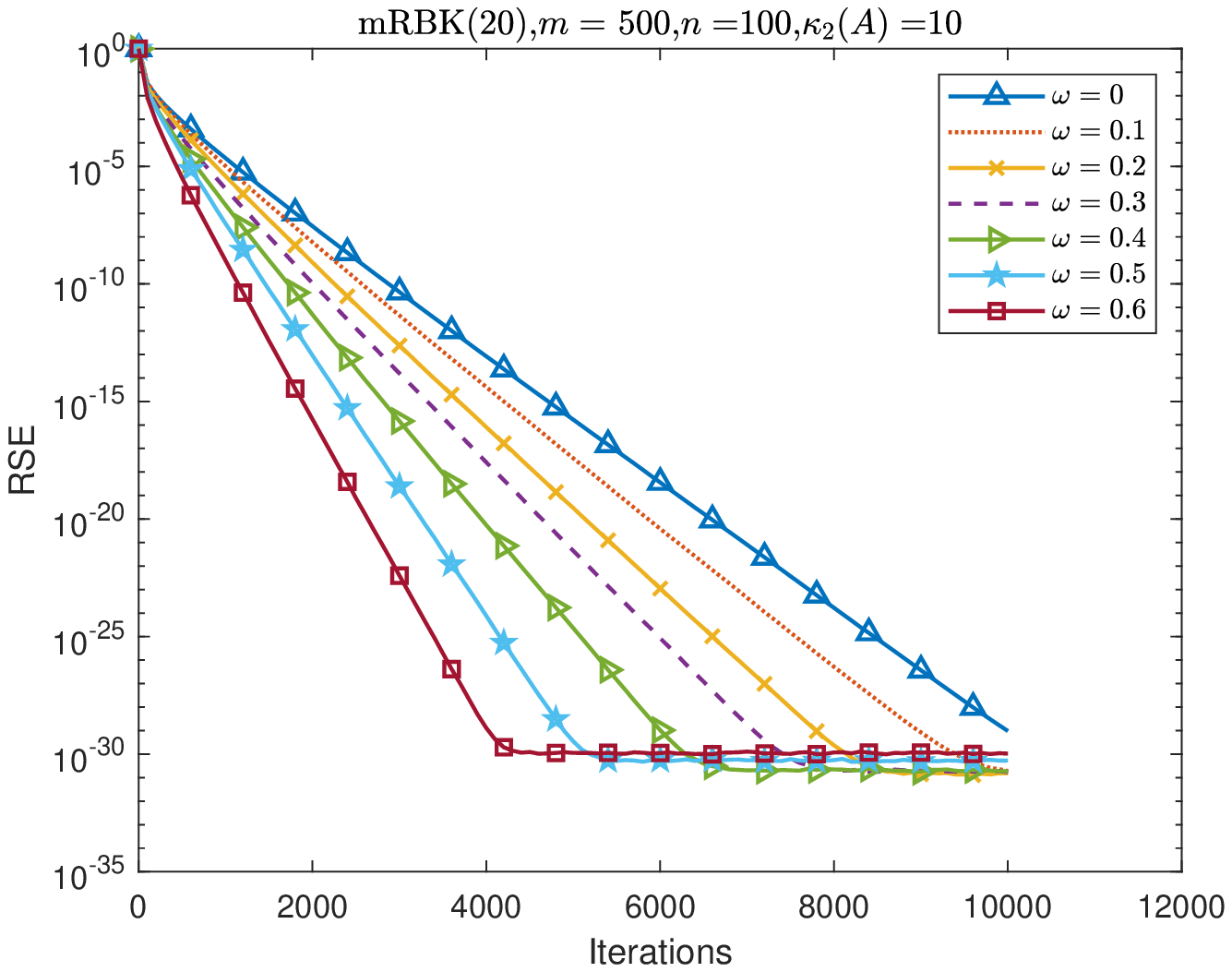}
        \includegraphics[width=0.24\linewidth]{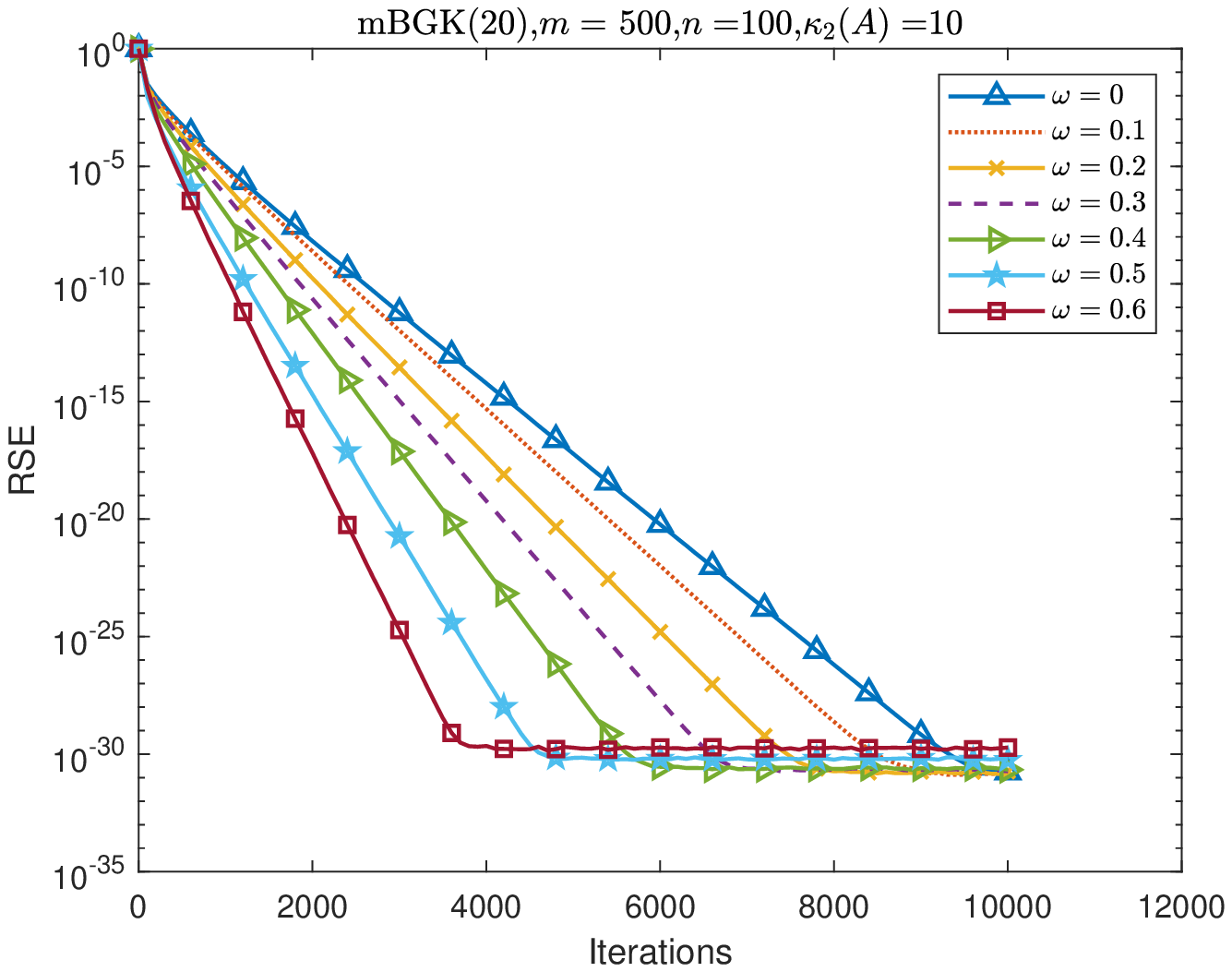}
        \includegraphics[width=0.24\linewidth]{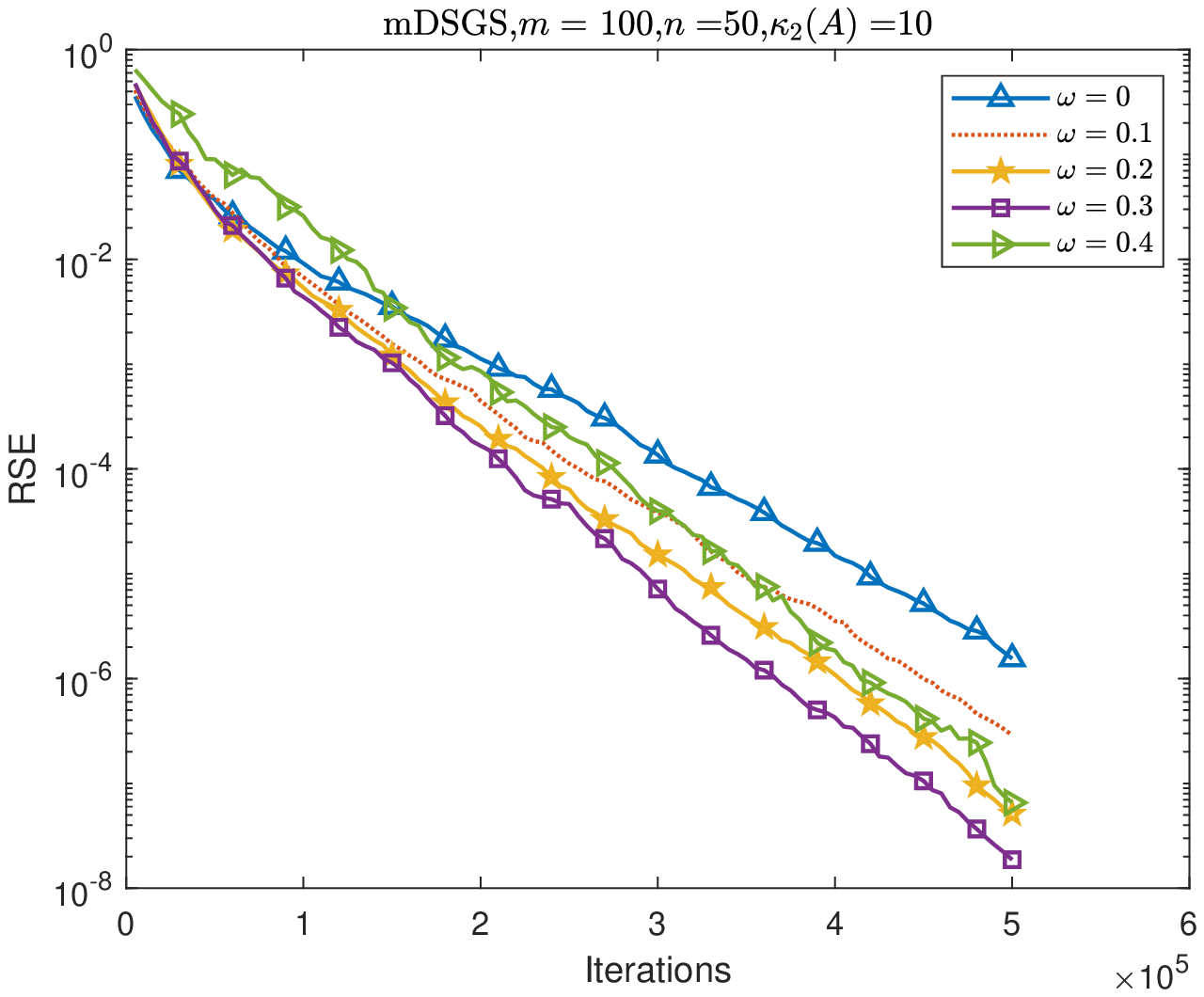}
        \\
        \includegraphics[width=0.24\linewidth]{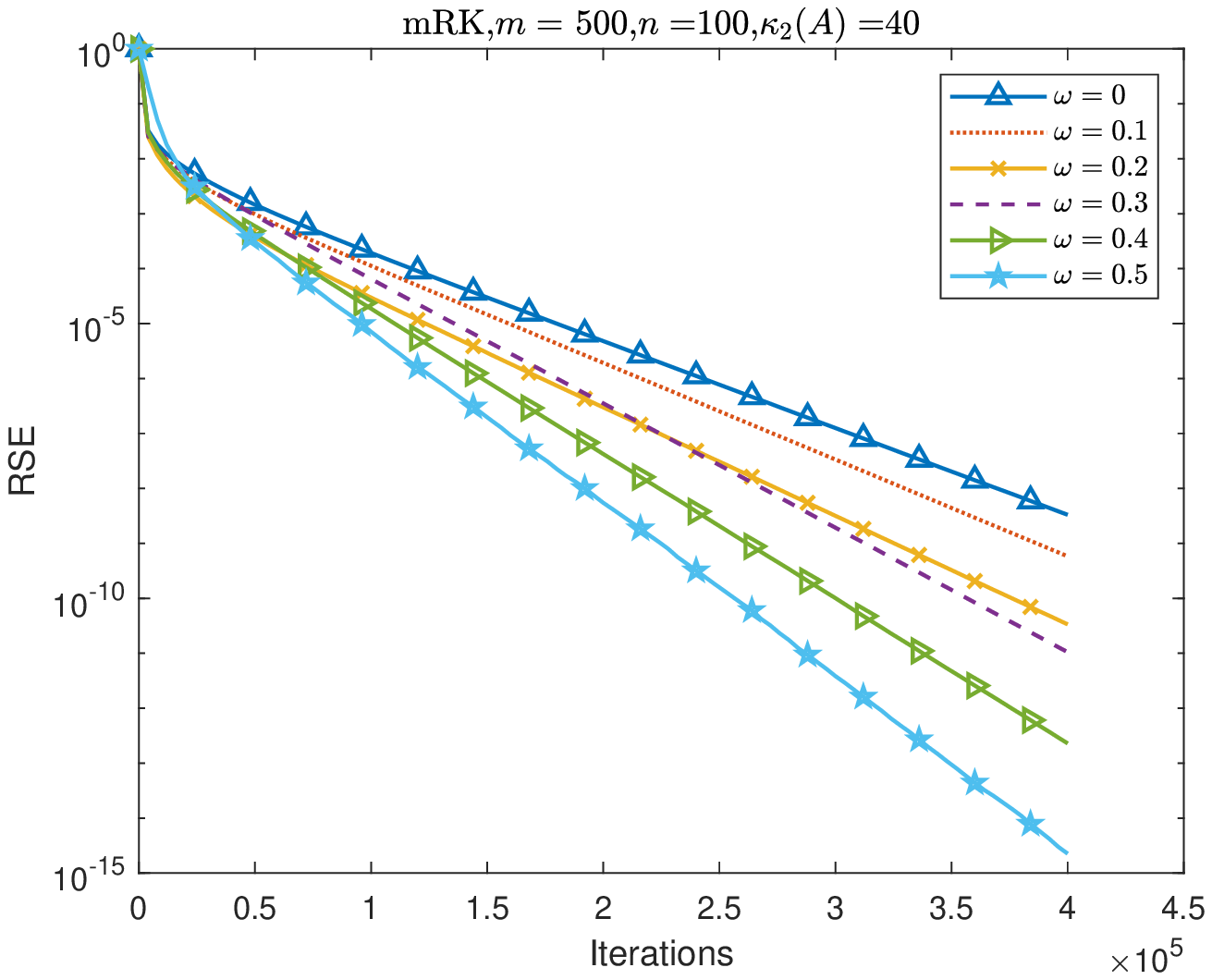}
        \includegraphics[width=0.24\linewidth]{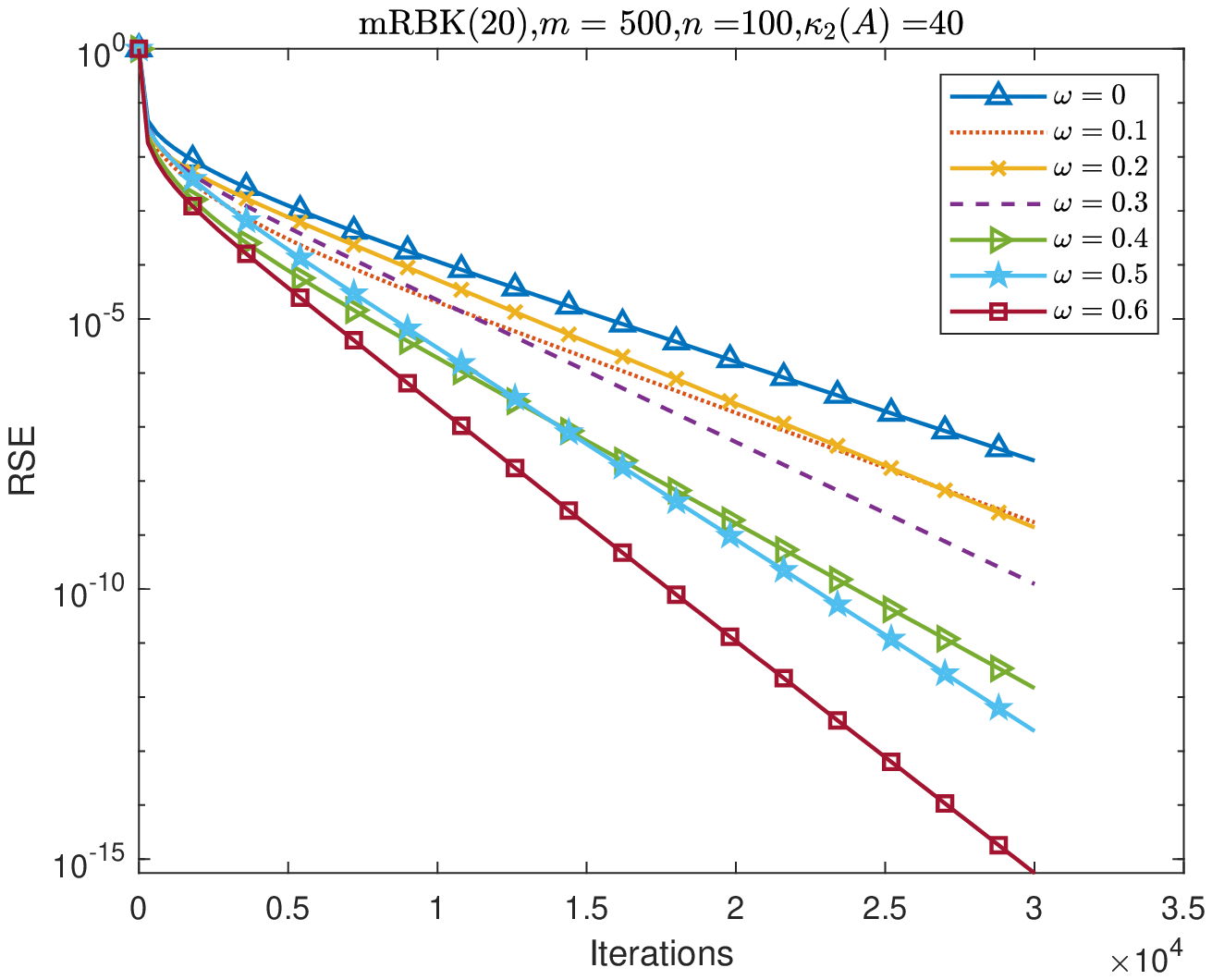}
        \includegraphics[width=0.24\linewidth]{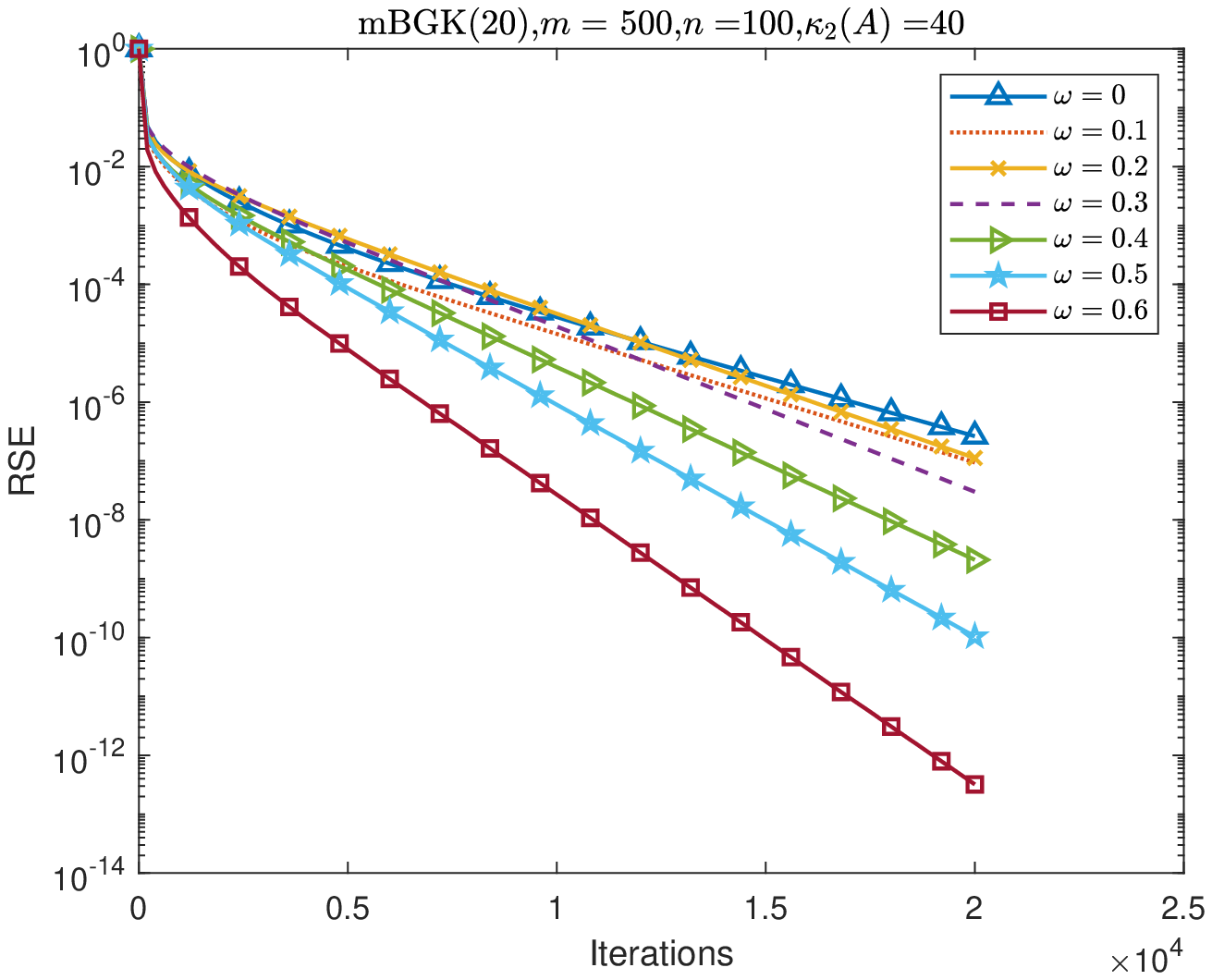}
        \includegraphics[width=0.24\linewidth]{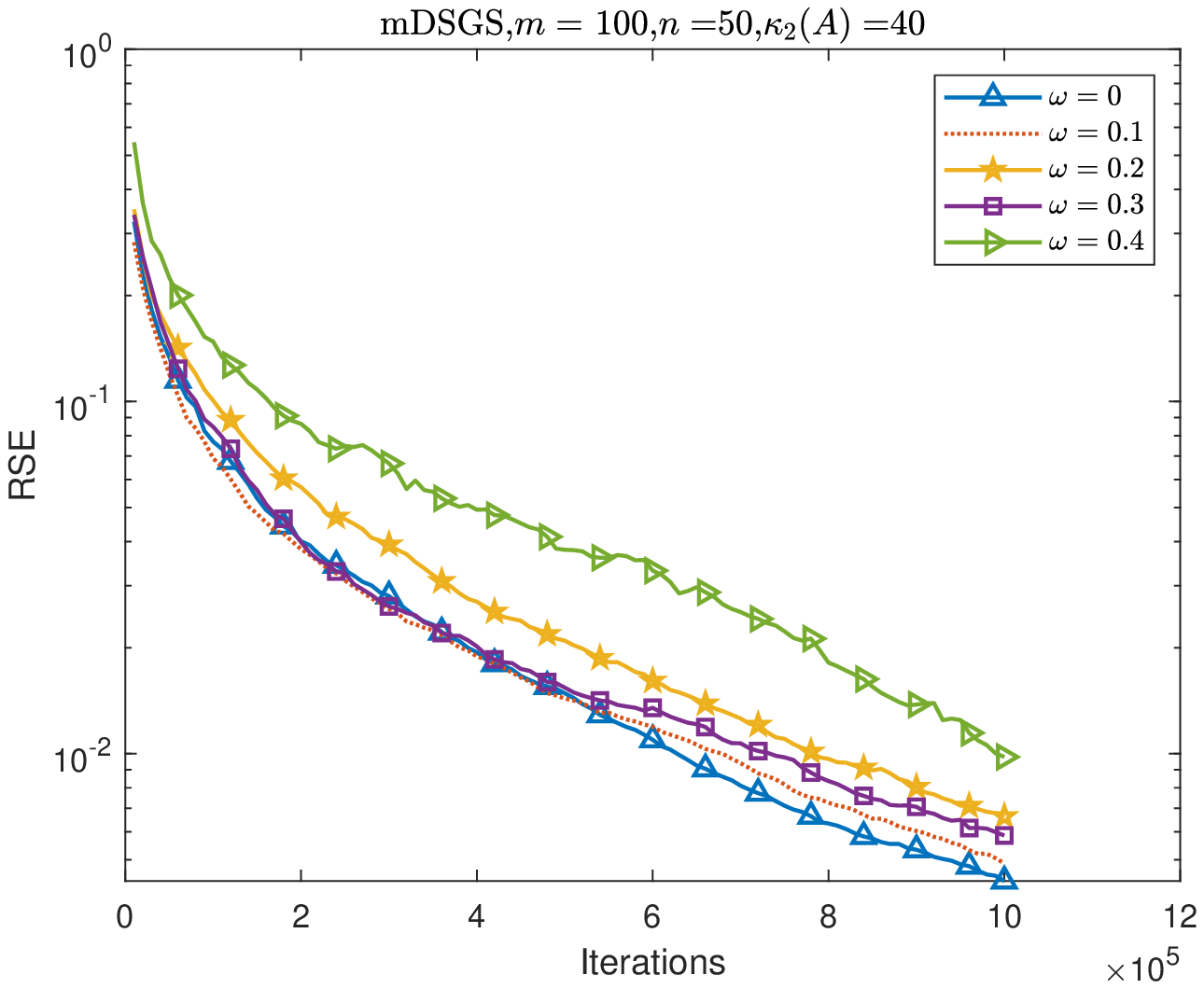}
	\end{tabular}
	\caption{Performance of mRK, mRBK, mBGK, and mDSGS with different momentum parameters $\omega$  for consistent linear systems with Gaussian matrix $A$. All plots are averaged over 10 trials. The title of each plot indicates the test algorithm,  the dimension of the matrix $A$, and the value of condition number $\kappa_2(A)$. We set $p=20$ for both mRBK and mBGK. }
	\label{figue722-1}
\end{figure}

\begin{figure}[hptb]
	\centering
	\begin{tabular}{cc}
        \includegraphics[width=0.24\linewidth]{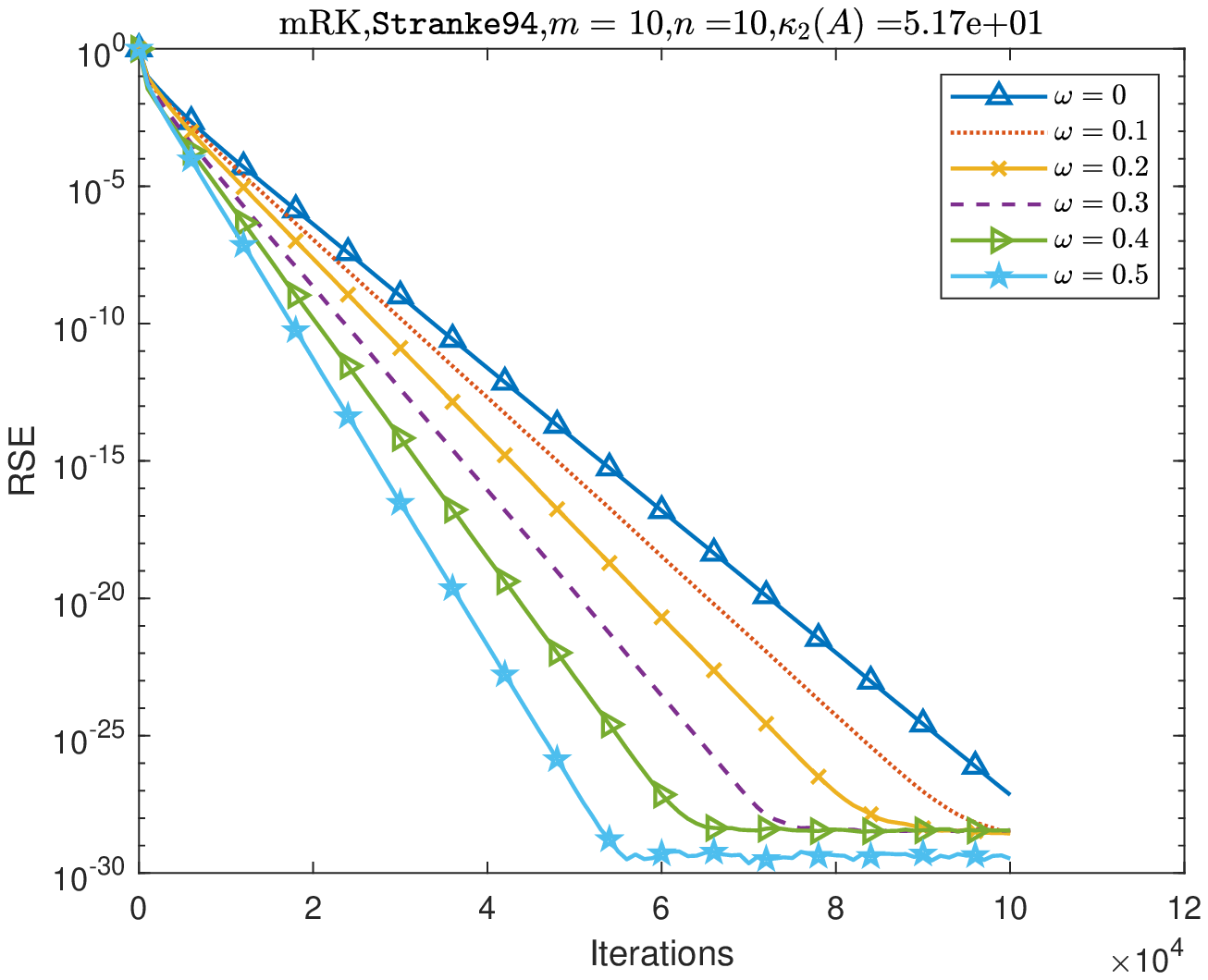}
        \includegraphics[width=0.24\linewidth]{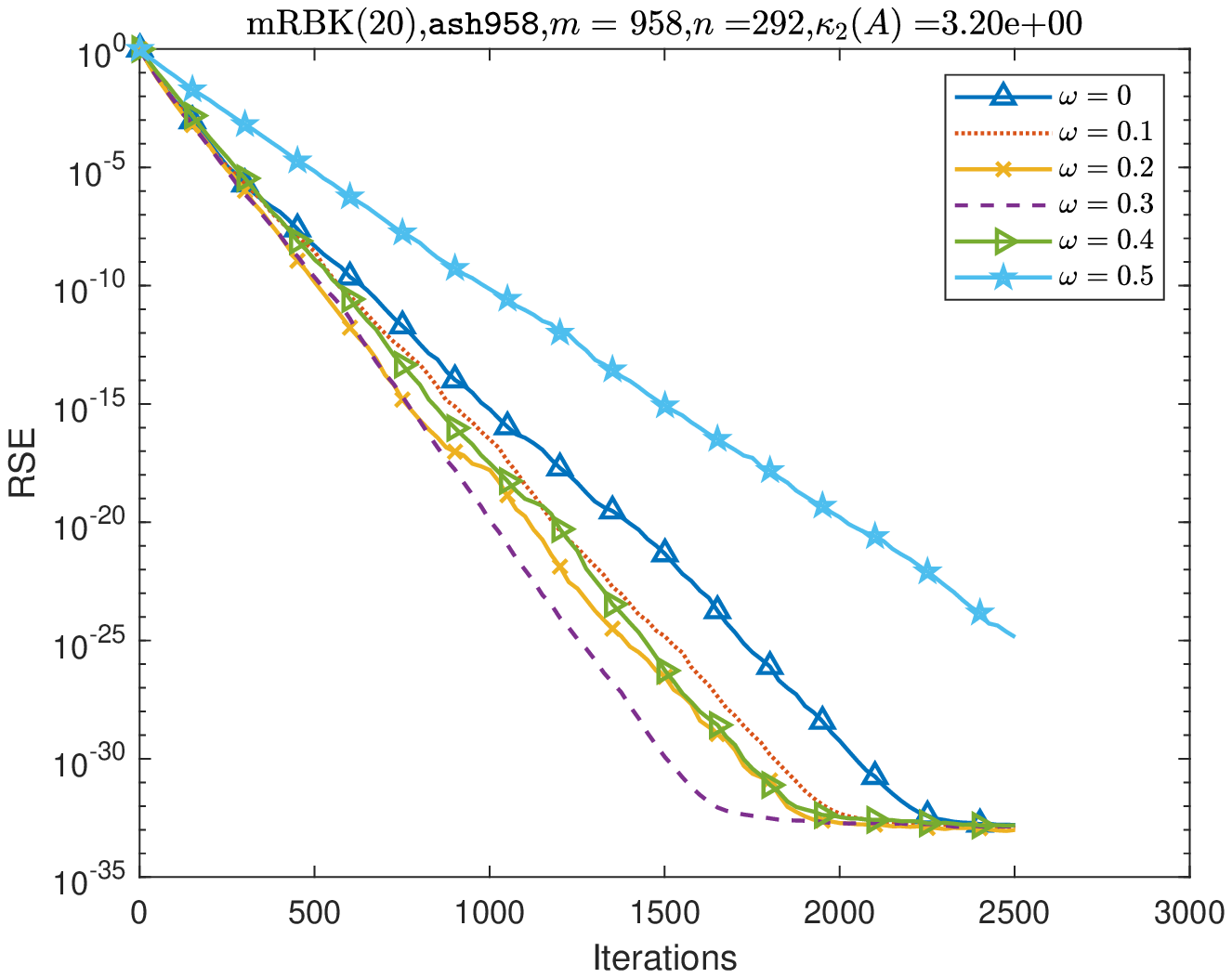}
        \includegraphics[width=0.24\linewidth]{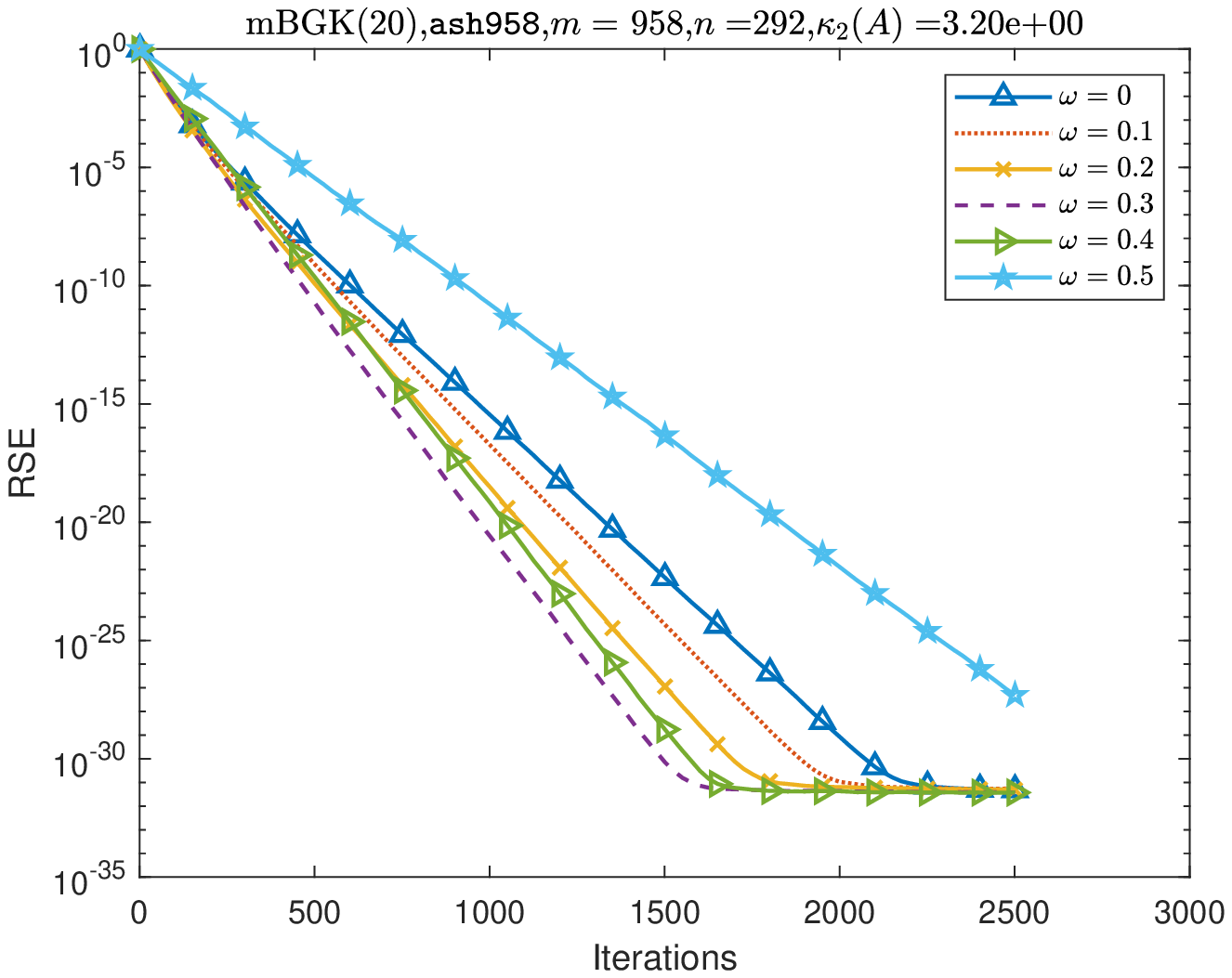}
        \includegraphics[width=0.24\linewidth]{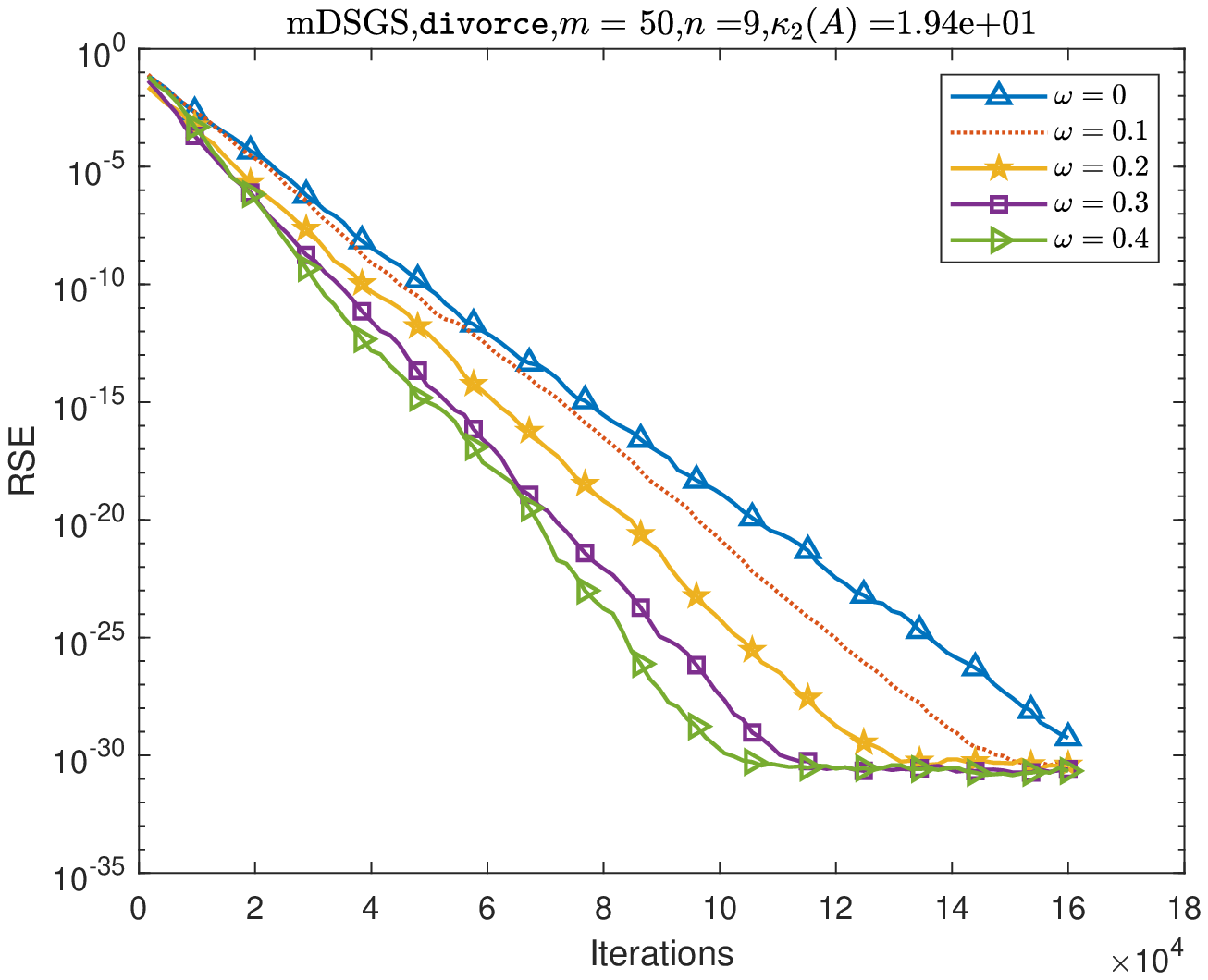}
        \\
        \includegraphics[width=0.24\linewidth]{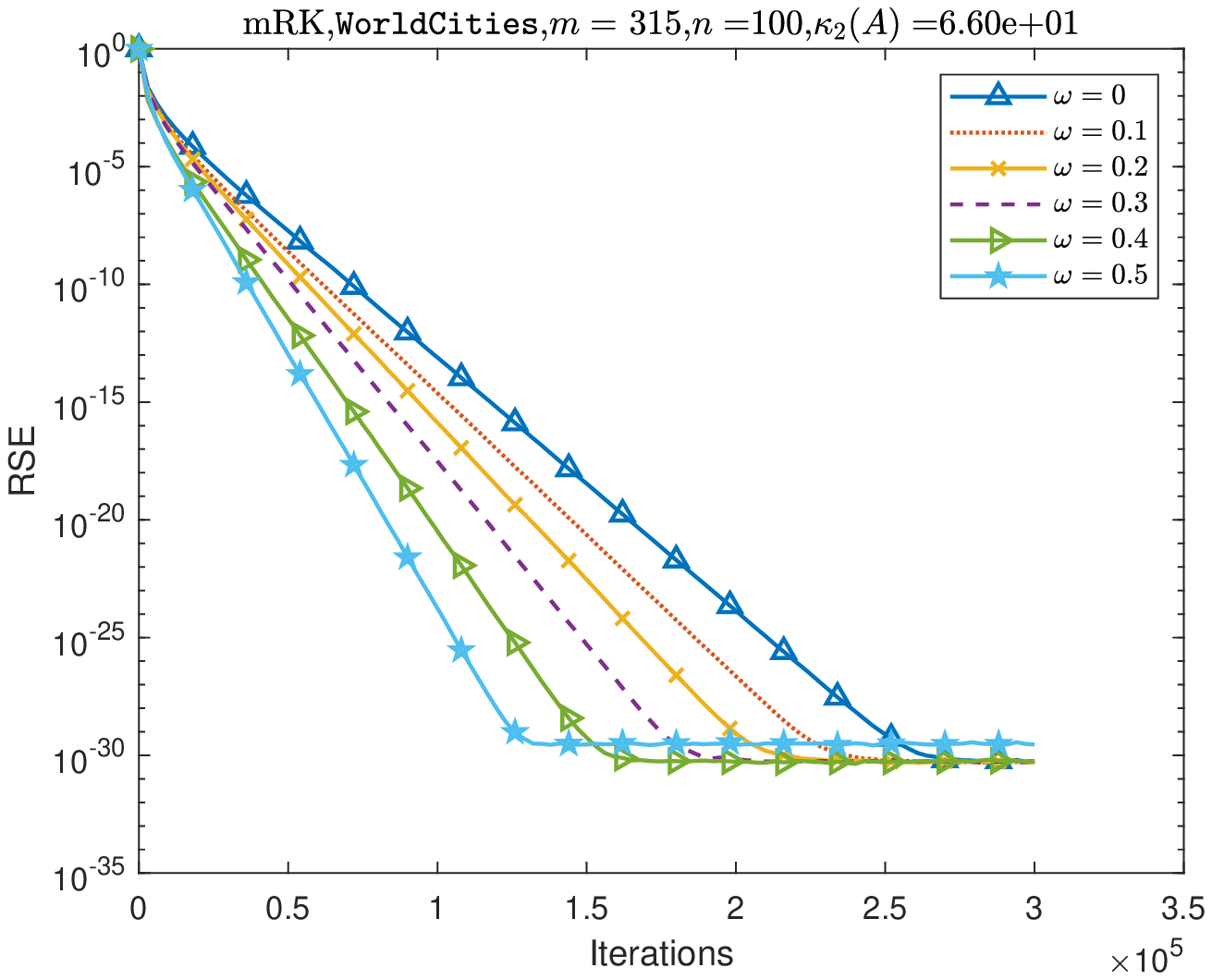}
        \includegraphics[width=0.24\linewidth]{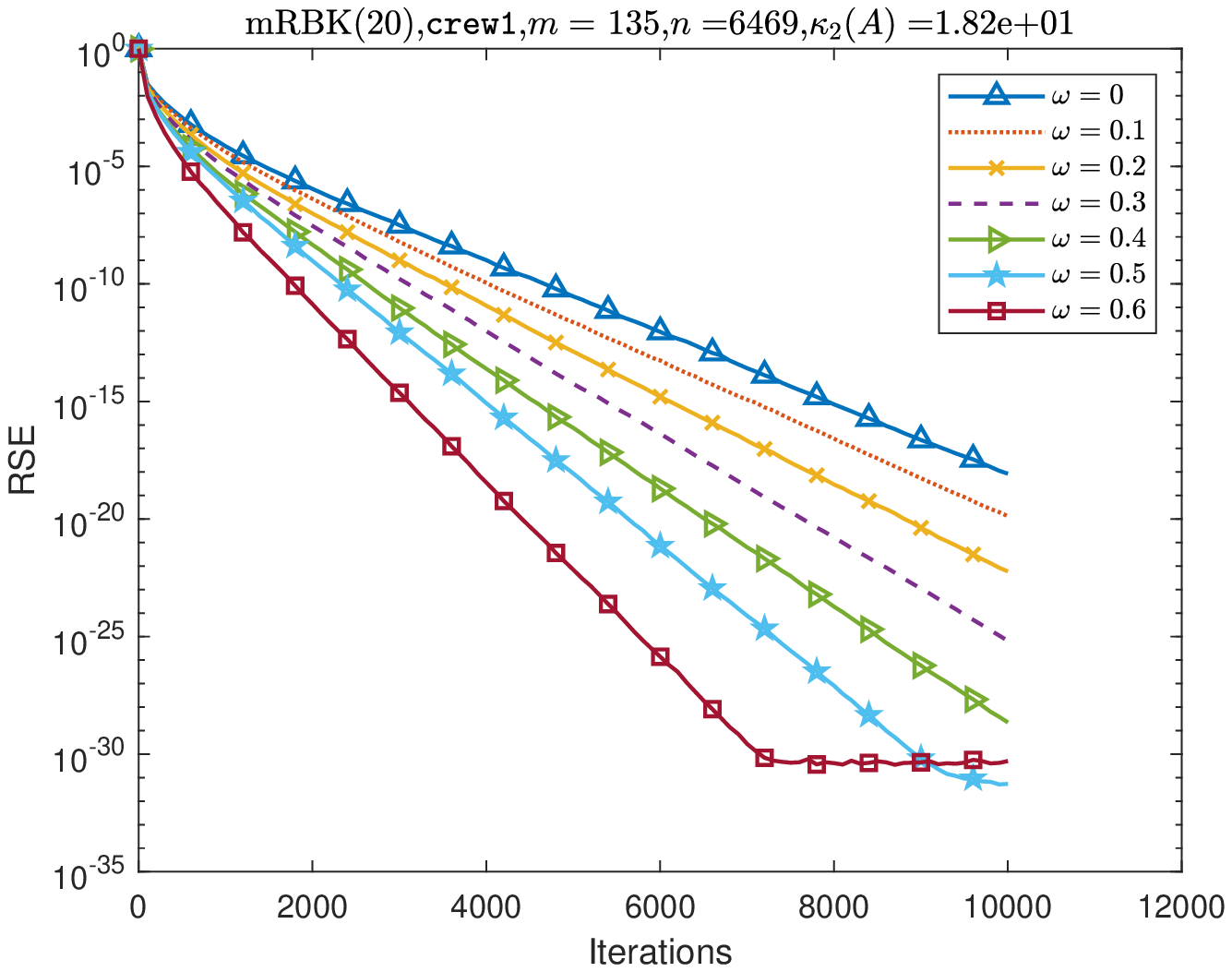}
        \includegraphics[width=0.24\linewidth]{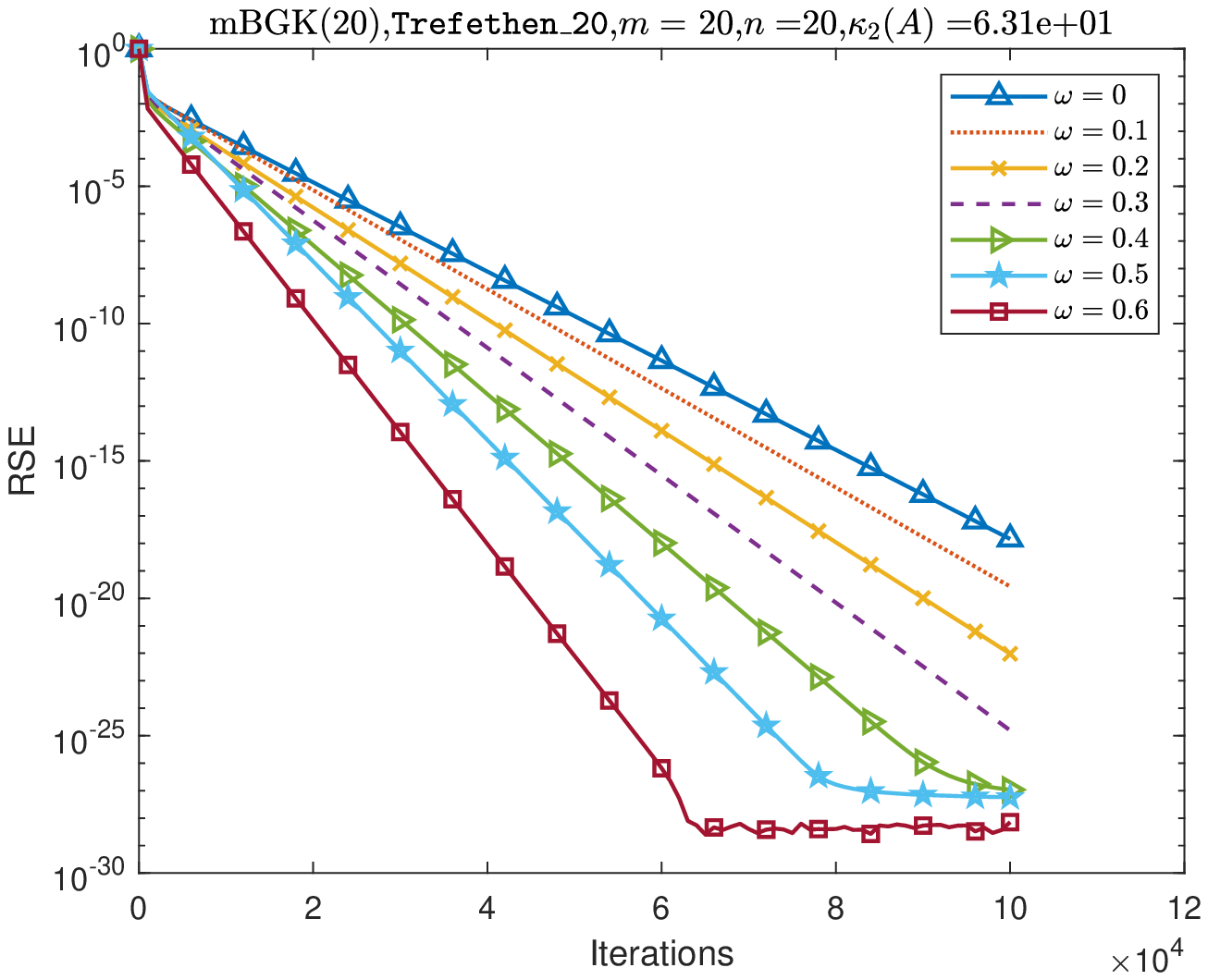}
        \includegraphics[width=0.24\linewidth]{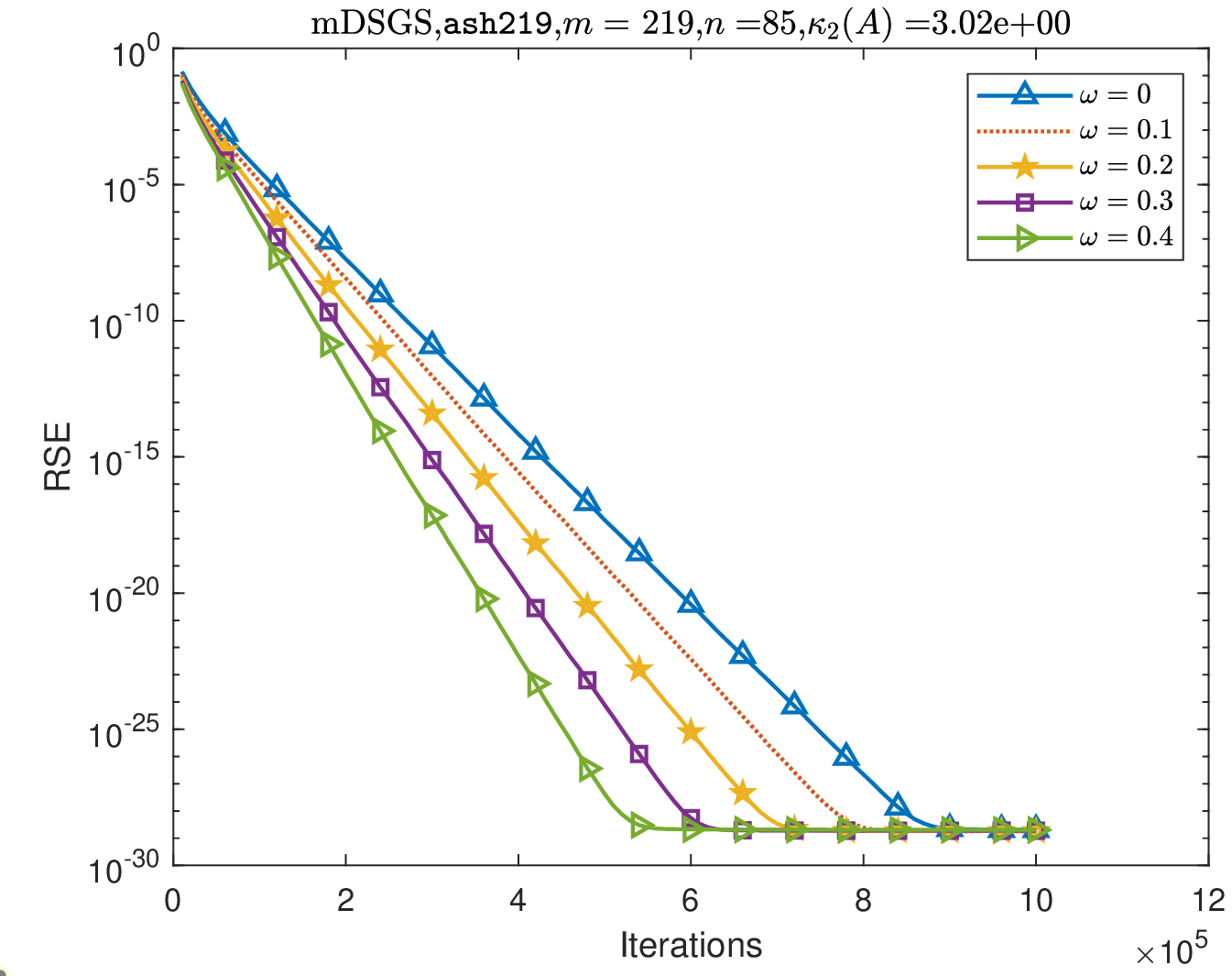}
	\end{tabular}
	\caption{Performance of mRK, mRBK, mBGK, and mDSGS with different momentum parameters $\omega$  for consistent linear systems. The coefficient matrices are from  SuiteSparse Matrix Collection \cite{Kol19}. All plots are averaged over 10 trials. The title of each plot indicates the test algorithm,  the dimension of the matrix $A$, and the value of condition number $\kappa_2(A)$. We set $p=20$ for both mRBK and mBGK.}
	\label{figue723-2}
\end{figure}

\begin{figure}[hptb]
	\centering
	\begin{tabular}{cc}
        \includegraphics[width=0.33\linewidth]{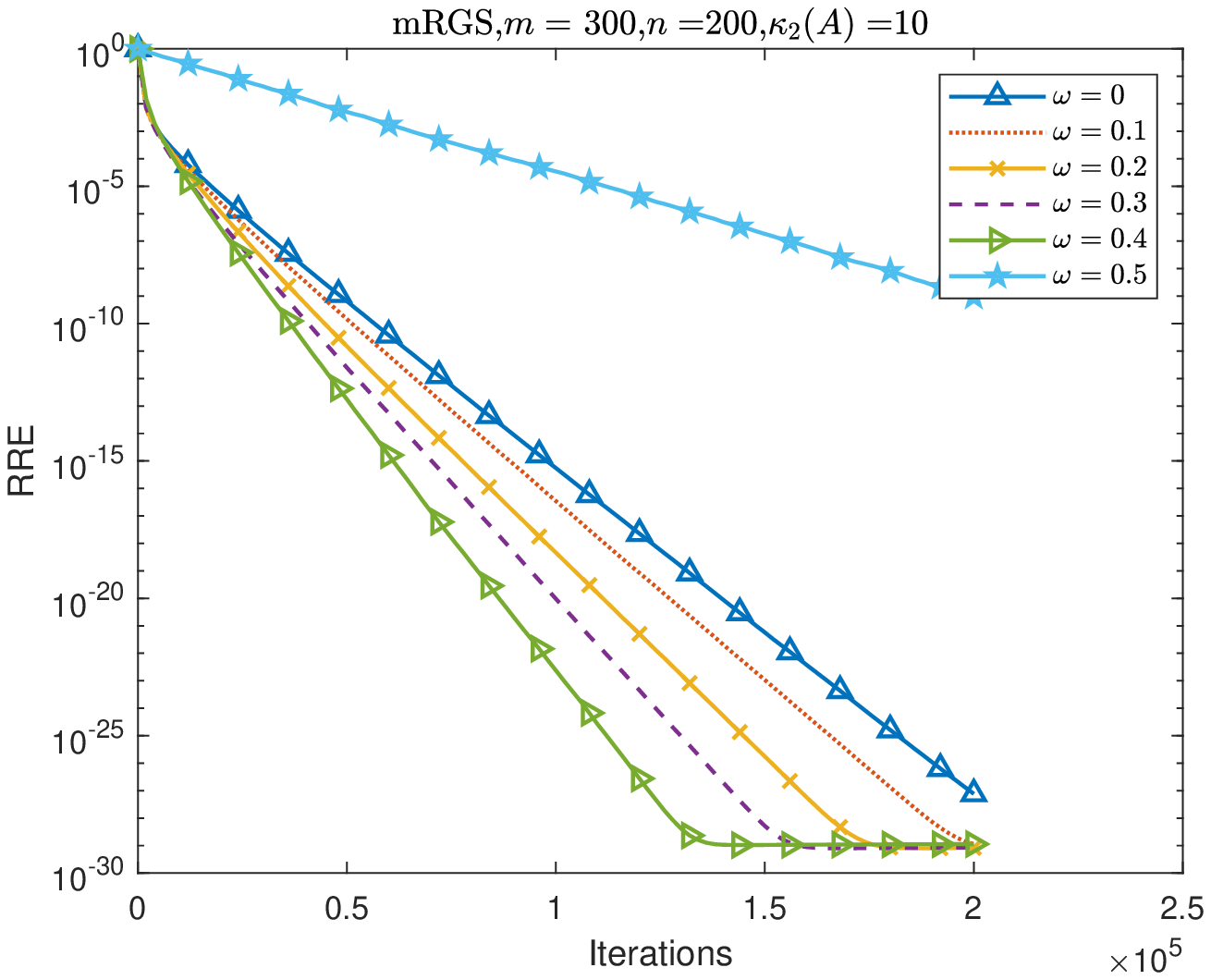}
        \includegraphics[width=0.33\linewidth]{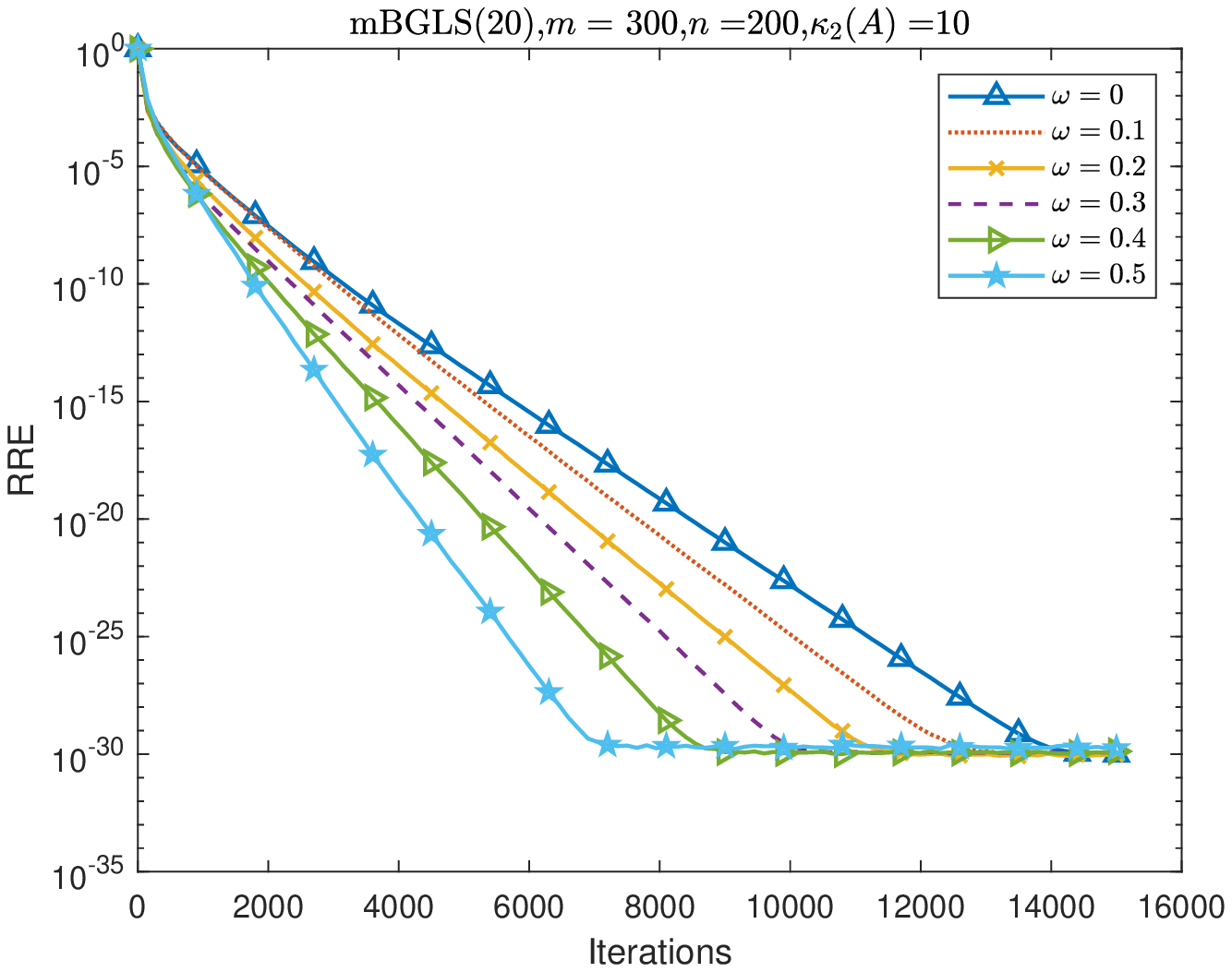}
        \includegraphics[width=0.33\linewidth]{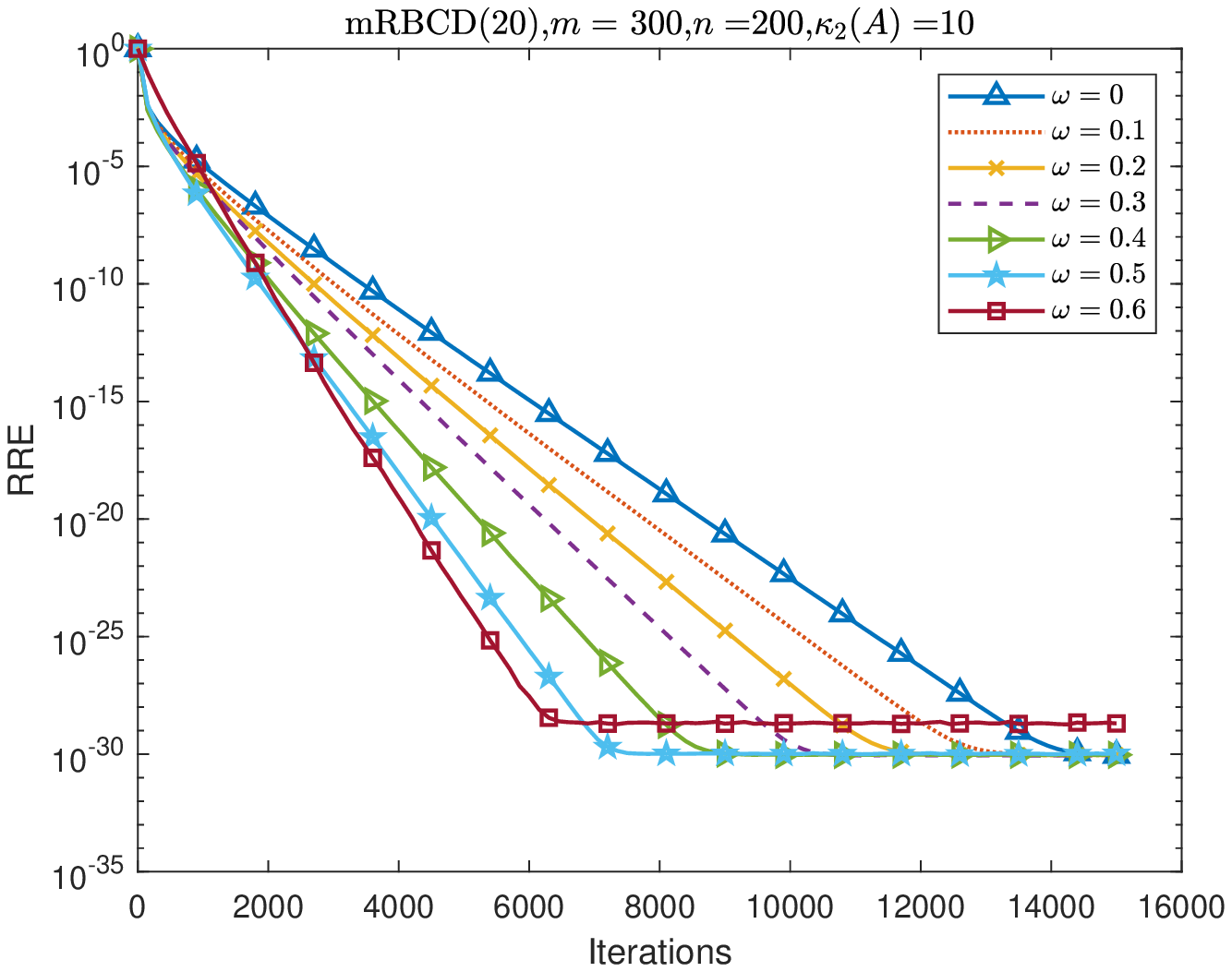}\\
        \includegraphics[width=0.33\linewidth]{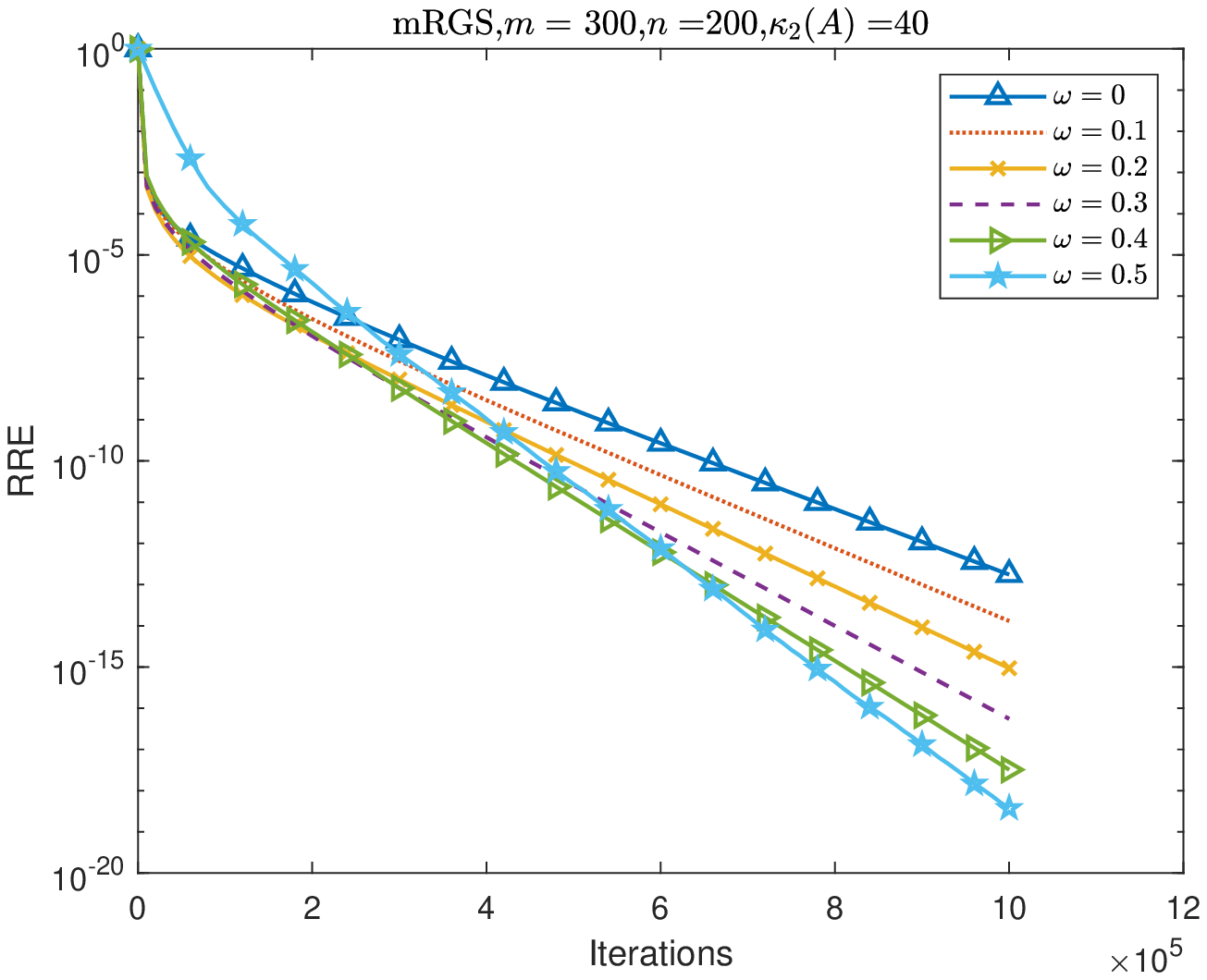}
        \includegraphics[width=0.33\linewidth]{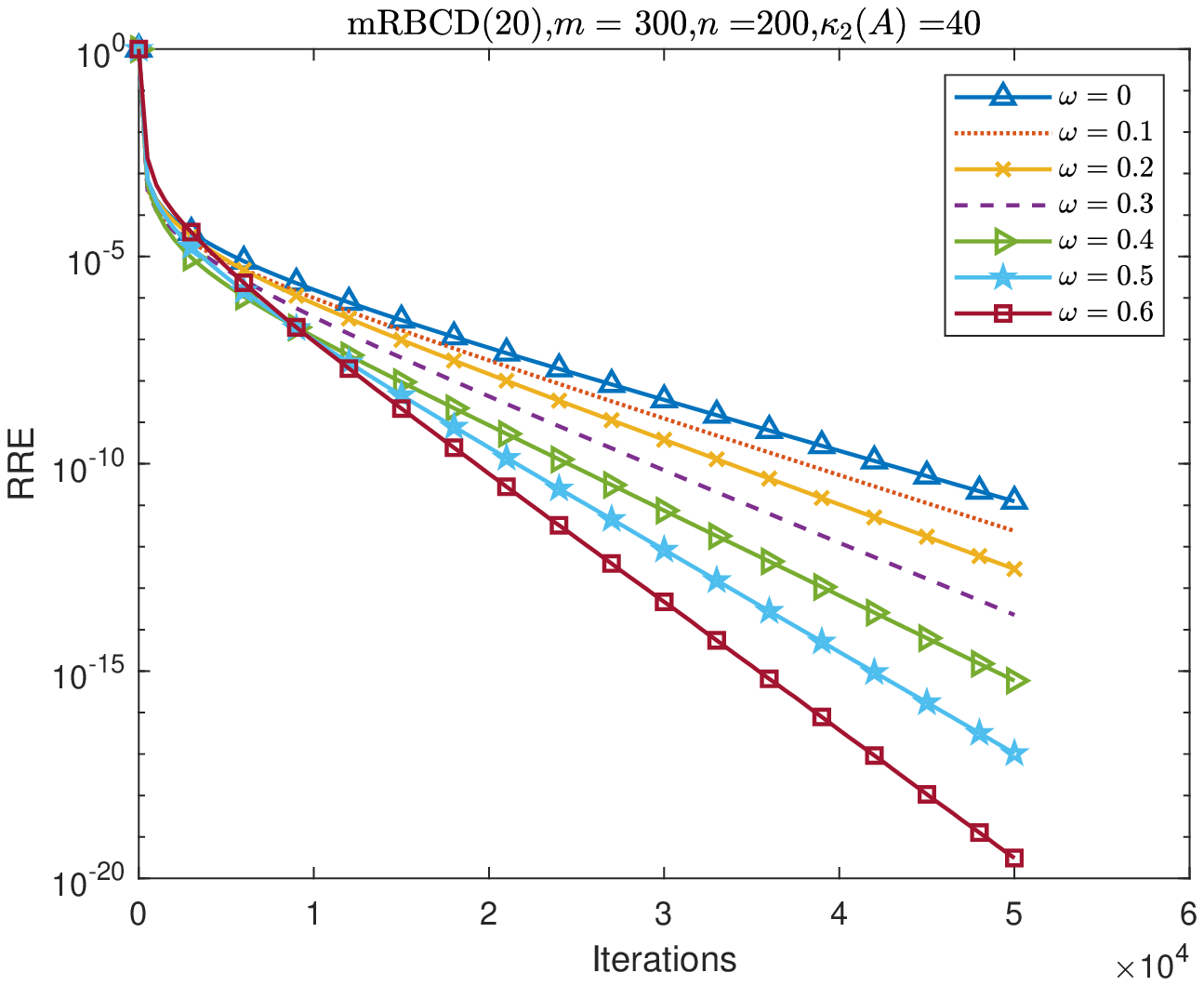}
        \includegraphics[width=0.33\linewidth]{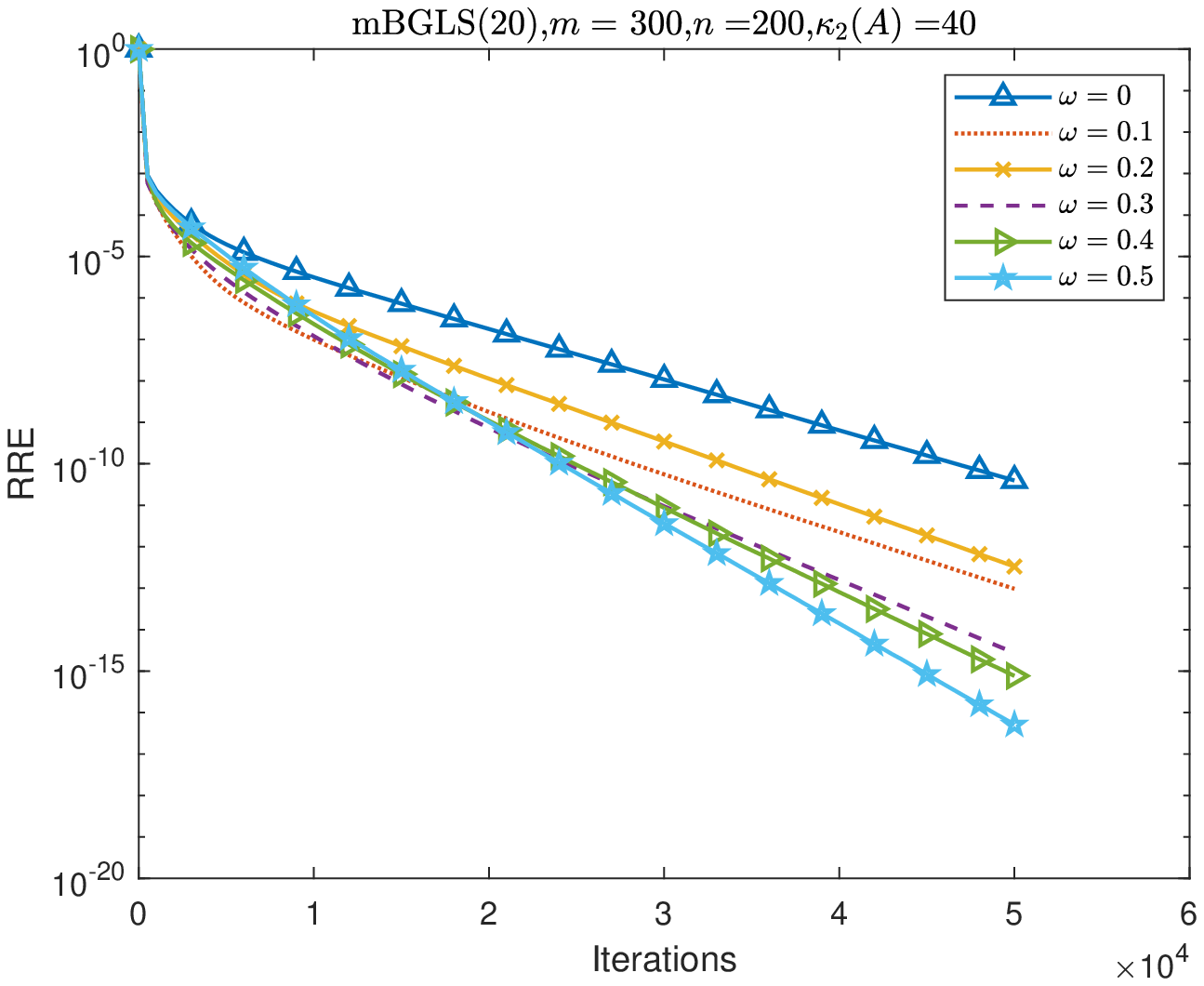}
	\end{tabular}
	\caption{Performance of mRGS, mRBCD, and mBGLS with different momentum parameters $\omega$  for inconsistent linear systems with Gaussian matrix $A$. All plots are averaged over 10 trials. The title of each plot indicates the test algorithm,  the dimension of the matrix $A$, and the value of condition number $\kappa_2(A)$. We set $s=20$ for both mRBCD and mBGLS.}
	\label{figue722-2}
\end{figure}

\begin{figure}[hptb]
	\centering
	\begin{tabular}{cc}
        \includegraphics[width=0.33\linewidth]{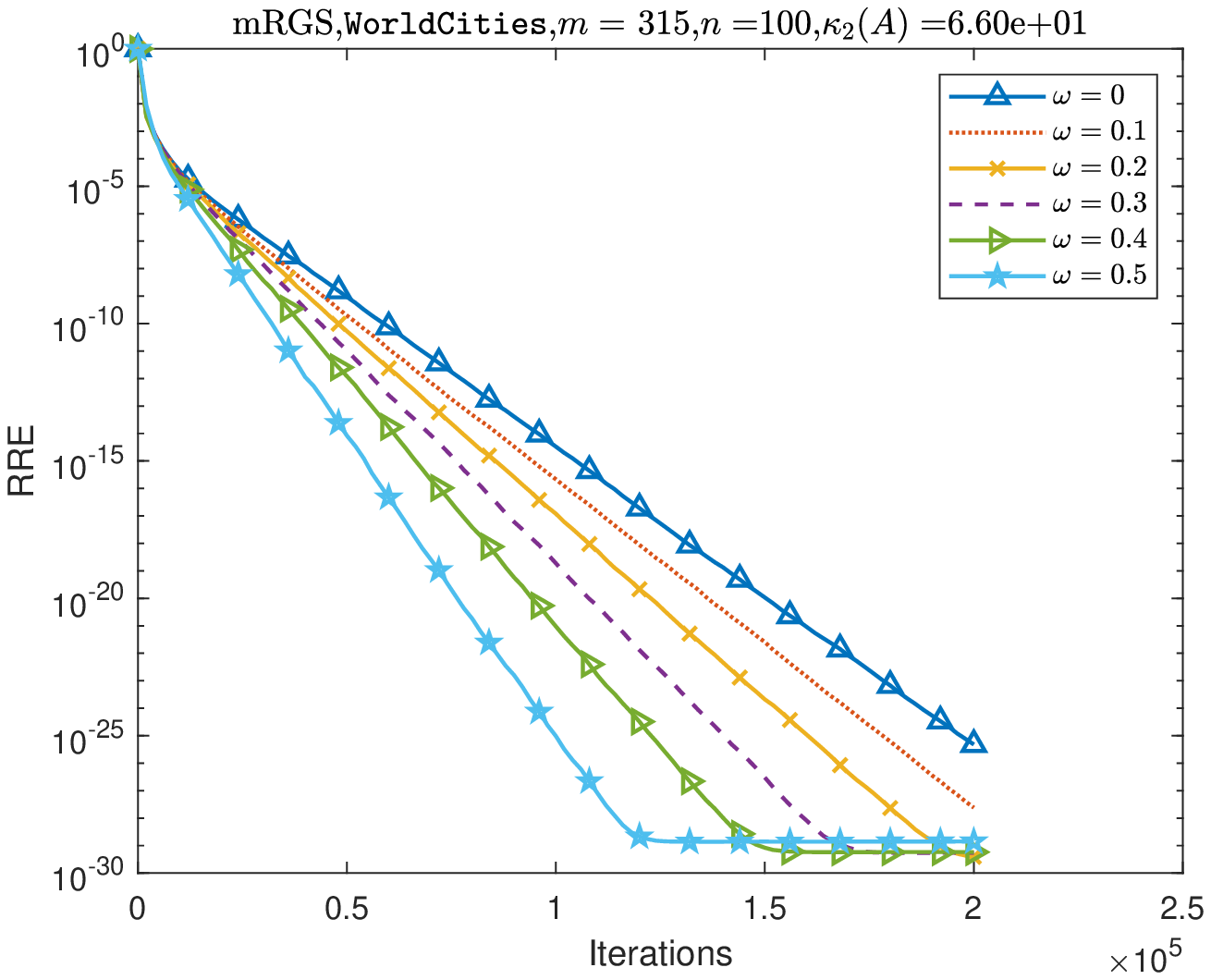}
        \includegraphics[width=0.33\linewidth]{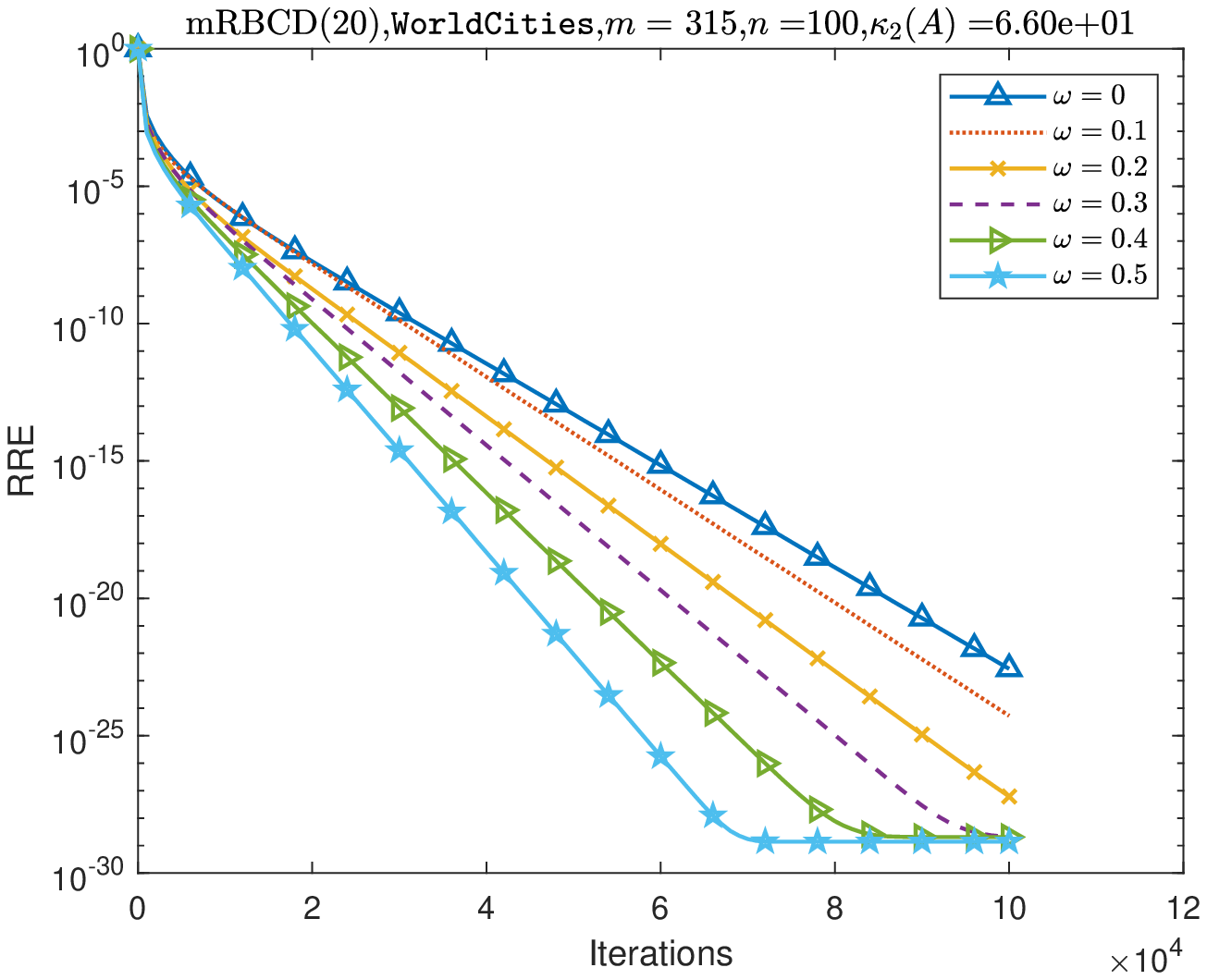}
        \includegraphics[width=0.33\linewidth]{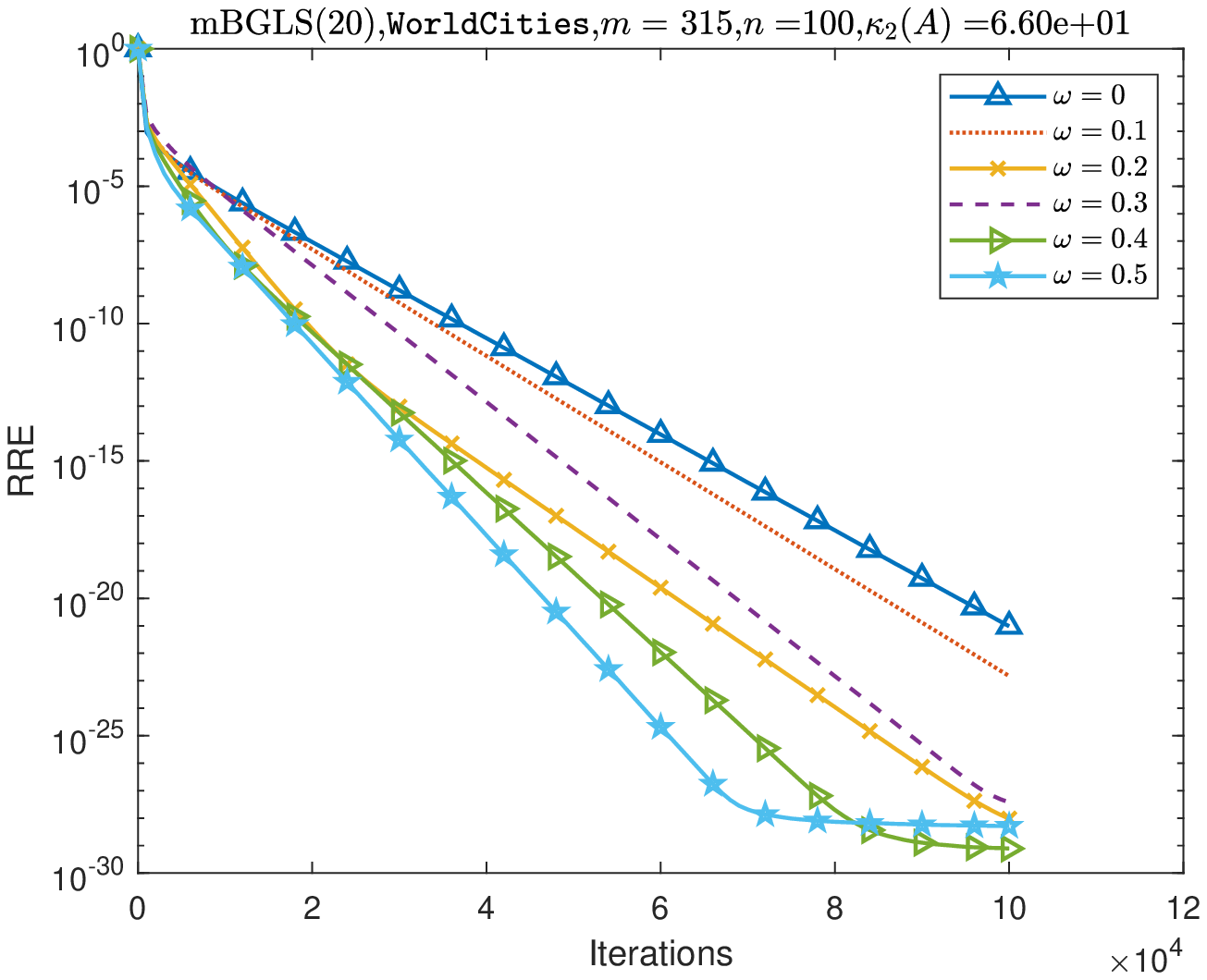}
        \\
        \includegraphics[width=0.33\linewidth]{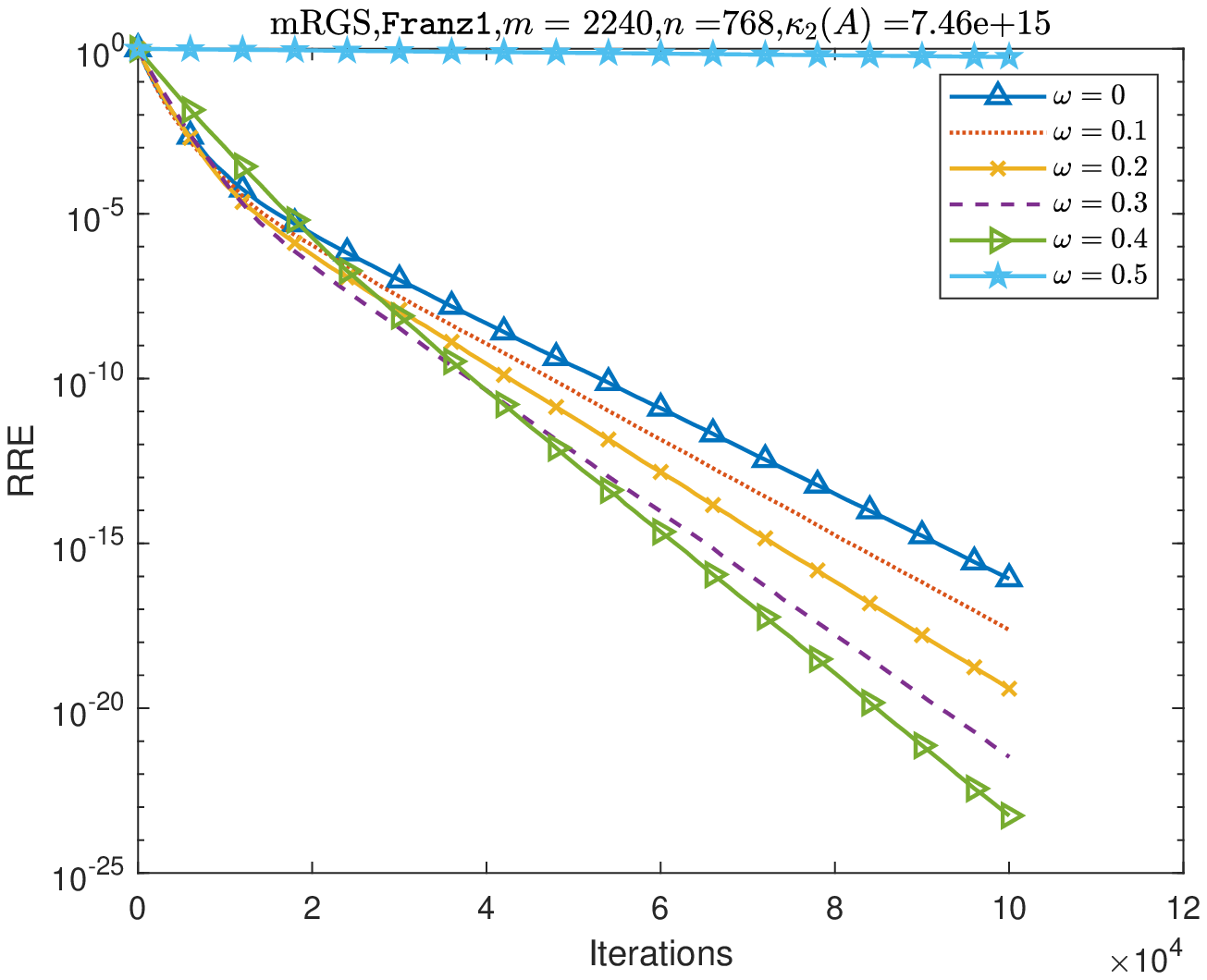}
        \includegraphics[width=0.33\linewidth]{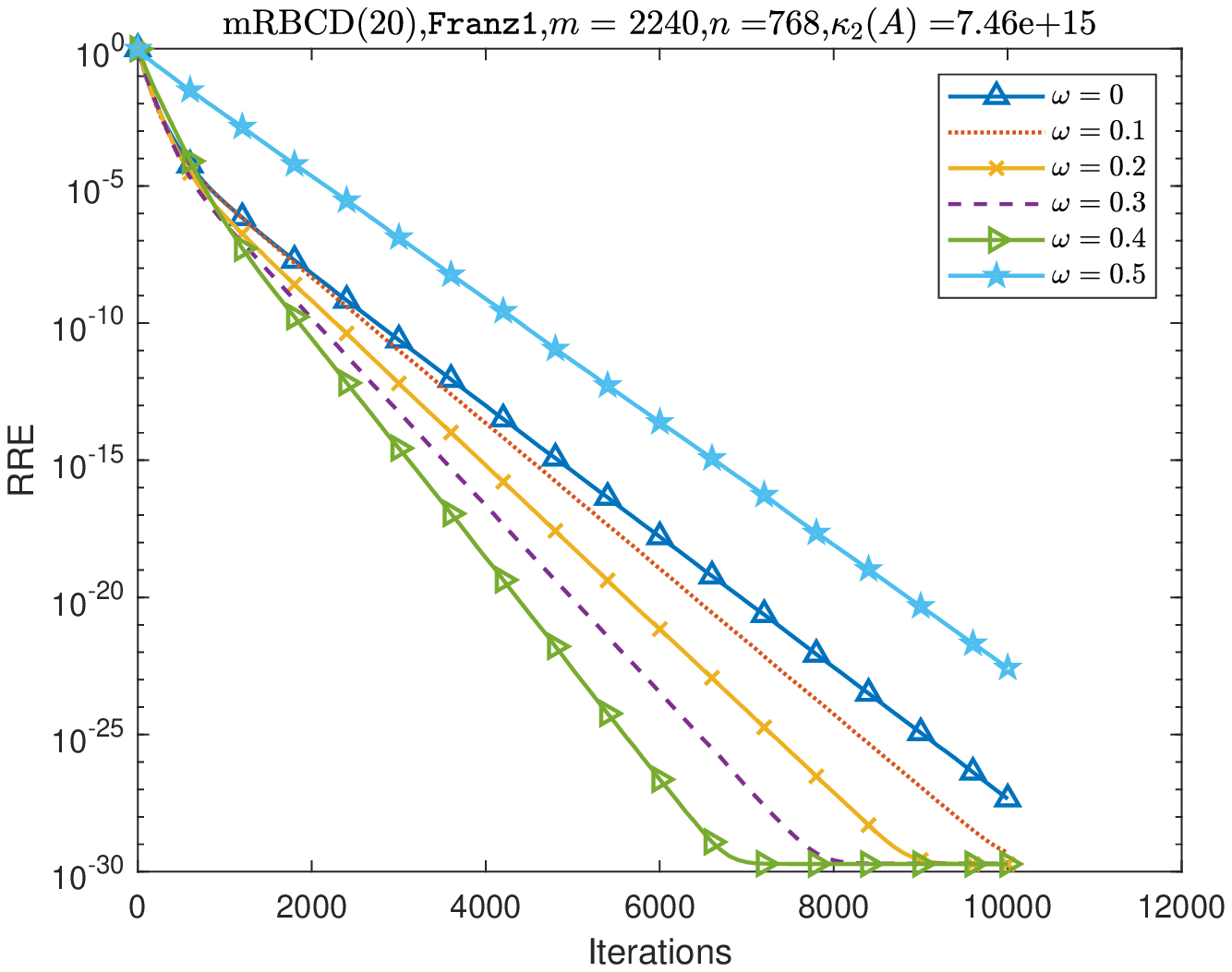}
        \includegraphics[width=0.33\linewidth]{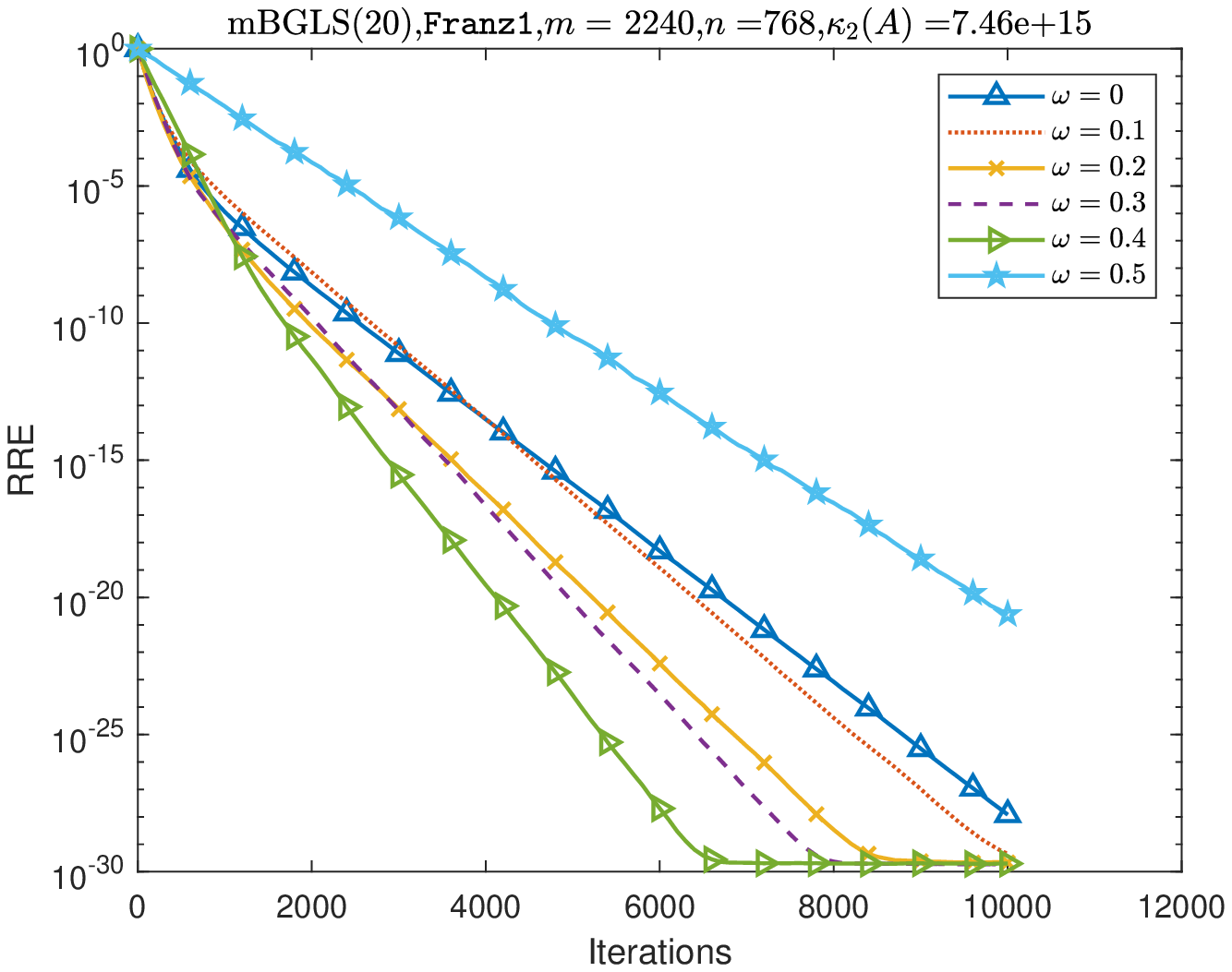}
	\end{tabular}
	\caption{Performance of mRGS, mRBCD, and mBGLS with different momentum parameters $\omega$  for inconsistent linear systems. The coefficient matrices are from  SuiteSparse Matrix Collection \cite{Kol19}. All plots are averaged over 10 trials. The title of each plot indicates the test algorithm,  the dimension of the matrix $A$, and the value of condition number $\kappa_2(A)$. We set $s=20$ for both mRBCD and mBGLS.}
	\label{figue723-1}
\end{figure}

\subsection{Comparison among the proposed methods}
In this subsection, we will compare mRK, mRBK, and mBGK for solving the consistent linear systems and compare mRGS, mRBCD, and mBGLS for solving the inconsistent linear systems.
 In comparing the methods, we use $x^0=0\in\mathbb{R}^n$ or $x^1=x^0=0\in\mathbb{R}^n$ as an initial point. The computations are terminated once RSE or RRE is less than $10^{-12}$.
For the horizontal axis, we use wall-clock time measured using the {\tt tic-toc} {\sc Matlab}
function.

First we compare  mRK, mRBK, and mBGK on synthetic linear systems generated with the {\sc Matlab}
functions {\tt randn} and {\tt sprandn}; see Figure \ref{figue724-1}.
The poor performance of mBGK on the dense problem generated using {\tt randn} can be found in Figure \ref{figue724-1} (a).
This is because that the Gaussian method almost always requires more flops to reach a solution as it requires the expensive matrix-vector product. We can observe from  Figure \ref{figue724-1}  (a) that mRBK performs better than mRK and mBGK on the dense problem.
In Figure \ref{figue724-1}  (b), we compare the methods on a sparse linear system generated
using the {\sc Matlab} sparse random matrix function {\tt sprandn(m,n,density,rc)},
where {\tt density} is the percentage of nonzero entries and {\tt rc} is the reciprocal of the
condition number. On this sparse problem, mBGK is more efficient than mRK and mRBK.

In Figure \ref{figue724-3}, we compare mRGS, mRBCD, and mBGLS on inconsistent linear system.
The high iteration cost of mBGLS resulted in poor performance on the dense problem generated using {\tt rand} can be found in Figure \ref{figue724-3}  (a). On this dense problem, mRBCD is more efficient. From Figure \ref{figue724-3}  (b), we can see that mBGLS performs better than mRGS and mRBCD for solving sparse problems. We note that since {\sc Matlab} performs automatic multithreading when calculating matrix-vector products, which was the bottleneck cost in the Gaussian sampling based methods.
Hence despite the higher iteration cost of the Gaussian sampling based methods, their performance, in
terms of the wall-clock time, is comparable to that of other methods when the coefficient matrix is sparse.

\begin{figure}[hptb]
 \centering
 \begin{tabular}{cc}
 \subfigure[]{\includegraphics[width=0.48\linewidth]{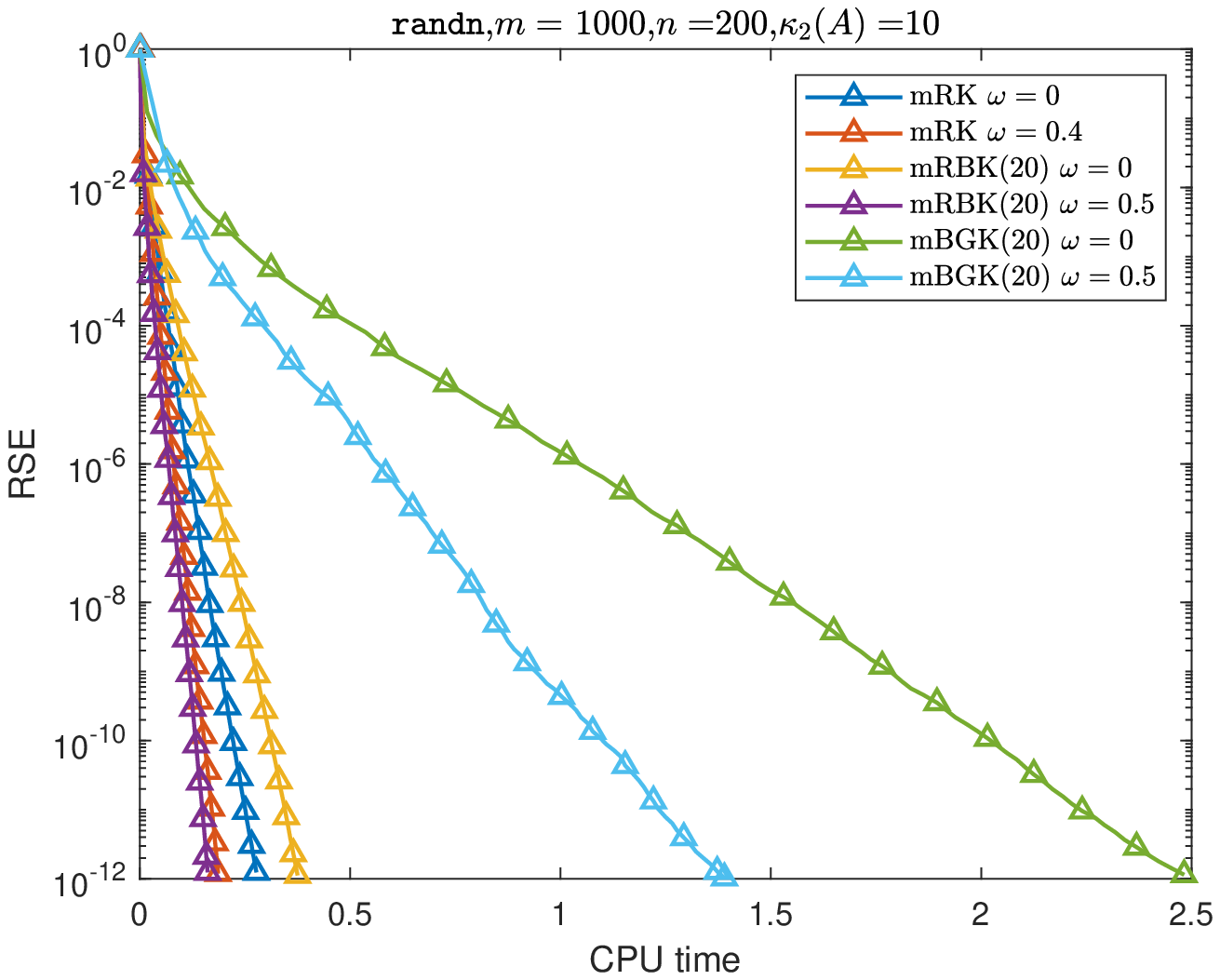}}
 \subfigure[]{\includegraphics[width=0.48\linewidth]{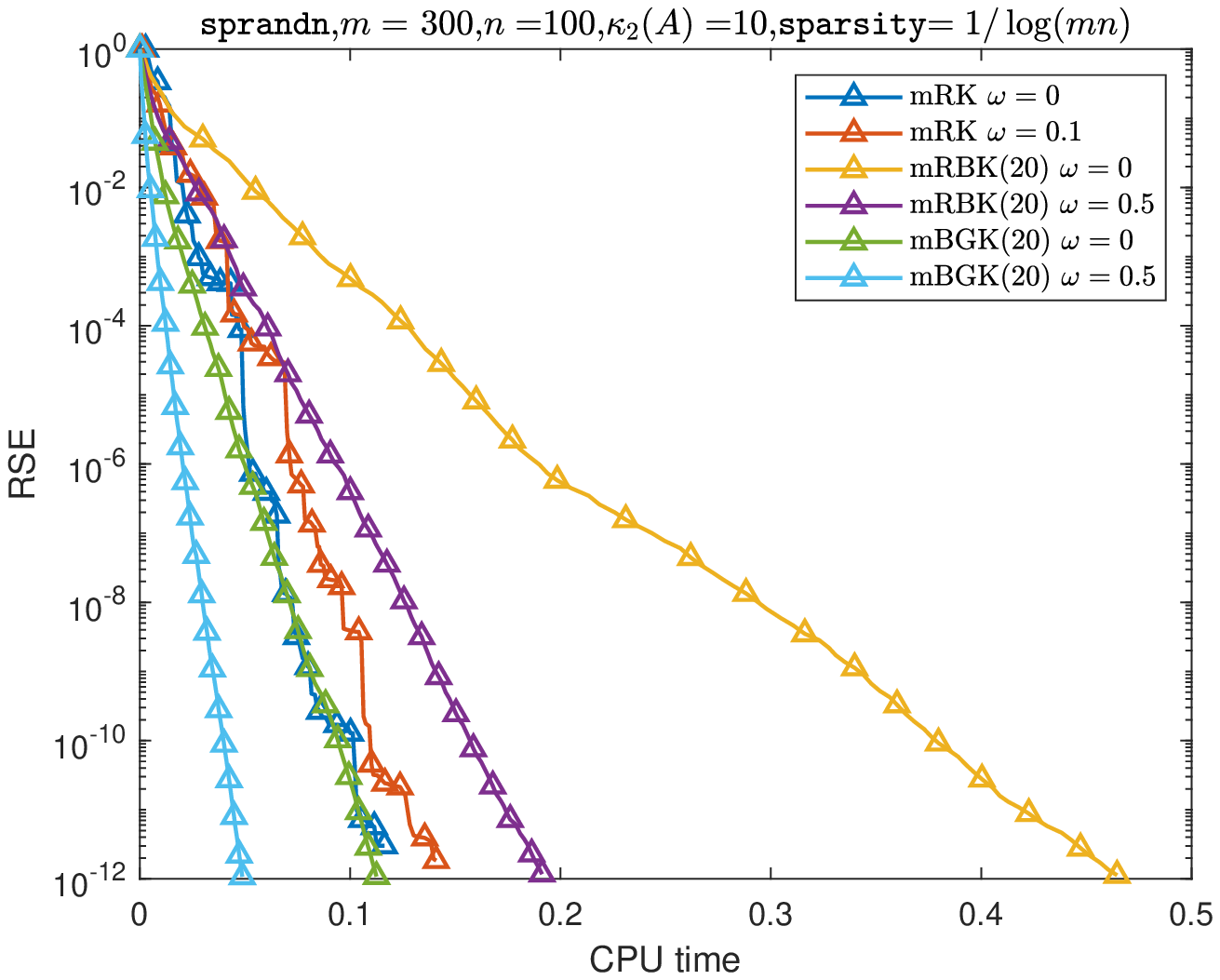}}
 \end{tabular}
 \caption{Performance of mRK, mRBK, and mBGK for consistent linear systems with Gaussian matrix (a) and sparse random matrix (b). The title of each plot indicates the test matrix,  the dimension of the matrix $A$, and the value of condition number $\kappa_2(A)$. We set $p=20$ for both mRBK and mBGK.}
	\label{figue724-1}
 \end{figure}

\begin{figure}[hptb]
 \centering
 \begin{tabular}{cc}
 \subfigure[]{\includegraphics[width=0.48\linewidth]{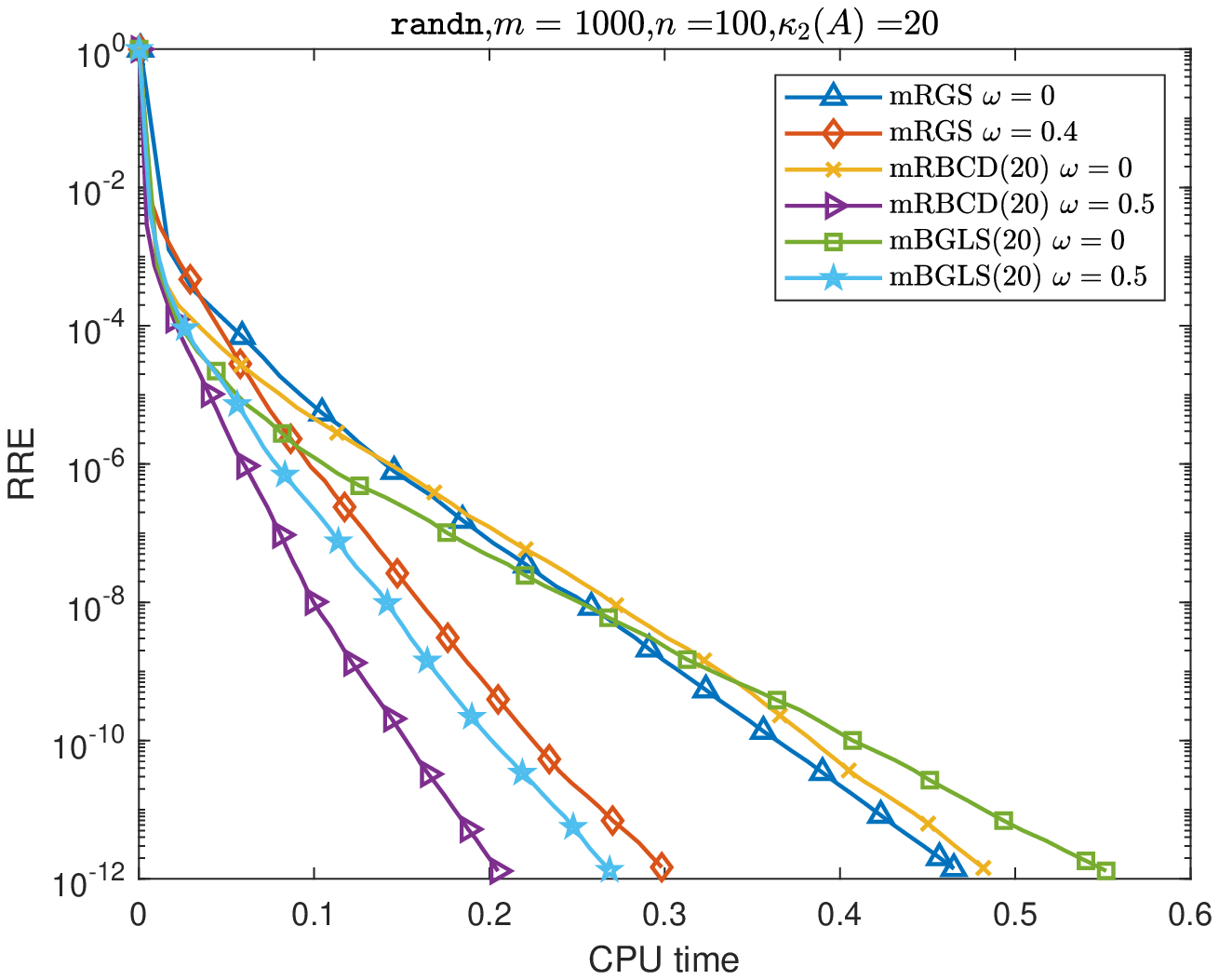}}
 \subfigure[]{\includegraphics[width=0.48\linewidth]{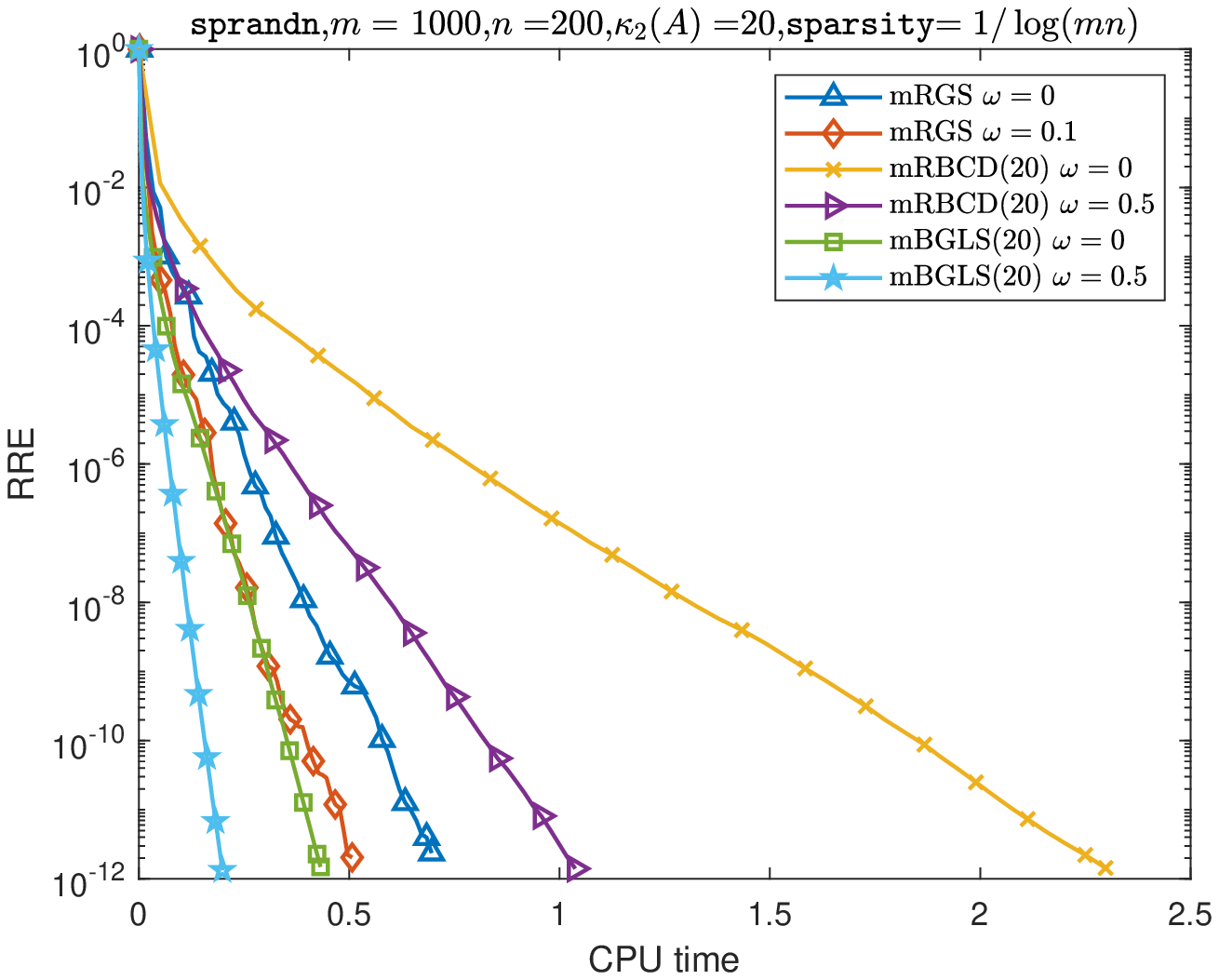}}
 \end{tabular}
 \caption{Performance of mRGS, mRBCD, and mBGLS  for inconsistent linear systems with Gaussian matrix (left) and sparse random matrix (right). The title of each plot indicates the test matrix,  the dimension of the matrix $A$, the value of condition number $\kappa_2(A)$, and the density of $A$. We set $s=20$ for both mRBCD and mBGLS.}
	\label{figue724-3}
 \end{figure}

\subsection{Faster method for average consensus}
Average consensus is a central problem in distributed computing and multi-agent systems \cite{boyd2006randomized,loizou2021revisiting}. It is raised in many applied areas, such as PageRank, coordination of autonomous agents, and rumor spreading in social networks.
In the average consensus (AC) problem, we are given an undirected connected network $G=(V, E)$ with node set $V=\{1,2, \ldots, n\}$ and edge set $E \ (|E|=m)$.
Each node $i \in V$ ``knows'' a private value $c_{i} \in \mathbb{R}$. The goal of the AC problem is for every node to compute the average of these private values, $\bar{c}:=\frac{1}{n} \sum_{i} c_{i}$,  in a
distributed fashion. That is, the exchange of information can only occur between connected nodes (neighbours).
Recently, under an appropriate setting, the famous randomized pairwise gossip algorithm \cite{boyd2006randomized} for solving the AC problem has been proved to be equivalent to the RK method. One may refer to \cite{loizou2016new,loizou2021revisiting} for more details.

In our test, the linear system is the homogeneous linear system ($Ax =
0$) with matrix $A \in\mathbb{R}^{m\times n}$ being the incidence matrix of the undirected graph.
The initial values of the nodes are chosen arbitrarily, i.e. using the {\sc Matlab} function {\tt rand(n,1)}, and the algorithms will find the average of these values.
We note that the incidence matrix $A$ is \emph{rank deficient}. More specifically, it can be shown that $\mbox{rank}(A)=n-1$ \cite{loizou2016new,loizou2021revisiting}.  In comparing the methods, the initial vector is chosen as $x^0=c$ or $x^1=x^0=c$.
The computations are terminated once RSE  is less than $10^{-12}$.

We use three  popular graph topologies in the literature on wireless sensor networks: the line graph, the cycle graph, and the $2$-dimensional random geometric graph $G(n, p)$. In practice, $G(n, p)$ is considered ideal to model wireless sensor networks because of their particular formulation. In the experiments the $2$-dimensional $G(n, p)$ is formed by placing $n$ nodes uniformly at random in a unit square where there are only edges between nodes whose  Euclidean distance is within the given radius $p$.
To preserve the connectivity of $G(n, p)$, a radius $p=p(n)=\log (n) / n$ is recommended \cite{penrose2003random}.

Table \ref{table2}  summarizes the results of the experiment, where we use mRK, mRBK, and mBGK for solving the consistent linear system $Ax=0$. It is clear that the introduction of the momentum term improves
the performance of the methods.
It is shown that mRBK ($\omega=0.5$) is more efficient  than other methods.

\begin{landscape}
{\tiny
\begin{longtable}{  c c c c c c c c c c c c c c }
\caption{Performance of mRK, mRBK, and mBGK  for the AC problem. We set $p=20$ for both mRBK and mBGK. All results are averaged over 10 trials. The label ``--'' in the table
means that the algorithms fail to find the solution within 200 seconds.}
\label{table2}\\
\hline
\multirow{2}{*}{ Size }& \multirow{2}{*}{ Graph }  &\multicolumn{2}{c }{ mRK}  &\multicolumn{2}{c }{mRK}  &\multicolumn{2}{c }{mRBK(20)} &\multicolumn{2}{c }{mRBK(20)} &\multicolumn{2}{c }{mGBK(20)}  &\multicolumn{2}{c }{mGBK(20)}
\\%[-0.05cm]
   & &\multicolumn{2}{c }{$\omega=0$}  &\multicolumn{2}{c }{$\omega=0.5$}  &\multicolumn{2}{c }{$\omega=0$} &\multicolumn{2}{c }{$\omega=0.5$} &\multicolumn{2}{c }{$\omega=0$}  &\multicolumn{2}{c }{$\omega=0.5$}
\\%[0.1cm]
\hline
\hbox{$n$} & &  Iter & CPU    & Iter & CPU    & Iter & CPU   & Iter & CPU    & Iter & CPU   & Iter & CPU   \\%[0.1cm]
\hline

\hline
\\%[0.2cm]
$    100$ & Cycle &   5.94e+05  &   6.6913  & 3.56e+05  &   3.8582 & 3.55e+04  &   0.5535 & 1.77e+04  &   {\bf 0.2723} & 4.22e+04  &   0.9673 & 2.12e+04  &   0.4851 \\[0.1cm]
$    100$  &Line graph  &  2.18e+06  &  26.2444  & 1.33e+06  &  14.7973 & 1.31e+05  &   2.0442 & 6.26e+04  &   {\bf 1.0166} & 1.56e+05  &   3.7249 & 7.82e+04  &   1.8621 \\[0.1cm]
$    100$ & RGG  &  4.23e+04  &   0.6044  & 2.59e+04  &   0.3581 & 2.79e+03  &   0.0576 & 1.39e+03  &   {\bf 0.0296} & 2.96e+03  &   0.3454 & 1.46e+03  &   0.1701 \\[0.1cm]
$    200$  &Cycle  & 4.61e+06  &  67.1620  & 2.72e+06  &  38.3765 & 2.48e+05  &   4.1216 & 1.23e+05  &   {\bf 2.0790} & 2.74e+05  &  11.6263 & 1.37e+05  &   5.8053 \\[0.1cm]
$    200$  &Line graph  &  -- &  -- &  1.03e+07  & 155.7895 & 9.55e+05  &  16.1365 & 4.77e+05  &   {\bf 8.0497} & 1.06e+06  &  44.9593 & 5.28e+05  &  23.3718 \\[0.1cm]
$    200$ & RGG  &  9.42e+04  &   2.0505 &  5.82e+04  &   1.1823 & 5.28e+03  &   0.2039 & 2.67e+03  &   {\bf 0.0939} & 5.50e+03  &   2.8123 & 2.81e+03  &   1.4595 \\[0.1cm]
$    300$  &Cycle  &  --  &  --  & 8.90e+06  & 157.4100 & 8.07e+05  &  15.1372 & 4.04e+05  &  {\bf 7.6067} & 8.65e+05  &  54.2075 & 4.32e+05  &  27.6358 \\[0.1cm]
$    300$ & Line graph  & -- &  --  & --  &  -- & 2.93e+06  &  54.9336 & 1.48e+06  &  {\bf 28.1943} & 3.11e+06  & 194.9407 & 1.53e+06  &  95.3693 \\[0.1cm]
$    300$ & RGG  &  2.19e+05  &   6.2552  & 1.36e+05  &   3.8626 & 1.20e+04  &   0.5547 & 6.25e+03  &   {\bf 0.2608} & 1.25e+04  &  10.8691 & 6.43e+03  &   5.6038 \\[0.1cm]
$    400$  &Cycle  &  --  &   -- &  --  &   -- & 1.71e+06  &  39.6690 & 8.57e+05  &  {\bf 21.7547} & 1.80e+06  & 174.3996 & 8.92e+05  &  73.7291 \\[0.1cm]
$    400$ & Line graph &   --  &   -- &  --  &   -- & --  &-- & 3.72e+06  &  {\bf 94.9543} & --  &   --& --  & -- \\[0.1cm]
$    400$ & RGG  &  3.50e+05  &  12.1921 &  2.08e+05  &   7.0579 & 1.88e+04  &   1.0795 & 9.81e+03  &   {\bf 0.4774} & 1.91e+04  &  18.1648 & 9.55e+03  &   9.2028 \\[0.1cm]
$    500$  &Cycle  &  --  &   -- &  --  &  -- &3.69e+06  & 119.9571 & 1.85e+06  &  {\bf 61.3696} & --  & -- & -- & -- \\[0.1cm]
%$    500$  &Line graph &   --  &   -- &  --  &   -- & --  & -- & 6.06e+06  & 158.7906 & --  &   -- & --  & -- \\[0.1cm]
$    500$  &RGG  &  3.75e+05  &  15.7980 &  2.25e+05  &   9.7337 & 1.91e+04  &   1.3119 & 1.07e+04  &   {\bf 0.6695} & 1.94e+04  &  27.2015 & 9.63e+03  &  14.6479 \\[0.1cm]
\hline
\end{longtable}
}
\end{landscape}

\section{Concluding remarks}

In this work, we have proposed a generic  pseudoinverse-free randomized method for solving different  types of linear systems. Our method is formulated with general randomized sampling matrices as well as enriched with Polyak's heavy ball momentum. We proved the global convergence rates of the method as well as an accelerated linear rate for the case of the norm of expected iterates.
Our general framework can lead to the improvement of several existing algorithms, and can produce new algorithms.
Numerical tests reveal that the new methods based on Gaussian sampling are competitive on sparse problems, as compared to mRK, mRBK, mRGS, and mRBCD.

We believe that this work could open up several future avenues for research. In \cite{Zou12}, Zouzias and Freris studied the randomized extended Kaczmarz for solving least squares. They proved that REK converges to a solution of $Ax = AA^{\dagger}b$. We note that any solution of $Ax = AA^{\dagger}b$ is a solution of $Ax = b$ if it is
consistent or a least squares solution of $Ax = b$ if it is inconsistent. While in Theorem \ref{THMr-in629}, our convergence result is about the iterative residuals, and this is actually kind of hard to know which solution our method would converge to. It is a valuable topic to study the extended pseudoinverse-free methods with better convergence results. The randomized sampling matrices in our work need to satisfy Assumption \ref{assumption1},
it is convenient to use more general sampling matrices  for obtaining new classes of randomized algorithms.

\bibliographystyle{abbrv}
\bibliography{main0823}

\section{Appendix. Proof of the main results}

The following lemma is crucial in our proof.
\begin{lemma}
\label{lemma-key629}
Fix $F^1=F^0\geq 0,\zeta\geq0$, and let $\{F^k\}_{k\geq 0}$ be a sequence of nonnegative real numbers satisfying the relation
\begin{equation}\label{lemma6291}
F^{k+1}\leq \gamma_1 F^k+\gamma_2F^{k-1}+\zeta,\ \ \forall \ k\geq 1,
\end{equation}
where $\gamma_2\geq0,\gamma_1+\gamma_2<1$. Then the sequence satisfies the relation
$$F^{k+1}\leq q^k(1+\tau)F^0+\frac{1-q^k}{1-q}\zeta,\ \ \forall \ k\geq 0,$$
where $$q=\left\{
           \begin{array}{ll}
             \frac{\gamma_1+\sqrt{\gamma_1^2+4\gamma_2}}{2}, & \hbox{if $\gamma_2>0$;} \\
             \gamma_1, & \hbox{if $\gamma_2=0$,}
           \end{array}
         \right. \  \mbox{and} \ \tau=q-\gamma_1\geq 0.$$ Moreover,
$
\gamma_1+\gamma_2\leq q<1,
$
with equality if and only if $\gamma_2=0$.
\end{lemma}
\begin{proof}
We first show that $\tau \geq 0$ and $\gamma_{2} =\left(\gamma_{1}+\tau\right) \tau$. Indeed, since
$$\tau=q-\gamma_1=\left\{
           \begin{array}{ll}
             \frac{-\gamma_{1}+\sqrt{\gamma_{1}^{2}+4 \gamma_{2}}}{2}, & \hbox{if $\gamma_2>0$;} \\
             0, & \hbox{if $\gamma_2=0$,}
           \end{array}
         \right.$$
the nonnegativity of $\tau$ can be derived from $\gamma_2\geq0$. It is easy to verify that
$\tau$ satisfies
\begin{equation}\label{lemma6293}
\left(\gamma_{1}+\tau\right) \tau=\gamma_{2}.
\end{equation}
By \eqref{lemma6293} and adding $\tau F_{k}$ to both sides of \eqref{lemma6291}, we get
%\begin{equation}\label{lemma6292}
$$\begin{array}{ll}
F^{k+1}&\leq F^{k+1}+\tau F^{k}=\left(\gamma_{1}+\tau\right) F^{k}+\gamma_{2} F^{k-1}+\zeta \\
&=\left(\gamma_{1}+\tau\right)\left(F^{k}+\tau F^{k-1}\right)+\zeta
=q\left(F^{k}+\tau F^{k-1}\right)+\zeta\\
&=q^{k}\left(F^{1}+\tau F^{0}\right)+\zeta\sum\limits_{i=0}^{k-1} q^i
=q^{k}(1+\tau) F^{0}+\frac{1-q^k}{1-q}\zeta,
\end{array}$$
where the inequality follows from $\tau\geq0$.

Finally, let us establish $\gamma_1+\gamma_2\leq q < 1$. The inequality $q<1$ follows directly from the assumption $\gamma_{1}+\gamma_{2}<1$.
Noting that $\gamma_{1}=q-\tau$, and since in view of \eqref{lemma6293} we have $\gamma_{2}=q \tau$, we conclude that $\gamma_{1}+\gamma_{2}=q+\tau(q-1) \leq q$, where the inequality follows from $q<1$ and $\tau\geq0$.
\end{proof}

\subsection{Proof of Theorems \ref{THMr-in} and \ref{THMr-in629}}
In fact, Theorem \ref{THMr-in} can be directly derived from Theorem \ref{THMr-in629} by letting $\omega=0$. We include a simple proof of Theorem \ref{THMr-in} here for ease of reading.

\begin{proof}[Proof of Theorem \ref{THMr-in}]
Straightforward calculations yield
$$\begin{array}{ll}
\|r^{k+1}-r^*\|^2_2&=\|Ax^{k+1}-b-r^*\|^2_2
\\
&=\|Ax^k-b-\alpha AT_1T_2^\top A^\top S_1S_2^\top(Ax^{k}-b)-r^*\|^2_2
\\
&=\|r^k-r^*-\alpha AT_1T_2^\top A^\top S_1S_2^\top r^{k}\|^2_2
\\
&=\|r^k-r^*\|^2_2-2\alpha\langle r^k-r^*, AT_1T_2^\top A^\top S_1S_2^\top r^{k}\rangle +\alpha^2\| AT_1T_2^\top A^\top S_1S_2^\top r^{k}\|^2_2.
\end{array}
$$
Noting that $\mathbb{E}[T_1T_2^\top A^\top S_1S_2^\top]=\frac{A^\top}{\|A\|^2_F}$ and $A^\top r^*=0$, we have
\begin{equation}\label{proof-THM-5}
\begin{array}{ll}
\mathbb{E}[\|r^{k+1}-r^*\|^2_2|x^k]
&=\|r^k-r^*\|^2_2-2\alpha\big\langle r^k-r^*, \mathop{\mathbb{E}}[AT_1T_2^\top A^\top S_1S_2^\top] r^{k}\big\rangle
\\
&\ \ \ +\alpha^2\mathbb{E}[\| AT_1T_2^\top A^\top S_1S_2^\top r^{k}\|^2_2]\\
&=\|r^k-r^*\|^2_2-2\frac{\alpha}{\|A\|^2_F} \langle r^k-r^*,A A^\top r^k\rangle
 +\alpha^2\mathbb{E}[\| AT_1T_2^\top A^\top S_1S_2^\top r^{k}\|^2_2]\\
&=\|r^k-r^*\|^2_2-2\frac{\alpha}{\|A\|^2_F}  \langle r^k-r^*,A A^\top (r^k-r^*)\rangle
 \\
 &\ \ \ +\alpha^2\mathbb{E}[\| AT_1T_2^\top A^\top S_1S_2^\top r^{k}\|^2_2].
\end{array}
\end{equation}
Since $r^k-r^*\in \mbox{Range}(A)$, we know that
\begin{equation}\label{proof-THM-4}
\langle r^k-r^*, AA^\top (r^k-r^*)\rangle\geq\sigma^2_{\min}(A)\|r^k-r^*\|^2_2.
\end{equation}
Using the fact that $\|a-b\|^2_2\leq 2\|a\|^2_2+2\|b\|^2_2$, we have
\begin{equation}
\label{proof-THM-2}
\begin{array}{ll}
\| AT_1T_2^\top A^\top S_1S_2^\top r^{k}\|^2_2&=\| AT_1T_2^\top A^\top S_1S_2^\top (r^{k}-r^*+r^*)\|^2_2\\
&\leq 2\| AT_1T_2^\top A^\top S_1S_2^\top(r^{k}-r^*)\|_2^2+2\| AT_1T_2^\top A^\top S_1S_2^\top r^*\|_2^2.
\end{array}
\end{equation}
Note that
\begin{equation}
\label{proof-THM-3}
\begin{array}{ll}
&\mathbb{E}[\| AT_1T_2^\top A^\top S_1S_2^\top (r^{k}-r^*)\|_2^2]\\
&=(r^k-r^*)^\top\mathbb{E}\big[ ( AT_1T_2^\top A^\top S_1S_2^\top)^\top AT_1T_2^\top A^\top S_1S_2^\top \big](r^k-r^*)\\
&\leq \beta\|r^k-r^*\|^2_2
\end{array}
\end{equation}
and similarly,
\begin{equation}
\label{proof-THM-6}
\mathbb{E}[\| AT_1T_2^\top A^\top S_1S_2^\top r^*\|_2^2]\leq \beta\|r^*\|^2_2.
\end{equation}
Combining  \eqref{proof-THM-5}, \eqref{proof-THM-4}, \eqref{proof-THM-2}, \eqref{proof-THM-3} and \eqref{proof-THM-6}  yields
$$
\begin{array}{ll}
\mathbb{E}[\|r^{k+1}-r^*\|^2_2|x^k]
&\leq\|r^k-r^*\|^2_2-2\alpha\frac{\sigma^2_{\min}(A)}{\|A\|^2_F}\|r^k-r^*\|^2_2\\
&\ \ \ +2\alpha^2\beta \|r^k-r^*\|^2_2+2\alpha^2\beta\|r^*\|^2_2\\
&=\eta \|r^k-r^*\|^2_2+2\alpha^2\beta\|r^*\|^2_2.
\end{array}
$$
Taking the expectation over the entire history,
we have
$$
\begin{array}{ll}
\mathbb{E}[\|r^{k}-r^*\|^2_2]
&=\eta \mathbb{E}[\|r^{k-1}-r^*\|^2_2]+2\alpha^2\beta\|r^*\|^2_2
\\
&\leq \eta^{k} \|r^0-r^*\|^2_2+2\alpha^2\beta\|r^*\|^2_2\sum\limits_{i=0}^{k-1}\eta^i\\
&=\eta^{k} \|r^0-r^*\|^2_2+\frac{2\alpha^2\beta(1-\eta^k)\|r^*\|^2_2}{1-\eta}
\end{array}
$$
as desired.
\end{proof}

Now we will prove Theorem \ref{THMr-in629}.

\begin{proof}[Proof of Theorem \ref{THMr-in629}]
Straightforward calculations yield
\begin{equation}\label{proof-THM-6291}
\begin{array}{ll}
\|r^{k+1}-r^*\|^2_2&=\|Ax^{k+1}-b-r^*\|^2_2
\\
&=\|Ax^k-b-\alpha AT_1T_2^\top A^\top S_1S_2^\top(Ax^{k}-b)+\omega(Ax^k-Ax^{k-1})-r^*\|^2_2
\\
&=\|r^k-r^*-\alpha AT_1T_2^\top A^\top S_1S_2^\top r^{k}+\omega(r^k-r^{k-1})\|^2_2
\\
&=\underbrace{\|r^k-r^*-\alpha AT_1T_2^\top A^\top S_1S_2^\top r^{k}\|^2_2}_{\textcircled{a}}\\
&\ \ \ +\underbrace{2\omega\langle r^k-r^*-\alpha AT_1T_2^\top A^\top S_1S_2^\top r^{k},  r^{k}-r^{k-1}\rangle }_{\textcircled{b}} +\underbrace{\omega^2\|r^{k}-r^{k-1}\|^2_2}_{\textcircled{c}}.
\end{array}
\end{equation}
We will now analyze the three expressions $\textcircled{a},\textcircled{b},\textcircled{c}$ separately. The first expression can be written as
$$
\textcircled{a}=\|r^k-r^*\|^2_2-2\alpha\langle r^k-r^*, AT_1T_2^\top A^\top S_1S_2^\top r^{k}\rangle +\alpha^2\| AT_1T_2^\top A^\top S_1S_2^\top r^{k}\|^2_2.
$$
We will now bound the second expression. First, we have
$$
\begin{array}{ll}
\textcircled{b}&=2\omega\langle r^k-r^*,r^k-r^{k-1}\rangle-2\omega\alpha\langle AT_1T_2^\top A^\top S_1S_2^\top r^{k},r^k-r^{k-1}\rangle\\
&=2\omega\langle r^k-r^*,r^k-r^*\rangle+2\omega\langle r^k-r^*,r^*-r^{k-1}\rangle-2\omega\alpha\langle AT_1T_2^\top A^\top S_1S_2^\top r^{k},r^k-r^{k-1}\rangle.
\end{array}
$$
Noting that
$$
2\langle r^{k}-r^*,r^*-r^{k-1} \rangle\leq \|r^k-r^*\|^2_2+\|r^{k-1}-r^*\|^2_2,
$$
which implies
$$
\textcircled{b}\leq 3\omega\|r^k-r^*\|^2_2+\omega\|r^{k-1}-r^*\|^2_2-2\omega\alpha\langle AT_1T_2^\top A^\top S_1S_2^\top r^{k},r^k-r^{k-1}\rangle.
$$
The third expression can be bounded as
$$\textcircled{c}\leq2\omega^2 \|r^{k}-r^*\|^2_2+2\omega^2\|r^{k-1}-r^*\|^2_2.$$
By substituting all the bounds  into \eqref{proof-THM-6291}, we obtain
$$
\begin{array}{ll}
\|r^{k+1}-r^*\|^2_2&\leq(1+3\omega+2\omega^2)\|r^k-r^*\|^2_2+(\omega+2\omega^2)\|r^{k-1}-r^*\|^2_2
\\
&\ \ \ -2\alpha\langle r^k-r^*, AT_1T_2^\top A^\top S_1S_2^\top r^{k}\rangle +\alpha^2\| AT_1T_2^\top A^\top S_1S_2^\top r^{k}\|^2_2
\\
&\ \ \ -2\omega\alpha\langle AT_1T_2^\top A^\top S_1S_2^\top r^{k},r^k-r^{k-1}\rangle.
\end{array}
$$
Now taking the conditional expectation and noting that $\mathbb{E}[T_1T_2^\top A^\top S_1S_2^\top]=\frac{A^\top}{\|A\|^2_F}$ and $A^\top r^*=0$, we have
\begin{equation}\label{prf-main-71-1}
\begin{array}{ll}
\mathbb{E}\left[\|r^{k+1}-r^*\|^2_2|x^k\right]&\leq(1+3\omega+2\omega^2)\|r^k-r^*\|^2_2+(\omega+2\omega^2)\|r^{k-1}-r^*\|^2_2
\\
&\ \ \ -2\alpha\langle r^k-r^*, \mathbb{E}[AT_1T_2^\top A^\top S_1S_2^\top] r^{k}\rangle +\alpha^2\mathbb{E}[\| AT_1T_2^\top A^\top S_1S_2^\top r^{k}\|^2_2]
\\
&\ \ \ -2\omega\alpha\langle \mathbb{E}[AT_1T_2^\top A^\top S_1S_2^\top] r^{k},r^k-r^{k-1}\rangle\\
&=(1+3\omega+2\omega^2)\|r^k-r^*\|^2_2+(\omega+2\omega^2)\|r^{k-1}-r^*\|^2_2
\\
&\ \ \ -\frac{2\alpha}{\|A\|^2_F}\langle r^k-r^*, A A^\top( r^{k}-r^*)\rangle+\frac{2\omega\alpha}{\|A\|^2_F}\langle A A^\top (r^{k}-r^*),r^{k-1}-r^{k}\rangle
\\
&\ \ \ +\alpha^2\mathbb{E}[\| AT_1T_2^\top A^\top S_1S_2^\top r^{k}\|^2_2].
\end{array}
\end{equation}
Since $r^k-r^*\in \mbox{Range}(A)$, we know that
\begin{equation}\label{proof-THM-6294}
\langle r^k-r^*, AA^\top (r^k-r^*)\rangle\geq\sigma^2_{\min}(A)\|r^k-r^*\|^2_2.
\end{equation}
From  \eqref{proof-THM-2}, \eqref{proof-THM-3} and \eqref{proof-THM-6}, we have
\begin{equation}
\label{proof-THM-6292}
\mathbb{E}[\| AT_1T_2^\top A^\top S_1S_2^\top r^{k}\|^2_2]\leq 2\beta\|r^k-r^*\|^2_2+2\beta\|r^*\|^2_2
\end{equation}
Note that
\begin{equation}\label{prf-main-71-2}
\begin{array}{ll}
2\langle A A^\top (r^{k}-r^*),r^{k-1}-r^{k}\rangle&=2\langle A A^\top (r^{k}-r^*),(r^{k-1}-r^*)-(r^{k}-r^*)\rangle
\\
&\leq \|A^\top(r^{k-1}-r^*)\|_2^2-\|A^\top(r^{k}-r^*)\|_2^2
\\
&\leq \sigma_{\max}^2(A)\|r^{k-1}-r^*\|_2^2-\sigma_{\min}^2(A)\|r^{k}-r^*\|_2^2.
\end{array}
\end{equation}
Combining  \eqref{prf-main-71-1}, \eqref{proof-THM-6294}, \eqref{proof-THM-6292} and \eqref{prf-main-71-2}  yields
$$
\begin{array}{ll}
\mathbb{E}\left[\|r^{k+1}-r^*\|^2_2|x^k\right]&\leq\bigg(\underbrace{1+3\omega+2\omega^2+2\alpha^2\beta
-(2\alpha+\alpha\omega)\frac{\sigma_{\min}^2(A)}{\|A\|^2_F}}_{\gamma_1}\bigg)
\|r^k-r^*\|^2_2
\\
&\ \ \ +\bigg(\underbrace{\omega+2\omega^2+\omega\alpha\frac{\sigma^2_{\max}(A)}{\|A\|^2_F}}_{\gamma_2}\bigg)\|r^{k-1}-r^*\|^2_2
+\underbrace{2\beta\alpha^2\|r^*\|^2_2}_{\zeta}.
\end{array}
$$
By taking expectation over the entire history, and letting $F^k:=\mathbb{E}[\|r^k-r^*\|^2_2]$, we get the relation
\begin{equation}\label{prf-THMm10}
F^{k+1}\leq \gamma_1F^k+\gamma_2F^{k-1}+\zeta.
\end{equation}
Noting that the conditions of the Lemma \ref{lemma-key629} are satisfied. Indeed, $\gamma_2\geq0$, and if $\gamma_2=0$, then $\omega=0$.
The condition $\gamma_1+\gamma_2<1$ holds by assumption.
Then apply Lemma \ref{lemma-key629} to \eqref{prf-THMm10}, one can get the theorem.
\end{proof}

\subsection{Proof of Theorems \ref{THMnm2}, \ref{THMmain}, and \ref{THMfm2}}
Let us first present  two useful lemmas.
\begin{lemma}\label{lemma-62912}
Suppose that the singular value decomposition of $A=U\Sigma V^\top$ and let $\{x^k\}_{k=0}^{\infty}$ be the sequences generated by Algorithm \ref{mPDFR} or Algorithm \ref{PFR1} $(\omega=0)$. Then
$$
V^\top\mathbb{E}[x^{k+1}-x^0_*]=\bigg((1+\omega)I-\alpha\frac{\Sigma^\top\Sigma}{\|A\|^2_F}\bigg)
V^\top\mathbb{E}[x^k-x^0_*]-\omega V^\top\mathbb{E}[x^{k-1}-x^0_*].
$$
\end{lemma}
\begin{proof}
First, we have
$$
\begin{array}{ll}
x^{k+1}-x^0_*&=x^k-\alpha T_1T_2A^\top S_1S_2^\top(Ax^k-b)+\omega(x^k-x^{k-1})-x_*^0
\\
&=(1+\omega)\left(x^k-x_*^0\right)-\alpha T_1T_2A^\top S_1S_2^\top(Ax^k-b)-\omega(x^{k-1}-x_*^0).
\end{array}
$$
Taking expectations and noting that $A^\top Ax_*^0=A^\top b$, we have
$$
\begin{array}{ll}
\mathbb{E}\big[x^{k+1}-x_*^0|x^k\big]&=(1+\omega)\left(x^k-x_*^0\right)-\alpha \mathbb{E}[T_1T_2A^\top S_1S_2^\top](Ax^k-b)-\omega(x^{k-1}-x_*^0)\\
&= (1+\omega)\left(x^k-x_*^0\right)-\alpha\frac{A^\top A}{\|A\|^2_F}\left(x^k-x_*^0\right)-\omega(x^{k-1}-x_*^0)\\
&=\bigg((1+\omega)I-\alpha\frac{A^\top A}{\|A\|^2_F}\bigg)(x^k-x_*^0)-\omega(x^{k-1}-x_*^0).
\end{array}
$$
Taking the expectations again, we have
\begin{equation}\label{lemmaeq-613}
\mathbb{E}[x^{k+1}-x_*^0]=\bigg((1+\omega)I-\alpha\frac{A^\top A}{\|A\|^2_F}\bigg)\mathbb{E}[x^k-x_*^0]-\omega\mathbb{E}[x^{k-1}-x_*^0].
\end{equation}
Plugging $A^\top A=V\Sigma^\top\Sigma V^\top$ into \eqref{lemmaeq-613}, and multiplying both sides form the left by $V^\top$, we can get the lemma.
%as desired.
\end{proof}

\begin{lemma}[\cite{fillmore1968linear,elaydi1996introduction}]\label{Lemma-relation}
Consider the second degree linear homogeneous recurrence relation:
$$
r^{k+1}=\gamma_{1} r^{k}+\gamma_{2} r^{k-1}
$$
with initial conditions $r^{0}, r^{1} \in \mathbb{R}$. Assume that the constant coefficients $\gamma_{1}$ and $\gamma_{2}$ satisfy the inequality $\gamma_{1}^{2}+4 \gamma_{2}<0$ (the roots of the characteristic equation $t^{2}-\gamma_{1} t-$ $\gamma_{2}=0$ are imaginary). Then there are complex constants $c_{0}$ and $c_{1}$ (depending on the initial conditions $r_{0}$ and $r_{1}$) such that:
$$
r^{k}=2 M^{k}\left(c_{0} \cos (\theta k)+c_{1} \sin (\theta k)\right)
$$
where $M=\left(\sqrt{\frac{\gamma_{1}^{2}}{4}+\frac{\left(-\gamma_{1}^{2}-4 \gamma_{2}\right)}{4}}\right)=\sqrt{-\gamma_{2}}$ and $\theta$ is such that $\gamma_{1}=2 M \cos (\theta)$ and $\sqrt{-\gamma_{1}^{2}-4 \gamma_{2}}=2 M \sin (\theta)$.
\end{lemma}

We first prove Theorem \ref{THMfm2}.
\begin{proof}[Proof of Theorem \ref{THMfm2}]
Set $s^k:=V^\top\mathbb{E}[x^{k}-x^0_*]$. Then from Lemma \ref{lemma-62912} we have
$$
s^{k+1}=\bigg((1+\omega)I-\alpha\frac{\Sigma^\top\Sigma}{\|A\|^2_F}\bigg)s^k-\omega s^{k-1},$$
which can be rewritten in a coordinate-by-coordinate form as follows:
\begin{equation}\label{prf-THMm2-1}
s^{k+1}_i=\big((1+\omega)-\alpha\sigma^2_i(A)/\|A\|^2_F\big)s_i^k-\omega s_i^{k-1}, \ \ \forall i=1,2,\ldots,n,
\end{equation}
where $s^k_i$ indicated the $i$-th coordinate of $s^k$.

We now consider two cases:
$\sigma_i(A)=0$ or $\sigma_i(A)>0$.

If $\sigma_i(A)=0$, then \eqref{prf-THMm2-1} takes the form:
$$
s^{k+1}_i=(1+\omega)s_i^k-\omega s_i^{k-1}.
$$
Since $x^1-x^0_*=x^0-x^0_*=A^\dagger(Ax^0-b)$, we have $s^0_i=s^1_i=v_i^\top A^\dagger(Ax^0-b)=0$.
So
$$
s^k_i=0 \ \ \mbox{for all} \ k\geq0.
$$

If $\sigma_i(A)>0$. We use Lemma \ref{Lemma-relation} to establish the desired bound. Since $0<\alpha\leq\frac{\|A\|^2_F}{\sigma_{\max}^2(A)}$, one can verify that $1-\alpha\sigma_i^2(A)/\|A\|^2_F\geq 0$ and since $\omega\geq0$, we have $1+\omega-\alpha\sigma_i^2(A)/\|A\|^2_F\geq 0$ and hence
$$
\begin{array}{ll}
\gamma_1^2+4\gamma_2&=\big(1+\omega-\alpha\sigma_i^2(A)/\|A\|^2_F\big)^2-4\omega
\\
&\leq\big(1+\omega-\alpha\sigma_{\min}^2(A)/\|A\|^2_F\big)^2-4\omega
\\
&<0,
\end{array}
$$
where the last inequality can be shown to hold for
$$\left(1-\sqrt{\alpha\sigma_{\min}^2(A)/\|A\|^2_F}\right)^2<\omega<1.$$
Using Lemma \ref{Lemma-relation}, the following bound can be deduced
$$
s^{k}_{i}=2\left(-\gamma_{2}\right)^{k / 2}\left(c_{0} \cos (\theta k)+c_{1} \sin (\theta k)\right) \leq 2 \omega^{k / 2} p_{i},
$$
where $p_{i}$ is a constant depending on the initial conditions (we can simply choose $p_{i}=\left|c_{0}\right|+\left|c_{1}\right|.$)
Now put the two cases together, for all $k \geq 0$ we have
$$
\begin{aligned}
\left\|\mathbb{E}\left[x^{k}-x^0_{*}\right]\right\|_{2}^{2} &=\left\|V^{\top} \mathbb{E}\left[x^{k}-x^0_{*}\right]\right\|_{2}^{2}=\left\|s^{k}\right\|^{2}=\sum_{i=1}^{n}\left(s^{k}_{i}\right)^{2} \\
&=\sum_{i: \sigma_{i}(A)=0}\left(s^{k}_{i}\right)^{2}+\sum_{i: \sigma_{i}(A)>0}\left(s^{k}_{i}\right)^{2} = \sum_{i: \sigma_{i}(A)>0}\left(s^{k}_{i}\right)^{2} \\
& \leq \sum_{i: \sigma_{i}(A)>0} 4 \omega^{k} p_{i}^{2} \\
&=\omega^{k} c,
\end{aligned}
$$
where $c=4 \sum_{i: \lambda_{i}>0} p_{i}^{2}$.
\end{proof}

\begin{proof}[Proof of Theorem \ref{THMnm2}]
Theorem \ref{THMnm2} can be directly derived from Lemma \ref{lemma-62912}. Indeed, using similar arguments as that in the proof of Theorem \ref{THMfm2}, we can get
$$
\left\|\mathbb{E}\left[x^{k}-x^0_{*}\right]\right\|_{2}^{2} = \sum_{i: \sigma_{i}(A)>0}\left(s^{k}_{i}\right)^{2},
$$
where $$s_i^{k}=\left(1-\alpha\sigma^2_i(A)/\|A\|^2_F\right)s_i^{k-1}, i=1,\ldots,n.$$
Hence, we have
$$
  \left\|\mathbb{E}\left[x^k-x^0_*\right]\right\|^2_2\leq\left(1-\alpha\frac{\sigma_{\min}^2(A)}{\|A\|^2_F}
\right)^{2k}\left\|x^0-x_0^*\right\|^2_2.
$$
This completes the proof of the theorem.
\end{proof}

\begin{proof}[Proof of Theorem \ref{THMmain}]
Using similar arguments as that in the proof of Theorem \ref{THMfm2} and \eqref{prf-THMm2-1}, we know that
$$
s^{k+1}_\ell=\big((1-\alpha)+\alpha\big(1-2\sigma_{\ell}/\|A\|^2_F\big)^r\big)s_\ell^k, \ \ \forall \ell=1,2,\ldots,n,
$$
where
$s^{k+1}_{\ell}=\mathop{\mathbb{E}}[\langle x^{k+1}-x^0_*,v_{\ell}\rangle]$ and hence the theorem holds.
\end{proof}

\subsection{Proof of Theorems \ref{THMx} and \ref{THMcfull}}

\begin{proof}[Proof of Theorem \ref{THMx}]
Straightforward calculations yield
\begin{equation}\label{proof-THM-6301}
\begin{array}{ll}
\|x^{k+1}-A^{\dagger}b\|^2_2&=\|x^k-\alpha T_1T_2^\top A^\top S_1S_2^\top(Ax^k-b)+\omega(x^k-x^{k-1})-A^{\dagger}b\|^2_2
\\&=\underbrace{\|x^{k}-\alpha T_1T_2^\top A^\top S_1S_2^\top(Ax^k-b)-A^{\dagger}b\|^2_2}_{\textcircled{a}}\\
& \ \ \ +\underbrace{2\omega\big\langle x^{k}-\alpha T_1T_2^\top A^\top S_1S_2^\top(Ax^k-b)-A^{\dagger}b,x^k-x^{k-1}\big\rangle}_{\textcircled{b}}\\
&\ \ \ +\underbrace{\omega^2\|x^k-x^{k-1}\|^2_2}_{\textcircled{c}}.
\end{array}
\end{equation}
We will now analyze the three expressions $\textcircled{a},\textcircled{b},\textcircled{c}$ separately.
The first expression can be written as
$$
\textcircled{a}=\|x^{k}-A^{\dagger}b\|^2_2-2\alpha\big\langle x^{k}-A^{\dagger}b,T_1T_2^\top A^\top S_1S_2^\top(Ax^k-b)\big\rangle+\alpha^2\|T_1T_2^\top A^\top S_1S_2^\top (Ax^k-b)\|^2_2.
$$
We will now bound the second expression. We have
$$
\begin{array}{ll}
\textcircled{b}&=2\omega\big\langle x^{k}-A^{\dagger}b,x^k-x^{k-1}\big\rangle-2\omega\alpha\big\langle T_1T_2^\top A^\top S_1S_2^\top(Ax^k-b),x^k-x^{k-1}\big\rangle\\
&=2\omega\big\langle x^{k}-A^{\dagger}b,x^k-A^{\dagger}b\big\rangle+2\omega\big\langle x^{k}-A^{\dagger}b,A^{\dagger}b-x^{k-1}\big\rangle
\\
&\ \ \ -2\omega\alpha\big\langle T_1T_2^\top A^\top S_1S_2^\top(Ax^k-b),x^k-x^{k-1}\big\rangle
\\
&\leq 3\omega\|x^{k}-A^{\dagger}b\|^2_2+\omega\|x^{k-1}-A^{\dagger}b\|^2_2-2\omega\alpha\big\langle T_1T_2^\top A^\top S_1S_2^\top(Ax^k-b),x^k-x^{k-1}\big\rangle.
\end{array}
$$
The third expression can be bounded as
$$\textcircled{c}\leq2\omega^2 \|x^{k}-A^{\dagger}b\|^2_2+2\omega^2\|x^{k-1}-A^{\dagger}b\|^2_2.$$
By substituting all the bounds  into \eqref{proof-THM-6301}, we obtain
$$
\begin{array}{ll}
\|x^{k+1}-A^{\dagger}b\|^2_2&\leq(1+3\omega+2\omega^2)\|x^{k}-A^{\dagger}b\|^2_2+(\omega+2\omega^2)\|x^{k-1}-A^{\dagger}b\|^2_2\\
&\ \ \ -2\alpha\big\langle x^{k}-A^{\dagger}b,T_1T_2^\top A^\top S_1S_2^\top(Ax^k-b)\big\rangle
\\
&\ \ \ +\alpha^2\|T_1T_2^\top A^\top S_1S_2^\top (Ax^k-b)\|^2_2
\\
&\ \ \ -2\omega\alpha \big\langle T_1T_2^\top A^\top S_1S_2^\top(Ax^k-b),x^k-x^{k-1}\big\rangle.
\end{array}
$$
Now taking the conditional expectation and noting that $\mathbb{E}[T_1T_2^\top A^\top S_1S_2^\top]=\frac{A^\top}{\|A\|^2_F}$ and $A^\top AA^\dagger b=A^\top b$, we have
\begin{equation}\label{prf-main-71-432}
\begin{array}{ll}
\mathbb{E}\left[\|x^{k+1}-A^{\dagger}b\|^2_2|x^k\right]&\leq (1+3\omega+2\omega^2)\|x^{k}-A^{\dagger}b\|^2_2+(\omega+2\omega^2)\|x^{k-1}-A^{\dagger}b\|^2_2\\
&\ \ \ -2\alpha\big\langle x^{k}-A^{\dagger}b,\mathbb{E}[T_1T_2^\top A^\top S_1S_2^\top](Ax^k-b)\big\rangle
\\
&\ \ \ +\alpha^2\mathbb{E}\big[\|T_1T_2^\top A^\top S_1S_2^\top (Ax^k-b)\|^2_2\big]
\\
&\ \ \ -2\omega\alpha\big\langle \mathbb{E}[T_1T_2^\top A^\top S_1S_2^\top](Ax^k-b),x^k-x^{k-1}\big\rangle
\\
&=(1+3\omega+2\omega^2)\|x^{k}-A^{\dagger}b\|^2_2+(\omega+2\omega^2)\|x^{k-1}-A^{\dagger}b\|^2_2\\
&\ \ \ -\frac{2\alpha}{\|A\|^2_F}\big\langle x^{k}-A^{\dagger}b,A^\top A(x^{k}-A^{\dagger}b)\big\rangle
\\
&\ \ \ +\alpha^2\mathbb{E}\big[\|T_1T_2^\top A^\top S_1S_2^\top (Ax^k-b)\|^2_2\big]
\\
&\ \ \ +\frac{2\omega\alpha}{\|A\|^2_F}\big\langle A^\top A(x^{k}-A^{\dagger}b),x^{k-1}-x^{k}\big\rangle.
\end{array}
\end{equation}
%Since $A$ has full column rank so that
%\begin{equation}\label{prf-main-71-431}
%\big\langle A(x^{k}-A^{\dagger}b),  A (x^k-A^\dagger b)\big\rangle\geq\sigma^2_{\min}(A)\|x^k-A^\dagger b\|_2^2.
%\end{equation}
Using the inequality $\|a-b\|^2_2\leq 2\|a\|^2_2+2\|b\|^2_2$, we have
\begin{equation}\label{proof-THMfr2}
\begin{array}{ll}
\|T_1T_2^\top A^\top S_1S_2^\top(Ax^k-b)\|^2_2
%&=\|T_1T_2^\top A^\top S_1S_2^\top (Ax^k-AA^\dagger b+AA^\dagger b-b)\|^2_2\\
&\leq 2\|T_1T_2^\top A^\top S_1S_2^\top A (x^k-A^\dagger b)\|^2_2
\\
&\ \ \ +2\|T_1T_2^\top A^\top S_1S_2^\top (AA^\dagger b-b)\|^2_2.
\end{array}
\end{equation}
Note that
\begin{equation}\label{proof-THMfr3}
\begin{array}{ll}
&\mathbb{E}\big[\|T_1T_2^\top A^\top S_1S_2^\top A (x^k-A^\dagger b)\|^2_2\big]\\
&= (x^k-A^\dagger b)^\top A^\top \mathbb{E}\big[S_2S_1^\top AT_2T_1^\top T_1T_2^\top A^\top S_1S_2^\top\big] A (x^k-A^\dagger b)\\
&\leq \beta_1 (x^k-A^\dagger b)^\top A^\top A (x^k-A^\dagger b)
\end{array}
\end{equation}
and similarly,
\begin{equation}\label{proof-THMfr4}
\mathbb{E}\big[\|T_1T_2^\top A^\top S_1S_2^\top (AA^\dagger b-b)\|^2_2]\leq \beta_1\|AA^\dagger b-b\|^2_2.
\end{equation}
We also have
\begin{equation}\label{prf-main-71-433}
\begin{array}{ll}
2\langle A^\top A (x^{k}-A^\dagger b),x^{k-1}-x^{k}\rangle&=2\langle A^\top A (x^{k}-A^\dagger b),(x^{k-1}-A^\dagger b)-(x^{k}-A^\dagger b)\rangle
\\
&\leq \|A(x^{k-1}-A^\dagger b)\|_2^2-\|A(x^{k}-A^\dagger b)\|_2^2
\\
&\leq \sigma_{\max}^2(A)\|x^{k-1}-A^\dagger b\|_2^2-\sigma_{\min}^2(A)\|x^{k}-A^\dagger b\|_2^2.
\end{array}
\end{equation}
Since $0<\alpha<\frac{1}{\beta_1\|A\|^2_F}$ and $A$ has full column rank, we have
\begin{equation}
\label{prf0816}
\frac{2\alpha-2\alpha^2\beta_1\|A\|^2_F}{\|A\|^2_F}(x^k-A^\dagger b)^\top A^\top A (x^k-A^\dagger b)\geq \frac{2\alpha-2\alpha^2\beta_1\|A\|^2_F}{\|A\|^2_F}\sigma^2_{\min}(A)\|x^k-A^\dagger b\|_2^2.
\end{equation}
%where the inequality follows from the fact that $A$ has full column rank.
 Combining \eqref{prf-main-71-432}, \eqref{proof-THMfr2}, \eqref{proof-THMfr3}, \eqref{proof-THMfr4},  \eqref{prf-main-71-433}, and \eqref{prf0816} yields
$$
\begin{array}{ll}
\mathbb{E}\left[\|x^{k+1}-A^{\dagger}b\|^2_2|x^k\right]&\leq\bigg(\underbrace{1+3\omega+2\omega^2
-(2\alpha+\alpha\omega-2\alpha^2\beta_1\|A\|^2_F)\frac{\sigma_{\min}^2(A)}{\|A\|^2_F}}_{\gamma_1}\bigg)
\|x^k-A^\dagger b\|^2_2
\\
&\ \ \ +\bigg(\underbrace{\omega+2\omega^2+\omega\alpha\frac{\sigma^2_{\max}(A)}{\|A\|^2_F}}_{\gamma_2}\bigg)\|x^{k-1}-A^\dagger b\|^2_2
+\underbrace{2\beta_1\alpha^2\|AA^\dagger b-b\|^2_2}_{\zeta}.
\end{array}
$$
Then using Lemma \ref{Lemma-relation} and the same arguments as that in the proof of Theorem \ref{THMr-in629}, we can get the theorem.
\end{proof}

\begin{proof}[Proof of Theorem \ref{THMcfull}]
Note that the linear system is consistent, hence
$$\begin{array}{ll}
&\mathbb{E}\big[\|T_1T_2^\top A^\top S_1S_2^\top(Ax^k-b)\|^2_2\big]=\mathbb{E}\big[\|T_1T_2^\top A^\top S_1S_2^\top A(x^k-A^\dagger b)\|^2_2\big]\\
&= (x^k-A^\dagger b)^\top A^\top\mathbb{E}\big[S_2S_1^\top AT_2T_1^\top T_1T_2^\top A^\top S_1S_2^\top\big]A(x^k-A^\dagger b)\\
&\leq \beta_1 (x^k-A^\dagger b)^\top A^\top A(x^k-A^\dagger b),
\end{array}
$$
which together  with \eqref{prf-main-71-432} and \eqref{prf-main-71-433} implied that
$$
\begin{array}{ll}
\mathbb{E}\left[\|x^{k+1}-A^\dagger b\|^2_2|x^k\right]&\leq\bigg(1+3\omega+2\omega^2
-\alpha\omega\frac{\sigma_{\min}^2(A)}{\|A\|^2_F}\bigg)
\|x^k-A^\dagger b\|^2_2\\
&\ \ \
-\frac{2\alpha-\alpha^2\beta_1\|A\|^2_F}{\|A\|^2_F}\big\langle A(x^{k}-A^{\dagger}b),  A (x^k-A^\dagger b)\big\rangle
\\
&\ \ \ +\bigg(\omega+2\omega^2+\omega\alpha\frac{\sigma^2_{\max}(A)}{\|A\|^2_F}\bigg)\|x^{k-1}-A^\dagger b\|^2_2.
\end{array}
$$
Since $2\alpha-\alpha^2\beta_1\|A\|^2_F\geq 0$ and hence
$$
\frac{2\alpha-\alpha^2\beta_1\|A\|^2_F}{\|A\|^2_F}\big\langle A(x^{k}-A^{\dagger}b),  A (x^k-A^\dagger b)\big\rangle\geq \frac{2\alpha-\alpha^2\beta_1\|A\|^2_F}{\|A\|^2_F}\sigma_{\min}^2(A)\|x^{k}-A^{\dagger}b\|^2_2.
$$
Thus
$$
\begin{array}{ll}
\mathbb{E}\left[\|x^{k+1}-A^\dagger b\|^2_2|x^k\right]&\leq\bigg(\underbrace{1+3\omega+2\omega^2
-\big(2\alpha-\alpha^2\beta_1\|A\|^2_F+\alpha\omega\big)\frac{\sigma_{\min}^2(A)}{\|A\|^2_F}}_{\gamma_1}\bigg)
\|x^k-A^\dagger b\|^2_2
\\
&\ \ \ +\bigg(\underbrace{\omega+2\omega^2+\omega\alpha\frac{\sigma^2_{\max}(A)}{\|A\|^2_F}}_{\gamma_2}\bigg)\|x^{k-1}-A^\dagger b\|^2_2.
\end{array}
$$
Then using Lemma \ref{Lemma-relation} with $\zeta=0$ and  the same arguments as that in the proof of Theorem \ref{THMr-in629}, we can obtain the desired result.
\end{proof}

%%%%%%%%%%%%%%%%%%%%%%%%%%%%%%%%%%%%%%%
%%%%%%%%%%%%%%%%%%%%%%%%%%%%%%%%%%%%%
%%%%%%%%%%%%%%%%%%%%%%%%%%%%%%%%%%

\subsection{Proof of Theorems \ref{THMr}, \ref{THMconsclar}, and \ref{THMr7251}}

\begin{proof}[Proof of Theorem \ref{THMr}]
Since $S_1=S_2=\sqrt{\rho} I$, then we have
\begin{equation}\label{prf-main-71-15}
\begin{array}{ll}
\mathbb{E}[\|AT_1T_2^\top A^\top S_1S_2^\top r^k\|^2_2]&=\rho^2(r^{k}-r^*)^\top A\mathop{\mathbb{E}}[ T_2T_1^\top A^\top AT_1T_2^\top ]A^\top(r^{k}-r^*)
\\
&\leq \rho^2\beta_2\|A^\top(r^{k}-r^*)\|^2_2,
\end{array}
\end{equation}
and hence \eqref{prf-main-71-1} becomes
$$\begin{array}{ll}
\mathbb{E}\left[\|r^{k+1}-r^*\|^2_2|x^k\right]&\leq(1+3\omega+2\omega^2)\|r^k-r^*\|^2_2+(\omega+2\omega^2)\|r^{k-1}-r^*\|^2_2
\\
&\ \ \ -\big(\frac{2\alpha}{\|A\|^2_F}-\alpha^2\rho^2\beta_2\big)\langle r^k-r^*, A A^\top( r^{k}-r^*)\rangle
\\
&\ \ \ +\frac{2\omega\alpha}{\|A\|^2_F}\langle A A^\top (r^{k}-r^*),r^{k-1}-r^{k}\rangle.
\end{array}
$$
Since $\frac{2\alpha}{\|A\|^2_F}-\alpha^2\rho^2\beta_2\geq0$, by \eqref{proof-THM-6294} and \eqref{prf-main-71-2} we have
$$
\begin{array}{ll}
\mathbb{E}\left[\|r^{k+1}-r^*\|^2_2|x^k\right]&\leq\bigg(\underbrace{1+3\omega+2\omega^2
-\left(2\alpha-\alpha^2\rho^2\beta_2\|A\|^2_F+\alpha\omega\right)\frac{\sigma_{\min}^2(A)}{\|A\|^2_F}}_{\gamma_1}\bigg)
\|r^k-r^*\|^2_2
\\
&\ \ \ +\bigg(\underbrace{\omega+2\omega^2+\omega\alpha\frac{\sigma^2_{\max}(A)}{\|A\|^2_F}}_{\gamma_2}\bigg)\|r^{k-1}-r^*\|^2_2.
\end{array}
$$
Then using Lemma \ref{Lemma-relation} with $\zeta=0$ and  the same arguments as that in the proof of Theorem \ref{THMr-in629}, we can obtain the desired result.
\end{proof}

\begin{proof}[Proof of Theorem \ref{THMr7251}]
Since $T_1T_2^\top A^\top S_1^\top S_2r^*=0$, then \eqref{prf-main-71-15} becomes
$$\begin{array}{ll}
\mathbb{E}[\|AT_1T_2^\top A^\top S_1S_2^\top r^k\|^2_2]&=
\mathbb{E}[\|AT_1T_2^\top A^\top S_1S_2^\top (r^k-r^*)\|^2_2]
\\&=(r^{k}-r^*)^\top \mathbb{E}[S_2S_1^\top A T_2T_1^\top A^\top AT_1T_2^\top A^\top S_1S_2^\top](r^{k}-r^*)
\\
&\leq \beta\|r^{k}-r^*\|^2_2.
\end{array}
$$
By using similar arguments as that in the proof of Theorem \ref{THMr},  we know the theorem holds.
\end{proof}

\begin{proof}[Proof of Theorem \ref{THMconsclar}]
Similar to \eqref{prf-main-71-432} and note that $ Ax^0_*=b$, we have
\begin{equation}\label{prf-main-75-1}
\begin{array}{ll}
\mathbb{E}\left[\|x^{k+1}-x^0_*\|^2_2|x^k\right]&\leq(1+3\omega+2\omega^2)\|x^{k}-x^0_*\|^2_2+(\omega+2\omega^2)\|x^{k-1}-x^0_*\|^2_2\\
&\ \ \ -\frac{2\alpha}{\|A\|^2_F}\big\langle x^{k}-x^0_*,A^\top A(x^{k}-x^0_*)\big\rangle
\\
&\ \ \ +\alpha^2\mathbb{E}\big[\|T_1T_2^\top A^\top S_1S_2^\top A(x^k-x^0_*)\|^2_2\big]
\\
&\ \ \ +\frac{2\omega\alpha}{\|A\|^2_F}\big\langle A^\top A(x^{k}-x^0_*),x^{k-1}-x^{k}\big\rangle.
\end{array}
\end{equation}
Since $T_1=T_2=\sqrt{\rho}I$,  we have
\begin{equation}\label{prf-main-75-2}
\mathbb{E}\big[\|T_1T_2^\top A^\top S_1S_2^\top A(x^k-x^{0}_*)\|^2_2\big]\leq \beta_3\rho^2 (x^k-x^{0}_*)^\top A^\top A(x^k-x^{0}_*).
\end{equation}
We also have
$$\begin{array}{ll}
2\langle A^\top A (x^{k}-x^{0}_*),x^{k-1}-x^{k}\rangle&=2\langle A^\top A (x^{k}-x^{0}_*),(x^{k-1}-x^{0}_*)-(x^{k}-x^{0}_*)\rangle
\\
&\leq \|A(x^{k-1}-x^{0}_*)\|_2^2-\|A(x^{k}-x^{0}_*)\|_2^2
\\
&\leq \sigma_{\max}^2(A)\|x^{k-1}-x^{0}_*\|_2^2-\|A(x^{k}-x^{0}_*)\|_2^2.
\end{array}$$
which together  with \eqref{prf-main-75-1} and \eqref{prf-main-75-2} implied that
\begin{equation}\label{prf-main-75-3}
\begin{array}{ll}
\mathbb{E}\left[\|x^{k+1}-x^0_*\|^2_2|x^k\right]&\leq(1+3\omega+2\omega^2)\|x^{k}-x^0_*\|^2_2+(\omega+2\omega^2)\|x^{k-1}-x^0_*\|^2_2\\
&\ \ \ -\frac{2\alpha-\alpha^2\rho^2\beta_3\|A\|^2_F}{\|A\|^2_F}\big\langle x^{k}-x^0_*,A^\top A(x^{k}-x^0_*)\big\rangle
\\
&\ \ \ +\frac{\omega\alpha}{\|A\|^2_F}\sigma_{\max}^2(A)\|x^{k-1}-x^{0}_*\|_2^2-\frac{\omega\alpha}{\|A\|^2_F}\|A(x^{k}-x^{0}_*)\|_2^2.
\end{array}
\end{equation}
Note that the iteration of Algorithm \ref{mPDFR} now becomes
$$
x^{k+1}=x^k-\alpha\rho A^\top S_1S_2^\top(Ax^k-b)+\omega(x^k-x^{k-1}),
$$
and as we choose $x^1=x^0$, so $x^k\in \mbox{Range}(A^\top)+x^0$ for $k=0,1,\ldots$, and therefore, $$x^k-x^0_*=x^k-x^0-A^\dagger( b-Ax^0) \in \mbox{Range}(A^\top).$$
 Hence
$$
\big\langle A(x^{k}-x^0_*),  A (x^k-x^0_*)\big\rangle\geq \sigma_{\min}^2(A)\|x^{k}-x^0_*\|^2_2.
$$
Note that $2\alpha-\alpha^2\rho^2\beta_3\|A\|^2_F\geq 0$, from \eqref{prf-main-75-3}, we have
$$
\begin{array}{ll}
\mathbb{E}\left[\|x^{k+1}-x^0_*\|^2_2|x^k\right]&\leq\bigg(\underbrace{1+3\omega+2\omega^2
-\big(2\alpha-\alpha^2\rho^2\beta_3\|A\|^2_F+\alpha\omega\big)\frac{\sigma_{\min}^2(A)}{\|A\|^2_F}}_{\gamma_1}\bigg)
\|x^k-x^0_*\|^2_2
\\
&\ \ \ +\bigg(\underbrace{\omega+2\omega^2+\omega\alpha\frac{\sigma^2_{\max}(A)}{\|A\|^2_F}}_{\gamma_2}\bigg)\|x^{k-1}-x^0_*\|^2_2.
\end{array}
$$
Then using Lemma \ref{Lemma-relation} with $\zeta=0$ and  the same arguments as that in the proof of Theorem \ref{THMr-in629}, we can obtain the desired result.
\end{proof}

\subsection{Proof of Lemma \ref{lemma-gaussmatrix}}

The following lemma is needed for our proof.

\begin{lemma}
\label{lemma-gauss}
Consider a random vector $\eta\backsim\mathcal{N}(0,I_m)$ and a fixed matrix $A\in\mathbb{R}^{m\times n}$. Then
$$
\mathbb{E}[ \eta \eta^\top AA^\top \eta \eta^\top ]=2 AA^\top +\|A\|^2_F I.
$$
\end{lemma}

\begin{proof}
Assume that $\mbox{rank}(A)=r$ and the nonzero singular values of a matrix $A$ are given by $\sigma_1\geq\sigma_2\geq\ldots\geq\sigma_{r}(A)>0$. Let $A=U\Sigma V^\top$ be the singular value decomposition of $A$.  Set $\zeta:=U^\top \eta$, then we have
$$
\begin{array}{ll}
&\mathbb{E}[\eta \eta^\top AA^\top \eta \eta^\top ]=\mathbb{E}[\eta \eta^\top U\Sigma \Sigma^\top U^\top \eta\eta^\top ]
\\
&=\mathbb{E}[U \zeta \zeta^\top\Sigma \Sigma^\top\zeta \zeta^\top U^\top]=U\mathbb{E}\big[\big(\zeta^\top \Sigma \Sigma^\top\zeta\big)\zeta \zeta^\top\big]U^\top
\\
&=U\mathbb{E}\left[\bigg(\sum\limits_{\ell=1}^r\sigma_\ell^2\zeta_\ell^2\bigg)\left(
    \begin{array}{cccc}
      \zeta_1^2 & \zeta_1\zeta_2 & \ldots & \zeta_1\zeta_m \\
     \zeta_2\zeta_1 & \zeta_2^2 & \ldots & \zeta_2\zeta_m \\
      \vdots & \vdots & \ddots & \vdots  \\
      \zeta_m\zeta_1 & \zeta_m\zeta_2 & \ldots & \zeta_m^2  \\
    \end{array}
  \right)\right] U^\top.
\end{array}
$$
By the rotation invariant of  Gaussian distribution, we know that $\zeta\backsim \mathcal{N}(0,I_m)$. Thus for any $i,j\in[m]$ and $i\neq j$,
$$\mathbb{E}[\zeta_i^4]=3, \mathbb{E}[\zeta_i^2\zeta_j^2]=1,  \mathbb{E}[\zeta_i\zeta_j]=0, \ \mbox{and} \ \mathbb{E}[\zeta_i^3\zeta_j]=0.$$
This implies that
$$
\mathbb{E}\bigg[\bigg(\sum\limits_{\ell=1}^r\sigma_\ell^2\zeta_\ell^2\bigg)\zeta_i\zeta_j\bigg]=
\left\{
  \begin{array}{ll}
    3\sigma^2_i+\bigg(\sum\limits_{\ell\neq i}\sigma_\ell^2\bigg), & \hbox{if $i=j$;} \\
    0, & \hbox{if $i\neq j$.}
  \end{array}
\right.
$$
Hence we have
$$
\begin{array}{ll}
\mathbb{E}[ \eta \eta^\top AA^\top \eta \eta^\top ]&=2U\left(
                                                                                   \begin{array}{ccccc}
                                                                                     \sigma^2_1 & \ldots & 0 & \ldots& 0 \\
                                                                                     \vdots & \ddots & \vdots & \vdots& \vdots \\
                                                                                     0 & \ldots & \sigma^2_r & \ldots& 0 \\
                                                                                     \vdots & \vdots & \vdots & \ddots& \vdots \\
0 & \ldots & 0 & \ldots& 0 \\
                                                                                   \end{array}
                                                                                 \right)U^\top +\|A\|^2_FI\\
&=2U\Sigma\Sigma^\top U^\top+\|A\|^2_F I
\\&=2U\Sigma V^\top V\Sigma^\top U^\top+\|A\|^2_F I\\
&=2AA^\top +\|A\|^2_FI
\end{array}
$$
as desired.
\end{proof}

\begin{proof}[Proof of Lemma \ref{lemma-gaussmatrix}]
Assume that $\mbox{rank}(A)=r$ and the nonzero singular values of a matrix $A$ are given by $\sigma_1\geq\sigma_2\geq\ldots\geq\sigma_{r}(A)>0$. Let $A=U\Sigma V^\top$ be the singular value decomposition of $A$.  Set $Q:=U^\top S$, then we have
$$
\begin{array}{ll}
\mathbb{E}[S S^\top AA^\top S S^\top ]&=\mathbb{E}[ S S^\top U\Sigma \Sigma^\top U^\top SS^\top ]
\\
&=\mathbb{E}[U Q Q^\top\Sigma \Sigma^\top Q Q^\top U^\top]=U\mathbb{E}\big[ Q Q^\top\Sigma \Sigma^\top Q Q^\top \big]U^\top
\\
&=U\mathbb{E}\bigg[\big(\sum\limits_{i=1}^p\Sigma^\top Q_i Q_i^\top \big)^2\bigg]U^\top
\\
&=U\mathbb{E}\bigg[\sum\limits_{i=1}^p Q_i Q_i^\top\Sigma\Sigma^\top Q_i Q_i^\top+\sum\limits_{i\neq j}\big( Q_i Q_i^\top \Sigma\Sigma^\top Q_j Q_j^\top \big)\bigg]U^\top
\\
&=U\mathbb{E}\bigg[\sum\limits_{i=1}^p Q_i Q_i^\top\Sigma\Sigma^\top Q_i Q_i^\top\bigg]U^\top+U\mathbb{E}\bigg[\sum\limits_{i\neq j}\big( Q_i Q_i^\top \Sigma\Sigma^\top Q_j Q_j^\top \big)\bigg]U^\top,
\end{array}
$$
where $Q_i$ denotes the $i$-th column of the matrix $Q$. Let denote the $(\ell,i)$-th entry of $Q$ as $q_{\ell i}$, then
$$
\begin{array}{ll}
&\mathbb{E}\bigg[\sum\limits_{i\neq j}\big( Q_i Q_i^\top \Sigma\Sigma^\top Q_j Q_j^\top\big)\bigg]=\sum\limits_{i\neq j}\mathbb{E}\bigg[\big(Q_i^\top \Sigma\Sigma^\top Q_j\big) Q_i  Q_j^\top \bigg]
\\&=\sum\limits_{i\neq j}\mathbb{E}\left[\bigg(\sum\limits_{\ell=1}^r\sigma_\ell^2q_{\ell i}q_{\ell j}\bigg)\left(
    \begin{array}{cccc}
      q_{1i}q_{1j} & q_{1i}q_{2j} & \ldots & q_{1i}q_{mj}  \\
      q_{2i}q_{1j} & q_{2i}q_{2j} & \ldots & q_{2i}q_{mj}  \\
      \vdots & \vdots & \ddots & \vdots  \\
     q_{mi}q_{1j} & q_{mi}q_{2j} & \ldots & q_{mi}q_{mj}  \\
    \end{array}
  \right)\right].
\end{array}
$$
By the rotation invariant of  Gaussian distribution, we know that $q_{\ell i}\backsim \mathcal{N}(0,1)$ for any $i\in[p]$ and $\ell\in[m]$. Thus for any $i,j\in[p]$ with $i\neq j$ and any $\ell,k,\lambda\in[m]$ with $k\neq \lambda$, we have
$$\mathbb{E}[q_{\ell i}q_{\ell j}q_{\lambda i}q_{\lambda j}]=\left\{
                                                               \begin{array}{ll}
                                                                 0, & \hbox{$\ell\neq\lambda$;} \\
                                                                 1, & \hbox{$\ell=\lambda$.}
                                                               \end{array}
                                                             \right.
\ \ \mbox{ and} \ \ \mathbb{E}[q_{\ell i}q_{\ell j}q_{k i}q_{\lambda j}]=0
.$$
This implies that
$$
\mathbb{E}\bigg[\bigg(\sum\limits_{\ell=1}^r\sigma_\ell^2q_{\ell i}q_{\ell j}\bigg) q_{k i}q_{\lambda j}\bigg]=
\left\{
  \begin{array}{ll}
   \sigma_{k}^2, & \hbox{if $k=\lambda$;} \\
    0, & \hbox{if $k\neq \lambda$.}
  \end{array}
\right.
$$
Hence we have
$$
\begin{array}{ll}
\mathbb{E}[S S^\top AA^\top S S^\top ]&=2pA^\top AA^\top A+p\|A\|^2_F A^\top A +\sum\limits_{i\neq j}U\left(
                                                                                   \begin{array}{ccccc}
                                                                                     \sigma^2_1 & \ldots & 0 & \ldots& 0 \\
                                                                                     \vdots & \ddots & \vdots & \vdots& \vdots \\
                                                                                     0 & \ldots & \sigma^2_r & \ldots& 0 \\
                                                                                     \vdots & \vdots & \vdots & \ddots& \vdots \\
0 & \ldots & 0 & \ldots& 0 \\
                                                                                   \end{array}
                                                                                 \right)U^\top
\\
&=2p AA^\top +p\|A\|^2_F I+p(p-1)U\Sigma\Sigma^\top U^\top
\\&=2p AA^\top +p\|A\|^2_F I+p(p-1) AA^\top \\
&=(p^2+p) AA^\top +p\|A\|^2_FI
\end{array}
$$
as desired.
\end{proof}

\end{document}